\documentclass[11pt, twoside, leqno]{article}

\usepackage{amsmath}
\usepackage{amssymb}
\usepackage{titletoc}
\usepackage{mathrsfs}
\usepackage{amsthm}
\usepackage{indentfirst}
\usepackage{color}
\usepackage{txfonts}
\usepackage{enumerate}

\usepackage{anysize}

\allowdisplaybreaks

\newtheorem{theorem}{Theorem}[section]
\newtheorem{lemma}[theorem]{Lemma}
\newtheorem{corollary}[theorem]{Corollary}

\theoremstyle{definition}
\newtheorem{remark}[theorem]{Remark}
\newtheorem{definition}[theorem]{Definition}

\numberwithin{equation}{section}

\begin{document}

\title{\bf\Large Matrix-Weighted Besov-Type and Triebel--Lizorkin-Type Spaces II:
Sharp Boundedness of Almost Diagonal Operators
\footnotetext{\hspace{-0.35cm} 2020 {\it Mathematics Subject Classification}.
Primary 46E35; Secondary 47A56, 42B25, 42B35.\endgraf
{\it Key words and phrases.}
matrix weight,
Besov-type space,
Triebel--Lizorkin-type space,
$A_p$-dimension,
almost diagonal operator.\endgraf
This project is supported by
the National Key Research and Development Program of China
(Grant No.\ 2020YFA0712900),
the National Natural Science Foundation of China
(Grant Nos.\ 12371093, 12071197, 12122102, 12326307, and 12326308),
the Fundamental Research Funds
for the Central Universities (Grant No. 2233300008),
and the Research Council of Finland (Grant Nos.\ 314829, 346314, and 364208).}}
\date{}
\author{Fan Bu, Tuomas Hyt\"onen,
Dachun Yang\footnote{Corresponding author, E-mail:
\texttt{dcyang@bnu.edu.cn}/{\color{blue} August 19, 2024/Relation to new results of \cite{ks24} addressed.}}\ \ and Wen Yuan}

\maketitle

\vspace{-0.8cm}

\begin{center}
\begin{minipage}{13cm}
{\small {\bf Abstract}\quad
This article is the second one of three successive articles
of the authors on the matrix-weighted Besov-type and Triebel--Lizorkin-type spaces.
In this article, we obtain the sharp boundedness of almost diagonal operators
on matrix-weighted Besov-type and Triebel--Lizorkin-type sequence spaces.
These results not only possess broad generality but also improve {\color{red}several existing related results in various special cases covered by this family of spaces.
This improvement depends, on the one hand, on the notion of $A_p$-dimensions of matrix weights and their properties introduced in the first article of this series and, on the other hand, on a careful direct analysis of sequences of averages avoiding maximal operators.}
{\color{blue}While a recent matrix-weighted extension of the Fefferman--Stein vector-valued maximal inequality would provide an alternative route to some of our results in the restricted range of function space parameters $p,q\in(1,\infty)$, our approach covers the full scale of exponents $p\in(0,\infty)$ and $q\in(0,\infty]$ that is relevant in the theory of function spaces.}
}
\end{minipage}
\end{center}


\tableofcontents

\vspace{0.1cm}

\section{Introduction}

This article is the second one of our three successive articles in which
we develop a complete real-variable theory of matrix-weighted Besov-type
and Triebel--Lizorkin-type spaces of $\mathbb C^m$-valued distributions
on $\mathbb R^n$. {\color{red}
The article at hand deals with the sequential versions of these spaces, and can be studied largely independently of the other two; but the interplay of all three articles of becomes relevant in order to connect the present results to the actual function spaces of Besov-type and Triebel--Lizorkin-type. The attribute ``type'' refers to the fact that these are generalisations of the usual three-parameter Besov and Triebel--Lizorkin spaces, incorporating a fourth index $\tau$ that measures a Morrey space type nature of these spaces; the usual three-parameter versions are recovered by setting $\tau=0$.}
We consistently denote by $m$ the dimension of the target space
of our distributions and hence our matrix-weights take values
in the space of $m\times m$ complex matrices.

In the first article of this series \cite{bhyy1}, we introduced the matrix-weighted Besov-type space
$\dot B^{s,\tau}_{p,q}(W)$ and the matrix-weighted Triebel--Lizorkin-type space
$\dot F^{s,\tau}_{p,q}(W)$, for any {\color{red}smoothness $s\in\mathbb{R}$, Morrey index $\tau\in[0,\infty)$,
integrability $p\in(0,\infty)$, and microscopic parameter $q\in(0,\infty]$}, and established their $\varphi$-transform
characterisation {\color{red}in terms of related sequence spaces $\dot b^{s,\tau}_{p,q}(W)$ and $\dot f^{s,\tau}_{p,q}(W)$, whose elements are sequences $\vec t=\{\vec t_Q\}_{Q\in\mathscr Q}\subset \mathbb C^m$ of vectors $\vec t_Q\in\mathbb C^m$ indexed by the family of dyadic cubes
\begin{equation}\label{dyadic cubes}
  \mathscr Q:=\Big\{Q:=x_Q+\ell(Q)[0,1)^n\equiv x_Q+2^{-j_Q}[0,1)^n:j_Q=-\log_2\ell(Q)\in \mathbb Z, 2^{j_Q}x_Q\in\mathbb Z^n\Big\}
\end{equation}
of $\mathbb R^n$};
see \cite{bhyy1} for details and history of these spaces and more related references.

{\color{red}Thanks to the said characterisation, in this article, we concentrate on the sequence spaces $\dot a^{s,\tau}_{p,q}(W)\in\{\dot b^{s,\tau}_{p,q}(W),\dot f^{s,\tau}_{p,q}(W)\}$ only, and we consider a general class of operators on these spaces represented by infinite matrices $B:=\{b_{Q,P}\}_{Q,P\in\mathscr Q}\subset\mathbb C$ indexed by dyadic cubes. The coefficients $b_{Q,P}$ are assumed to satisfy the following form of decay in terms of the separation of the cubes $Q,P$ in both position and scale:
\begin{equation*}
   |b_{Q,R}|\leq Cb_{Q,R}^{DEF}
   :=C\left[1+\frac{|x_Q-x_R|}{\ell(Q)\vee\ell(R)}\right]^{-D}
\left\{\begin{aligned}
&\left[\frac{\ell(Q)}{\ell(R)}\right]^E&&\text{if }\ell(Q)\leq\ell(R),\\
&\left[\frac{\ell(R)}{\ell(Q)}\right]^F&&\text{if }\ell(R)<\ell(Q),
\end{aligned}\right.
\end{equation*}
with a positive constant $C$ and parameters $D,E,F\in\mathbb R$. Following Frazier and Jawerth \cite{fj90} and several subsequent works, such a matrix $B$ will be called {\em $(D,E,F)$-almost diagonal}. It is immediate that this condition becomes stronger when any of the parameters $D,E,F$ is increased; thus, it is of interest to obtain results for such matrices under minimal assumptions on $D,E,F$.

It is well known since the work of Frazier and Jawerth \cite{fj90} that several classical operators acting on spaces of distributions, when conjugated with the $\varphi$-transform, will be represented by almost diagonal operators on the related sequence spaces.} This method has been further developed and used by Bownik \cite{b05,b07},
Bownik and Ho \cite{bh06}, Yang and Yuan \cite{yy10}, Cleanthous
et al. \cite{cgn19}, Georgiadis et al. \cite{gkkp19,gkp21},
and Frazier and Roudenko \cite{fr04,fr21}.
{\color{red}In the final article of this series, \cite{bhyy3}, our present results on almost diagonal operators on the sequence spaces $\dot a^{s,\tau}_{p,q}(W)\in\{\dot b^{s,\tau}_{p,q}(W),\dot f^{s,\tau}_{p,q}(W)\}$ will be used
to obtain optimal characterizations by molecules and wavelets, trace theorems,
and the optimal boundedness of pseudo-differential operators and Calder\'on--Zygmund operators on the
function spaces $\dot A^{s,\tau}_{p,q}(W)\in\{\dot B^{s,\tau}_{p,q}(W),\dot F^{s,\tau}_{p,q}(W)\}$,
extending and improving the works just mentioned. This will depend on related improvements over existing results on the level of sequence spaces and almost diagonal operators in the work at hand.}

{\color{red}
To give an idea of our results, and how they compare with existing literature, let us consider the important special case of ``plain'' Besov and Triebel--Lizorkin sequence spaces $\dot a^s_{p,q}(W)=\dot a^{s,0}_{p,q}(W)$ with Morrey index $\tau=0$; here and below, we denote by $\dot a\in\{\dot b,\dot f\}$ a generic sequence space of either Besov or Triebel--Lizorkin type. To streamline the notation, it is also helpful to introduce the (usual) notation
\begin{equation*}
J:=\left\{\begin{aligned}
&\frac{n}{\min\{1,p\}}&&\text{if we are dealing with a Besov-type space},\\
&\frac{n}{\min\{1,p,q\}}&&\text{if we are dealing with a Triebel--Lizorkin-type space}.
\end{aligned}\right.
\end{equation*}

The classical work of Frazier and Jawerth \cite{fj90} shows that $(D,E,F)$-almost diagonal operators are bounded on the {\em unweighted} $\dot a^s_{p,q}$ spaces, provided that
\begin{equation}\label{fj90 conditions}
  D>J,\quad E>\frac{n}{2}+s,\quad\text{and}\quad F>J-\frac{n}{2}-s.
\end{equation}
(To be precise, \cite[Theorem 3.3]{fj90} proves this in the Triebel--Lizorkin case, but the Besov case is similar and easier.) Prior to our work, previous extensions of this result to matrix-weighted spaces have been stated in terms of the so-called {\em doubling exponent} $\beta_W\geq n$ of the weight $W$; since we will not use this notion except when comparing our results to previous works, we refer the reader to \cite[Definition 1.5]{ro03} for the definition. The previous sufficient $(D,E,F)$-almost diagonality conditions for boundedness on $\dot b^s_{p,q}(W)$, due to Roudenko \cite[Theorem 1.10]{ro03} for $p\in[1,\infty)$, and Frazier and Roudenko \cite[Theorem 4.1]{fr04} for $p\in(0,1]$, then read as
\begin{equation}\label{besov old}
  D>J+\frac{\beta_W-n}{p},\quad E>\frac{n}{2}+s,\quad\text{and}\quad F>J-\frac{n}{2}-s+\frac{\beta_W-n}{p},
\end{equation}
where the extra nonnegative term $(\beta_W-n)/p$ has been added to two of the conditions in \eqref{fj90 conditions}. On the other hand, to obtain boundedness on $\dot f^s_{p,q}(W)$, the recent work of Frazier and Roudenko \cite[Theorem 2.6]{fr21} required the still stronger assumption that
\begin{equation}\label{TL old}
  D>J+\frac{\beta_W}{p},\quad E>\frac{n}{2}+s+\frac{n}{p},\quad\text{and}\quad F>J-\frac{n}{2}-s+\frac{\beta_W-n}{p},
\end{equation}
where a further strictly positive term $n/p$ has been added to two of the conditions in \eqref{besov old}.}

Frazier and Roudenko \cite[p.\,527]{fr21} \emph{conjectured}
that possibly a weaker definition of the almost diagonality
can be found by using properties of $A_p$-matrix weights,
but they did not find a satisfactory answer.
We completely confirm this conjecture in this article via
introducing some methods that are new even for scalar weights. {\color{red}
As a special case of our new results, we show that the classical Frazier--Jawerth conditions \eqref{fj90 conditions} for unweighted $\dot a^s_{p,q}$ remain sufficient for boundedness of $(D,E,F)$-almost diagonal operators in the full generality of matrix-weighted $\dot a^s_{p,q}(W)$ for all $W\in A_p$. Thus, we not only improve the earlier conditions \eqref{besov old} and \eqref{TL old} but also simplify the picture by eliminating a dependence of these conditions on the parameters of the particular matrix weight. In the converse direction, we also provide examples to show that these conditions \eqref{fj90 conditions} are sharp for any of the Besov sequence space $\dot b^s_{p,q}$, as well as for ``most'' of the Triebel--Lizorkin sequence spaces $\dot f^s_{p,q}$, subject only to the restriction that $q\geq\min(1,p)$. This sharpness is new even in the unweighted spaces, where it complements the classical theory of Frazier and Jawerth \cite{fj90}. The examples showing the sharpness are somewhat elaborate, and actually constitute a significant part of this work; they will be presented in separate sections below.

We refer the reader to the body of the paper for the complete versions of our results in the full range of the $\dot a^{s,\tau}_{p,q}(W)$ spaces with the Morrey index $\tau$, and their relation to other existing results, like ones in \cite{t13,yy10}, that we will also extend or improve.}

{\color{red}As for our methods,} recall that, for both Besov and Triebel--Lizorkin spaces and weighted counterparts,
the boundedness of almost diagonal operators can be obtained
with the help of the (weighted) Fefferman--Stein vector-valued maximal inequality.
Whether the matrix-weighted Fefferman--Stein vector-valued inequality
still holds {\color{blue} was an open problem (see, for instance, \cite[Question 7.3]{dkps})
at the time of our investigation, so we had} to seek some new methods to escape this.
{\color{blue}After our completing this study, the mentioned problem has been recently solved in \cite{ks24}, and this could be used to give an alternative route to some of our results for exponents $p,q\in(1,\infty)$. However, even after the availability of the new results of \cite{ks24}, our approach has the benefit of covering the full scale of $p\in(0,\infty)$ and $q\in(0,\infty]$ that is relevant in the theory of function spaces.

Note that Frazier and Roudenko \cite{fr04,fr21}
established the boundedness of almost diagonal operators on
the matrix-weighted Triebel--Lizorkin sequence spaces $\dot f^s_{p,q}(W)$ in yet another way,
by a reduction to corresponding results on the unweighted sequence spaces $\dot f^s_{p,q}$.}
Their proof is straightforward but requires some
strong extra assumptions for almost diagonal operators.
This limitation makes the boundedness of non-convolution Calder\'on--Zygmund operators
on matrix-weighted Triebel--Lizorkin spaces $\dot F^s_{p,q}(W)$
be less natural \cite[p.\,527]{fr21},
so Frazier and Roudenko \cite{fr21} only studied
convolution Calder\'on--Zygmund operators.
Thus, in order to obtain the optimal boundedness
of non-convolution Calder\'on--Zygmund operators in {\color{red}our follow-up work} \cite{bhyy3},
we need to essentially improve their results on the boundedness
of almost diagonal operators.

{\color{red}
Our substitute for the lack of a Fefferman--Stein vector-valued maximal inequality {\color{blue} in the full range of parameters that we want to cover} is not a single tool to be stated, proved, and then applied, but rather a philosophy to be followed throughout the argument: Whenever we encounter an average over a cube in our analysis (and the reader familiar with the estimation of operators of harmonic analysis will know that this happens quite often), we must resist the temptation to just dominate it by the maximal operator, but instead keep whatever average is given to us, and then do the best that we can to estimate it in the subsequent steps. In particular, it will often happen that estimating an $\ell^q$ sum over a sequence of functions (as typical in the context of Triebel--Lizorkin norms) will produce a sequence of different averages applied to each of the functions; while dominating all them by the single maximal operator would certainly produce a simpler-looking estimate, it also seems to give away too much precision, and it is preferable to continue working with the seemingly more complicated but eventually more feasible to estimate expression. For an example of such an estimate with a sequence of averages on the right-hand side, and how to estimate it further, the reader can take a look at \eqref{ABt} and the proof of Lemma \ref{ad F1}, respectively.
}

The organization of the rest of this article is as follows:

{\color{red}
The first few sections present the relevant preliminaries. In Section \ref{preliminaries}, we recall the definition of $A_p$-matrix weights and the related concept of matrix $A_p$-dimension, which plays a key role in the formulation of our almost diagonality assumptions in the general case. In Section \ref{BTL}, we define the matrix-weighted Besov-type and Triebel--Lizorkin-type sequence spaces, and their classification, from \cite{bhyy1}, into supercritical, critical, and subcritical spaces, which will also be reflected in the way that these different cases are treated in the proof of the main results. In Section \ref{almost}, we {\color{red}discuss the almost diagonal operators, the main theme of this paper, and} recall the existing results, Theorem \ref{ad FJ-YY}, about their boundedness on several sequence spaces related to $\dot a^{s,\tau}_{p,q}(W)$.

In Section \ref{better}, we formulate the main new result, Theorem \ref{ad BF},
about the boundedness of almost diagonal operators on $\dot a^{s,\tau}_{p,q}(W)$,
which improves all cases of Theorem \ref{ad FJ-YY} in the subcritical range.
We also prove the sharpness of this theorem in several cases.
Since the proof of this theorem is lengthy, it is presented in {\color{red}several subsequent sections, starting with the somewhat simpler cases of ``plain'' Besov and Triebel--Lizorkin spaces with $\tau=0$, which have an independent interest. Moreover, the examples demonstrating the sharpness are presented in their own sections.}

In Section \ref{sec ad Besov}, we provide conditions for the boundedness of
almost diagonal operators on matrix-weighted Besov sequence spaces
$\dot b_{p,q}^s(W)$ and demonstrate their sharpness {\color{red}in Section \ref{sec sharp Besov}.}
In particular, we find that the almost diagonal conditions of $\dot b_{p,q}^s$
given by Frazier et al. \cite[Theorem 6.20]{fjw91} are sharp, which is new.
Recall that Frazier and Roudenko \cite{fr04,ro03} have already established
the boundedness of almost diagonal operators on $\dot b^{s}_{p,q}(W)$,
but our almost diagonal conditions improve theirs.

In Section \ref{sec ad TL},
using shifted systems of dyadic cubes of $\mathbb R^n$ and
substitutes of the Fefferman--Stein vector-valued inequality
in the matrix-weighted form, we obtain some analogous results for
matrix-weighted Triebel--Lizorkin sequence spaces $\dot f_{p,q}^s(W)$.
{\color{red}The sharpness of these results is proved in Section \ref{sec sharp TL},
but we only achieve this for $q\in[1\wedge p,\infty]$.}
In this range, we also obtain the sharpness of the almost diagonal conditions
of $\dot f_{p,q}^s$ given by Frazier et al. \cite[Theorem 6.20]{fjw91}.
Moreover, we improve the boundedness of almost diagonal operators
on $\dot f^{s}_{p,q}(W)$ obtained by Frazier and Roudenko \cite{fr21}.

In Section \ref{sec ad general},
using the results obtained in Sections \ref{sec ad Besov} and \ref{sec ad TL}
and $A_p$-dimensions for matrix weights and their elaborate properties
in terms of reducing operators obtained in \cite{bhyy1},
we then complete the proof of Theorem \ref{ad BF}.
{\color{red}The sharpness in yet another range of parameters is obtained in Section \ref{sec sharp 1p}.}

In Section \ref{special}, we explore some improvements that
we can make in the critical and the supercritical regimes,
so as to also improve all related cases of Theorem \ref{ad FJ-YY}.
In these cases, by \cite[Corollary 4.18]{bhyy1}, we find that
$\dot a^{s,\tau}_{p,q}(W)$ reduces to $\dot f_{\infty,q}^s(\mathbb{A})$.
Therefore, we obtain the boundedness of almost diagonal operators
on $\dot a^{s,\tau}_{p,q}(W)$ by first establishing the boundedness
on $\dot f_{\infty,q}^s(\mathbb{A})$, with sharp bounds when $q\in[1,\infty]$.
In particular, in the same range,
we also find that the almost diagonal conditions of $\dot f_{\infty,q}^s$
given by Frazier and Jawerth \cite[p.\,81]{fj90} are sharp, which is new.
In Section \ref{summary}, we make a summary and obtain the final version
on the boundedness of almost diagonal operators on $\dot a^{s,\tau}_{p,q}(W)$,
which improve Theorem \ref{ad FJ-YY} in all ranges of indices.}

Finally, we make some conventions on notation.
A \emph{cube} $Q$ of $\mathbb{R}^n$ always has finite edge length
and edges of cubes are always assumed to be parallel to coordinate axes,
but $Q$ is not necessary to be open or closed.
For any cube $Q$ of $\mathbb{R}^n$,
let $c_Q$ be its center and $\ell(Q)$ its edge length.
For any $\lambda\in(0,\infty)$ and any cube $Q$ of $\mathbb{R}^n$,
let $\lambda Q$ be the cube with the same center as $Q$ and the edge length $\lambda\ell(Q)$.

{\color{red}A special role is played by the collection $\mathscr Q$ of {\em dyadic cubes}, and the main notation related to them was already fixed in \eqref{dyadic cubes} above. In particular, $j_Q:=-\log_2\ell(Q)\in\mathbb Z$ denotes the generation of $Q$, and $x_Q\in\mathbb R^n$ its ``lower left'' corner.}

The \emph{ball} $B$ of $\mathbb{R}^n$,
centered at $x\in\mathbb{R}^n$ with radius $r\in(0,\infty)$,
is defined by setting
$$
B:=\{y\in\mathbb{R}^n:\ |x-y|<r\}=:B(x,r);
$$
moreover, for any $\lambda\in(0,\infty)$, $\lambda B:=B(x,\lambda r)$.
For any $r\in\mathbb{R}$, $r_+$ is defined as $r_+:=\max\{0,r\}$
and $r_-$ as $r_-:=\max\{0,-r\}$.
For any $a,b\in\mathbb{R}$, $a\wedge b:=\min\{a,b\}$ and $a\vee b:=\max\{a,b\}$.
The symbol $C$ denotes a positive constant which is independent
of the main parameters involved, but may vary from line to line.
The symbol $A\lesssim B$ means that $A\leq CB$ for some positive constant $C$,
while $A\sim B$ means $A\lesssim B\lesssim A$.
Let $\mathbb N:=\{1,2,\ldots\}$ and $\mathbb Z_+:=\mathbb N\cup\{0\}$.
We use $\mathbf{0}$ to denote the \emph{origin} of $\mathbb{R}^n$.
For any set $E\subset\mathbb{R}^n$,
we use $\mathbf 1_E$ to denote its \emph{characteristic function}.
The \emph{Lebesgue space} $L^p(\mathbb{R}^n)$
is defined to be the set of all measurable functions
$f$ on $\mathbb{R}^n$ such that $\|f\|_{L^p(\mathbb{R}^n)}<\infty$, where
$$
\|f\|_{L^p(\mathbb{R}^n)}:=\left\{\begin{aligned}
&\left[\int_{\mathbb{R}^n}|f(x)|^p\,dx\right]^{\frac{1}{p}}
&&\text{if }p\in(0,\infty),\\
&\mathop{\mathrm{\,ess\,sup\,}}_{x\in\mathbb{R}^n}|f(x)|
&&\text{if }p=\infty.
\end{aligned}\right.
$$
The \emph{locally integrable Lebesgue space}
$L^p_{\mathrm{loc}}(\mathbb{R}^n)$ is defined to be the set of
all measurable functions $f$ on $\mathbb{R}^n$ such that,
for any bounded measurable set $E$,
$\|f\|_{L^p(E)}:=\|f\mathbf{1}_E\|_{L^p(\mathbb{R}^n)}<\infty.$
In what follows, we denote $L^p(\mathbb{R}^n)$ and $L^p_{\mathrm{loc}}(\mathbb{R}^n)$
simply, respectively, by $L^p$ and $L^p_{\mathrm{loc}}$.
For any measurable function $w$ on $\mathbb{R}^n$
and any measurable set $E\subset\mathbb{R}^n$, let
$w(E):=\int_Ew(x)\,dx.$
For any measurable function $f$ on $\mathbb{R}^n$
and any measurable set $E\subset\mathbb{R}^n$ with $|E|\in(0,\infty)$, let
$$
\fint_Ef(x)\,dx:=\frac{1}{|E|}\int_Ef(x)\,dx.
$$
Also, when we prove a theorem
(and the like), in its proof we always use the same
symbols as those appearing in
the statement itself of the theorem (and the like).

\section{Preliminaries on Matrix Weights}\label{preliminaries}

In this section, we recall the definitions of $A_p$-matrix weights and
matrix-weighted Besov-type and Triebel--Lizorkin-type sequence spaces.
Let us begin with some basic concepts of matrices.

For any $m,n\in\mathbb{N}$,
the set of all $m\times n$ complex-valued matrices is denoted by $M_{m,n}(\mathbb{C})$
and $M_{m,m}(\mathbb{C})$ is simply denoted by $M_{m}(\mathbb{C})$.
For any $A\in M_{m,n}(\mathbb{C})$,
the \emph{conjugate transpose} of $A$ is denoted by $A^*$.

For any $A\in M_m(\mathbb{C})$, let
$$
\|A\|:=\sup_{\vec z\in\mathbb{C}^m,\,|\vec z|=1}|A\vec z|.
$$

In what follows, we regard $\mathbb{C}^m$ as $M_{m,1}(\mathbb{C})$
and let $\vec{\mathbf{0}}:=(0,\ldots,0)^\mathrm{T}\in\mathbb{C}^m$.
Moreover, for any $\vec z:=(z_1,\ldots,z_m)^\mathrm{T}\in\mathbb{C}^m$,
let $|\vec z|:=(\sum_{i=1}^m|z_i|^2)^{\frac12}$.

A matrix $A\in M_m(\mathbb{C})$ is said to be \emph{positive definite}
if, for any $\vec z\in\mathbb{C}^m\setminus\{\vec{\mathbf{0}}\}$, $\vec z^*A\vec z>0$,
and $A$ is said to be \emph{nonnegative definite} if,
for any $\vec z\in\mathbb{C}^m$, $\vec z^*A\vec z\geq0$.
The matrix $A$ is said to be \emph{invertible} if
there exists a matrix $A^{-1}\in M_m(\mathbb{C})$ such that $A^{-1}A=I_m$,
where $I_m$ is identity matrix.

Now, we recall the concept of matrix weights (see, for instance, \cite{nt96,tv97,v97}).

\begin{definition}
A matrix-valued function $W:\ \mathbb{R}^n\to M_m(\mathbb{C})$ is called
a \emph{matrix weight} if $W$ satisfies that
\begin{enumerate}[\rm(i)]
\item for any $x\in\mathbb{R}^n$, $W(x)$ is nonnegative definite;
\item for almost every $x\in\mathbb{R}^n$, $W(x)$ is invertible;
\item the entries of $W$ are all locally integrable.
\end{enumerate}
\end{definition}

Now, we recall the concept of $A_p$-matrix weights
(see, for instance, \cite[p.\,490]{fr21}).

\begin{definition}\label{def ap}
Let $p\in(0,\infty)$. A matrix weight $W$ on $\mathbb{R}^n$
is called an $A_p(\mathbb{R}^n,\mathbb{C}^m)$-\emph{matrix weight}
if $W$ satisfies that, when $p\in(0,1]$,
$$
[W]_{A_p(\mathbb{R}^n,\mathbb{C}^m)}
:=\sup_{\mathrm{cube}\,Q}\mathop{\mathrm{\,ess\,sup\,}}_{y\in Q}
\fint_Q\left\|W^{\frac{1}{p}}(x)W^{-\frac{1}{p}}(y)\right\|^p\,dx
<\infty
$$
or that, when $p\in(1,\infty)$,
$$
[W]_{A_p(\mathbb{R}^n,\mathbb{C}^m)}
:=\sup_{\mathrm{cube}\,Q}
\fint_Q\left[\fint_Q\left\|W^{\frac{1}{p}}(x)W^{-\frac{1}{p}}(y)\right\|^{p'}
\,dy\right]^{\frac{p}{p'}}\,dx
<\infty,
$$
where $\frac{1}{p}+\frac{1}{p'}=1$.
\end{definition}

In what follows, if there exists no confusion,
we denote $A_p(\mathbb{R}^n,\mathbb{C}^m)$ simply by $A_p$.
Next, we recall the concept of reducing operators (see, for instance, \cite[(3.1)]{v97}).

\begin{definition}\label{reduce}
Let $p\in(0,\infty)$, $W$ be a matrix weight,
and $E\subset\mathbb{R}^n$ a bounded measurable set satisfying $|E|\in(0,\infty)$.
The matrix $A_E\in M_m(\mathbb{C})$ is called a \emph{reducing operator} of order $p$ for $W$
if $A_E$ is positive definite and,
for any $\vec z\in\mathbb{C}^m$,
$$
\left|A_E\vec z\right|
\sim\left[\fint_E\left|W^{\frac{1}{p}}(x)\vec z\right|^p\,dx\right]^{\frac{1}{p}},
$$
where the positive equivalence constants depend only on $m$ and $p$.
\end{definition}

\begin{remark}
In Definition \ref{reduce}, the existence of $A_E$ is guaranteed by
\cite[Proposition 1.2]{g03} and \cite[p.\,1237]{fr04}; we omit the details.
\end{remark}

To obtain the sharp estimate of $\|A_QA_R^{-1}\|$,
we need the concept of $A_p$-dimensions of matrix weights.

\begin{definition}\label{Ap dim}
Let $p\in(0,\infty)$, $d\in\mathbb{R}$, and $W$ be a matrix weight.
Then $W$ is said to have the \emph{$A_p$-dimension $d$}
if there exists a positive constant $C$ such that,
for any cube $Q\subset\mathbb{R}^n$ and any $i\in\mathbb{Z}_+$,
when $p\in(0,1]$,
\begin{equation*}
\mathop{\mathrm{\,ess\,sup\,}}_{y\in2^iQ}\fint_Q \left\|W^{\frac{1}{p}}(x)W^{-\frac{1}{p}}(y)\right\|^p\,dx
\leq C2^{id}
\end{equation*}
or, when $p\in(1,\infty)$,
\begin{equation*}
\fint_Q\left[\fint_{2^iQ}\left\|W^{\frac{1}{p}}(x)W^{-\frac{1}{p}}(y)
\right\|^{p'}\,dy\right]^{\frac{p}{p'}}\,dx
\leq C2^{id},
\end{equation*}
where $\frac{1}{p}+\frac{1}{p'}=1$.
\end{definition}

The following sharp estimate is just \cite[Lemma 2.29]{bhyy1},
which plays a key role in establishing the sharp boundedness of almost diagonal operators.

\begin{lemma}\label{22 precise}
Let $p\in(0,\infty)$, let $W\in A_p$ have the $A_p$-dimension $d\in[0,n)$,
and let $\{A_Q\}_{\mathrm{cube}\,Q}$ be a family of
reducing operators of order $p$ for $W$. If $p\in(1,\infty)$,
let further $\widetilde W:=W^{-\frac 1{p-1}}$ (which belongs to $A_{p'}$)
have $A_{p'}$-dimension $\widetilde d$,
while, if $p\in(0,1]$, let $\widetilde d:=0$.
Let
\begin{equation}\label{Delta}
\Delta:=\frac{d}{p}+\frac{\widetilde d}{p'}.
\end{equation}
Then there exists a positive constant $C$ such that,
for any cubes $Q$ and $R$ of $\mathbb{R}^n$,
\begin{equation}\label{22 p<1}
\left\|A_QA_R^{-1}\right\|
\leq C\max\left\{\left[\frac{\ell(R)}{\ell(Q)}\right]^{\frac{d}{p}},
\left[\frac{\ell(Q)}{\ell(R)}\right]^{\frac{\widetilde d}{p'}}\right\}
\left[1+\frac{|c_Q-c_R|}{\ell(Q)\vee\ell(R)}\right]^\Delta.
\end{equation}
\end{lemma}

\begin{definition}\label{def Ap dims}
Let $p\in(0,\infty)$ and $W\in A_p$ be a matrix weight.
We say that $W$ has $A_p$-dimensions $(d,\widetilde d,\Delta)$ if
\begin{enumerate}[{\rm(i)}]
\item $W$ has the $A_p$-dimension $d$,
\item $p\in(0,1]$ and $\widetilde d=0$, or $p\in(1,\infty)$
and $W^{-\frac 1{p-1}}$(which belongs to $A_{p'}$)
has the $A_{p'}$-dimension $\widetilde d$, and
\item $\Delta$ is the same as in \eqref{Delta}.
\end{enumerate}
\end{definition}

The following lemma is well known; we omit the details.

\begin{lemma}\label{famous 2}
Let $\alpha\in(0,1]$.
Then, for any $\{z_i\}_{i\in\mathbb{N}}\subset\mathbb{C}$,
$(\sum_{i=1}^\infty|z_i|)^{\alpha}
\leq\sum_{i=1}^\infty|z_i|^{\alpha}.$
\end{lemma}

Sometimes dealing with the range of exponents $q\in(0,\infty)$
leads to a clumsy case study of $q\in(0,1]$, where Lemma \ref{famous 2} can be applied,
and $q\in(1,\infty)$, where this is replaced by H\"older's inequality.
For the sake of efficiency,
we record the following elementary inequality that conveniently combines the two,
and thereby removed the need of some separate case studies.

\begin{lemma}\label{one fits all}
Let $q\in(0,\infty]$ and $a_i,b_i\in[0,\infty)$ for any $i\in\mathscr I$
be numbers with some countable index set $\mathscr I$. Then
\begin{equation*}
\sum_{i\in\mathscr I}a_ib_i
\leq\left[\sum_{i\in\mathscr I}a_i\right]^{(1-\frac1q)_+}
\left[\sum_{i\in\mathscr I}a_i^{\min\{1,q\}}b_i^q\right]^{\frac1q},
\end{equation*}
where the interpretation of the second factor
when $q=\infty$ is just $\|b\|_\infty:=\sup_{i\in\mathscr I}b_i$.
\end{lemma}

\begin{proof}
For $q=\infty$, the estimate is trivial.
For any $q\in(1,\infty)$, this is the usual H\"older's inequality,
applied to $a_i^{\frac 1{q'}}$ and $a_i^{\frac 1q}b_i$.
For any $q\in(0,1]$, this is Lemma \ref{famous 2}, applied to $z_i=a_ib_i$;
note that the first factor in the claim
is raised to power $0$ in this case and becomes just $1$.
This finishes the proof of Lemma \ref{one fits all}.
\end{proof}

The following lemmas can be proved by some simple computations; we omit the details.

\begin{lemma}\label{33y}
For any cubes $Q,R\subset\mathbb{R}^n$,
any $x,x'\in Q$, and any $y,y'\in R$,
$$
1+\frac{|x-y|}{\ell(Q)\vee\ell(R)}
\sim1+\frac{|x'-y'|}{\ell(Q)\vee\ell(R)},
$$
where the positive equivalence constants depend only on $n$.
\end{lemma}

\begin{lemma}
Let $a\in(n,\infty)$.
Then, for any $j\in\mathbb{Z}$ and $y\in\mathbb{R}^n$,
\begin{equation}\label{33x}
\int_{\mathbb{R}^n}\frac{2^{jn}}{(1+|2^jx+y|)^a}\,dx
\sim1.
\end{equation}
Moreover, for any $j\in\mathbb{Z}$ with $j\leq0$ and for any $y\in\mathbb{R}^n$,
\begin{equation}\label{33z}
\sum_{k\in\mathbb{Z}^n}\frac{2^{jn}}{(1+|2^jk+y|)^a}
\sim1.
\end{equation}
Here all the positive equivalence constants depend only on $a$ and $n$.
\end{lemma}

\begin{lemma}\label{33}
Let $P\in\mathscr{Q}$ and $k\in\mathbb{Z}^n$ with $\|k\|_{\infty}\geq2$.
Then, for any $j\in\{j_P,j_P+1,\ldots\}$, $x\in P$, and $y\in P+k\ell(P)$,
$1+2^j|x-y|\sim2^{j-j_P}|k|,$
where the positive equivalence constants depend only on $n$.
\end{lemma}

\section{Besov-Type and Triebel--Lizorkin-Type Sequence Spaces}\label{BTL}

Now, we recall the concept of matrix-weighted
Besov-type and Triebel--Lizorkin-type sequence spaces
(see \cite[Definitions 3.24 and 3.26]{bhyy1}).
We need the following notation.

Let $s\in\mathbb{R}$, $\tau\in[0,\infty)$, and $p,q\in(0,\infty]$.
For any sequence $\{f_j\}_{j\in\mathbb Z}$ of measurable functions on $\mathbb{R}^n$,
any subset $J\subset\mathbb Z$, and any measurable set $E\subset\mathbb{R}^n$, let
\begin{align*}
\|\{f_j\}_{j\in\mathbb Z}\|_{L\dot B_{pq}(E\times J)}
:=\|\{f_j\}_{j\in\mathbb Z}\|_{\ell^qL^p(E\times J)}
:=\|\{f_j\}_{j\in\mathbb Z}\|_{\ell^q(J;L^p(E))}
:=\left[\sum_{j\in J}\|f_j\|_{L^p(E)}^q\right]^{\frac{1}{q}}
\end{align*}
and
\begin{align*}
\|\{f_j\}_{j\in\mathbb Z}\|_{L\dot F_{pq}(E\times J)}
:=\|\{f_j\}_{j\in\mathbb Z}\|_{L^p\ell^q(E\times J)}
:=\|\{f_j\}_{j\in\mathbb Z}\|_{L^p(E;\ell^q(J))}
:=\left\|\left(\sum_{j\in J}|f_j|^q\right)^{\frac{1}{q}}\right\|_{L^p(E)}
\end{align*}
with the usual modification made when $q=\infty$.
For simplicity of the presentation, in what follows, we may drop the domain $E\times J$ from these symbols, when it is the full space $E\times J=\mathbb R^n\times\mathbb Z$.
We use $L\dot A_{pq}\in\{L\dot B_{pq},L\dot F_{pq}\}$ as a generic notation in statements that apply to both types of spaces.

In particular, for any $P\in\mathscr{Q}$, we abbreviate $\widehat{P}:=P\times\{j_P,j_P+1,\ldots\}$ so that
$$
\|\{f_j\}_{j\in\mathbb Z}\|_{L\dot B_{pq}(\widehat{P})}
=\|\{f_j\}_{j\in\mathbb Z}\|_{\ell^qL^p(\widehat{P})}
=\left[\sum_{j=j_P}^\infty\|f_j\|_{L^p(P)}^q\right]^{\frac{1}{q}}
$$
and
$$
\|\{f_j\}_{j\in\mathbb Z}\|_{L\dot F_{pq}(\widehat{P})}
=\|\{f_j\}_{j\in\mathbb Z}\|_{L^p\ell^q(\widehat{P})}
=\left\|\left(\sum_{j=j_P}^\infty|f_j|^q\right)^{\frac{1}{q}}\right\|_{L^p(P)}.
$$
Let us further define
\begin{equation}\label{LApq}
\|\{f_j\}_{j\in\mathbb Z}\|_{L\dot A_{p,q}^\tau}
:=\sup_{P\in\mathscr{Q}}|P|^{-\tau}\|\{f_j\}_{j\in\mathbb Z}\|_{L\dot A_{pq}(\widehat{P})}
\end{equation}
for both choices of $L\dot A_{p,q}^\tau\in\{L\dot B_{p,q}^\tau,L\dot F_{p,q}^\tau\}$.

In what follows, for any $Q\in\mathscr{Q}$,
let $\widetilde{\mathbf{1}}_Q:=|Q|^{-\frac12}\mathbf{1}_Q$.

To motivate the definition of matrix-weighted Besov and Triebel--Lizorkin spaces,
we first recall the concept of the matrix-weighted Lebesgue space
(see, for instance, \cite[p.\,450]{v97}).

\begin{definition}
Let $p\in(0,\infty)$ and $W$ be a matrix weight.
The \emph{matrix-weighted Lebesgue space} $L^p(W,\mathbb{R}^n)$
is defined to be the set of all measurable vector-valued functions
$\vec f:\ \mathbb{R}^n\to\mathbb{C}^m$ such that
$$
\left\|\vec{f}\right\|_{L^p(W,\mathbb{R}^n)}
:=\left[\int_{\mathbb{R}^n}
\left|W^{\frac{1}{p}}(x)\vec{f}(x)\right|^p\,dx\right]^{\frac{1}{p}}<\infty.
$$
\end{definition}

In what follows, we denote $L^p(W,\mathbb{R}^n)$ simply by $L^p(W)$.
For any measurable vector-valued function $\vec f:\ \mathbb{R}^n\to\mathbb{C}^m$
and any measurable set $E$, we define
$\|\vec{f}\|_{L^p(W,E)}
:=\|\vec{f}\mathbf{1}_E\|_{L^p(W)}.$

\begin{definition}\label{def BTL}
Let $s\in\mathbb{R}$, $\tau\in[0,\infty)$, $p\in(0,\infty)$, $q\in(0,\infty]$, and $W\in A_p$.
The \emph{homogeneous matrix-weighted Besov-type sequence space} $\dot b^{s,\tau}_{p,q}(W)$
and the \emph{homogeneous matrix-weighted Triebel--Lizorkin-type sequence space}
$\dot f^{s,\tau}_{p,q}(W)$
are defined to be the sets of all sequences
$\vec t:=\{\vec t_Q\}_{Q\in\mathscr{Q}}\subset\mathbb{C}^m$ such that
$$
\left\|\vec{t}\right\|_{\dot a^{s,\tau}_{p,q}(W)}
:=\left\|\left\{2^{js}\left|W^{\frac{1}{p}}\vec t_j
\right|\right\}_{j\in\mathbb Z}\right\|_{L\dot A_{p,q}^\tau}<\infty,
$$
where, for any $j\in\mathbb{Z}$,
\begin{equation}\label{vec tj}
\vec t_j:=\sum_{Q\in\mathscr{Q}_j}\vec{t}_Q\widetilde{\mathbf{1}}_Q
\end{equation}
and $\|\cdot\|_{L\dot A_{p,q}^\tau}$ is the same as in \eqref{LApq}.
\end{definition}

\begin{definition}\label{def BTL ave}
Let $s\in\mathbb{R}$, $\tau\in[0,\infty)$, $p\in(0,\infty)$, $q\in(0,\infty]$, $W\in A_p$,
and $\mathbb{A}:=\{A_Q\}_{Q\in\mathscr{Q}}$ be a sequence of
reducing operators of order $p$ for $W$.
The \emph{homogeneous averaging matrix-weighted Besov-type sequence space}
$\dot{b}^{s,\tau}_{p,q}(\mathbb{A})$
and the \emph{homogeneous averaging matrix-weighted Triebel--Lizorkin-type sequence space}
$\dot{f}^{s,\tau}_{p,q}(\mathbb{A})$
are defined to be the sets of all sequences
$\vec t:=\{\vec t_Q\}_{Q\in\mathscr{Q}}\subset\mathbb{C}^m$ such that
$$
\left\|\vec{t}\right\|_{\dot{a}^{s,\tau}_{p,q}(\mathbb{A})}
:=\left\|\left\{2^{js}\left|A_j\vec t_j\right|\right\}_{j\in\mathbb Z}\right\|_{L\dot A_{p,q}^\tau}<\infty,
$$
where $\vec t_j$ for any $j\in\mathbb{Z}$ and $\|\cdot\|_{L\dot B_{p,q}^\tau}$
are the same as, respectively, in \eqref{vec tj} and \eqref{LApq}
and, for any $j\in\mathbb{Z}$,
\begin{equation}\label{Aj}
A_j:=\sum_{Q\in\mathscr{Q}_j}A_Q\mathbf{1}_Q.
\end{equation}
\end{definition}

Here and in what follows, it is understood that the symbol $a$
should be consistently replaced either by $b$ or by $f$ throughout the entire statement.

{\color{red}
The following classification of these spaces was introduced in the first part of this series \cite[Definition 4.20]{bhyy1}, and it will also be useful in the work at hand:

\begin{definition}\label{critical}
We say that a space of Besov-type or Triebel--Lizorkin-type, with parameters $(p,q,s,\tau)$, is
\begin{enumerate}[\rm(i)]
\item \emph{supercritical} if $\tau>\frac{1}{p}$ or $(\tau,q)=(\frac{1}{p},\infty)$,
\item \emph{critical} if $\tau=\frac1p$ and $q<\infty$ and the space is of Triebel--Lizorkin-type,
\item \emph{subcritical} if $\tau<\frac1p$, or if $\tau=\frac1p$ and $q<\infty$ and the space is of Besov-type.
\end{enumerate}
\end{definition}
}

{\color{red}For critical and supercritical spaces,
we have shown in \cite[Corollary 4.18]{bhyy1} that
\begin{equation*}
  \dot a^{s,\tau}_{p,q}(W)=\dot f^{s+n(\tau-\frac1p)}_{\infty,q+\infty(\tau-\frac1p)}(\mathbb A),\qquad\infty\cdot 0:=0,
\end{equation*}
where the space on the right is defined in Definition \ref{def TLpInfty}. (Note that $p=\infty$ is excluded in Definitions \ref{def BTL} and \ref{def BTL ave}.)
In this case, we will obtain the boundedness on $\dot a^{s,\tau}_{p,q}(W)$ of almost diagonal operators
by first establishing the boundedness on $\dot f_{\infty,q}^s(\mathbb{A})$ of almost diagonal operators, with sharp bounds in some cases.}

\section{Almost Diagonal Operators}
\label{almost}

In this section, we recall the definition and existing results about the boundedness of
almost diagonal operators on several spaces related to $\dot a^{s,\tau}_{p,q}(W)$.
Let us begin with some symbols.

Let $B:=\{b_{Q, P}\}_{Q, P\in\mathscr{Q}}\subset\mathbb{C}$.
For any sequence $\vec t:=\{\vec t_R\}_{R\in\mathscr{Q}}\subset\mathbb{C}^m$,
we define $B\vec t:=\{(B\vec t)_Q\}_{Q\in\mathscr{Q}}$ by setting,
for any $Q\in\mathscr{Q}$,
$(B\vec t)_Q:=\sum_{R\in\mathscr{Q}}b_{Q,R}\vec t_R$
if the above summation is absolutely convergent.

Now, we recall the concept of almost diagonal operators.

\begin{definition}\label{def AD}
Let $D,E,F\in\mathbb{R}$. We define the special infinite matrix
$B^{DEF}:=\{b_{Q,R}^{DEF}\}_{Q,R\in\mathscr{Q}}\subset\mathbb{C}$
by setting, for any $Q,R\in\mathscr{Q}$,
\begin{equation}\label{bDEF}
b_{Q,R}^{DEF}
:=\left[1+\frac{|x_Q-x_R|}{\ell(Q)\vee\ell(R)}\right]^{-D}
\left\{\begin{aligned}
&\left[\frac{\ell(Q)}{\ell(R)}\right]^E&&\text{if }\ell(Q)\leq\ell(R),\\
&\left[\frac{\ell(R)}{\ell(Q)}\right]^F&&\text{if }\ell(R)<\ell(Q).
\end{aligned}\right.
\end{equation}

An infinite matrix $B:=\{b_{Q,R}\}_{Q,R\in\mathscr{Q}}\subset\mathbb{C}$
is said to be \emph{$(D,E,F)$-almost diagonal} if there exists a positive constant $C$
such that, for any $Q,R\in\mathscr{Q}$,
$|b_{Q,R}|\leq Cb_{Q,R}^{DEF}.$
\end{definition}

\begin{remark}
\begin{enumerate}[\rm(i)]
\item If $E+F>0$ (which is always the case in all situations of interest to us),
then the second factor on the right-hand side of \eqref{bDEF} could be equivalently written as
\begin{equation*}
\min\left\{\left[\frac{\ell(Q)}{\ell(R)}\right]^E,
\left[\frac{\ell(R)}{\ell(Q)}\right]^F\right\}.
\end{equation*}

\item It is immediate from the definition that the matrix $B^{DEF}$ itself is $(D,E,F)$-almost diagonal. We will use this matrix in several counterexamples that demonstrate the sharpness of the almost diagonality assumptions in some results below.
\end{enumerate}
\end{remark}

The relevance of almost diagonal operators lies in the results of the following type,
obtained in \cite{fj90,fr04,fr21,ro03,t13,yy10,yy13}.
In the next theorem, as well as in the remainder of this article,
it is convenient to adopt the following usual convention.

\begin{definition}
In the context of any version (including in particular the weighted ones) of the Besov-type and Triebel--Lizorkin-type spaces $\dot A^{s,\tau}_{p,q}$ or the corresponding sequence spaces $\dot a^{s,\tau}_{p,q}$, the symbol $J$ will always carry the following meaning:
\begin{equation}\label{J}
J:=\left\{\begin{aligned}
&\frac{n}{\min\{1,p\}}&&\text{if we are dealing with a Besov-type space},\\
&\frac{n}{\min\{1,p,q\}}&&\text{if we are dealing with a Triebel--Lizorkin-type space}.
\end{aligned}\right.
\end{equation}
\end{definition}

Note that this definition is consistent with the fact that the two scales of spaces coincide when $p=q$ because also $\min\{1,p,q\}=\min\{1,p\}$ in this case.

\begin{theorem}\label{ad FJ-YY}
Let $s\in\mathbb{R}$, $\tau\in[0,\infty)$, $p\in(0,\infty)$, and $q\in(0,\infty]$.
Let $W\in A_p$ be a matrix weight and let $\beta_W$ be its doubling exponent of order $p$
(see, for instance, \cite[Definition 2.20]{bhyy1}).
Let $w\in A_p$ be a scalar weight and let $\alpha_1,\alpha_2\in(0,\infty)$ satisfy that
there exist two positive constants $C_1,C_2$ such that,
for any cubes $Q,R\subset\mathbb{R}^n$ with $Q\subset R$,
$$
C_1\left(\frac{|Q|}{|R|}\right)^{\alpha_2}
\leq\frac{w(Q)}{w(R)}
\leq C_2\left(\frac{|Q|}{|R|}\right)^{\alpha_1}.
$$
Let $B:=\{b_{Q,R}\}_{Q,R\in\mathscr{Q}}\subset\mathbb{C}$
be $(D,E,F)$-almost diagonal for some $D,E,F\in\mathbb R$. Then

\begin{enumerate}[\rm(i)]
\item\label{FJ 3.1}$B$ is bounded on $\dot a^s_{p,q}$ if
\begin{equation}\label{ad FJ}
D>J,\
E>\frac{n}{2}+s,\text{ and }
F>J-\frac{n}{2}-s;
\end{equation}

\item\label{YY 4.1}$B$ is bounded on $\dot a^{s,\tau}_{p,q}$ if
\begin{equation}\label{ad YY}
D>J_\tau,\
E>\frac{n}{2}+s+n\left(\tau-\frac{1}{p}\right)_+,\text{ and }
F>J_\tau-\frac{n}{2}-s-n\left(\tau-\frac{1}{p}\right)_+,
\end{equation}
where
\begin{equation}\label{tildeJ}
J_\tau:=\begin{cases}
n&\text{if }\tau>\frac{1}{p}\text{ or }(\tau,q)=(\frac{1}{p},\infty)\quad\textup{(}\text{``supercritical case''}\textup{)},\\
\displaystyle\frac{n}{\min\{1,q\}}&\text{if }\dot a^{s,\tau}_{p,q}=\dot f^{s,\frac{1}{p}}_{p,q}\text{ and }q<\infty\quad\textup{(}\text{``critical case''}\textup{)},\\
J&\text{if }\tau<\frac1p\text{, or }\dot a^{s,\tau}_{p,q}=\dot b^{s,\frac1p}_{p,q}\text{ and }q<\infty\quad\textup{(}\text{``subcritical case''}\textup{)};
\end{cases}
\end{equation}

\item\label{T 3}$B$ is bounded on $\dot a^{s,\tau}_{p,q}(w)$ if
$D>J+\varepsilon_0,$ $E>\frac{n}{2}+s+\frac{\varepsilon_0}{2},$ and
$F>J-\frac{n}{2}-s+\frac{\varepsilon_0}{2}$
with
$$
\varepsilon_0:=\max\left\{\frac{n(\alpha_2-\alpha_1)}{p},\
2n\left(\tau-\frac{\alpha_1}{p}\right)\right\};
$$

\item\label{Ro 1.10}$B$ is bounded on $\dot b^{s}_{p,q}(W)$ if
$D>J+\frac{\beta_W-n}{p},$ $E>\frac{n}{2}+s,$ and
$F>J-\frac{n}{2}-s+\frac{\beta_W-n}{p};$

\item\label{FR 2.6}$B$ is bounded on $\dot f^{s}_{p,q}(W)$ if
$D>J+\frac{\beta_W}{p},$ $E>\frac{n}{2}+s+\frac{n}{p},$ and
$F>J-\frac{n}{2}-s+\frac{\beta_W-n}{p}.$
\end{enumerate}
\end{theorem}

\begin{remark}\label{rem FJ-YY}
The statements of Theorem \ref{ad FJ-YY} are contained in the literature as follows.

Part \eqref{FJ 3.1}, when $\dot a^s_{p,q}=\dot f^s_{p,q}$, is \cite[Theorem 3.3]{fj90},
while the case $\dot a^s_{p,q}=\dot b^s_{p,q}$ is similar and easier.
Both cases of part \eqref{FJ 3.1} are also special cases of part \eqref{YY 4.1} with $\tau=0$.

On part \eqref{YY 4.1}, in the proof of \cite[Theorem 1]{yy13},
it is shown that $\dot a^{s,\tau}_{p,q}=\dot a^{s+n(\tau-\frac1p)}_{\infty,\infty}$
whenever either $\tau>\frac1p$ or $(\tau,q)=(\frac1p,\infty)$,
while, by \cite[Corollary 5.7]{fj90},
we find that $\dot f^{s,\frac 1p}_{p,q}=\dot f^s_{\infty,q}$.
Thus, both the supercritical and the critical cases of part \eqref{YY 4.1}
follow from the natural extension to $p=\infty$ of part \eqref{FJ 3.1}
(noting in particular that $J(\dot f^s_{\infty,q})=\frac{n}{\min\{1,q\}}$,
which is equal to $n$ if $q=\infty$), and that \cite[Theorem 3.3]{fj90}
``can be extended easily'' to this case was observed in \cite[page 81]{fj90}.
The remaining subcritical case of part \eqref{YY 4.1}
follows from \cite[Theorem 4.1]{yy10}, which deals more generally
with $\tau< \frac1p+\frac{\varepsilon}{2n}$, and,
in the subcritical case of part \eqref{YY 4.1}, we have $\tau\leq\frac1p$
and hence we can apply \cite[Theorem 4.1]{yy10} with an arbitrarily small $\varepsilon>0$.

Part \eqref{T 3} is a part of \cite[Theorem 3]{t13}
in which scalar weight $w\in A_\infty$; here,
for the convenience of comparison, we only consider the case that $w\in A_p$.
We should mention that the number $\frac{n(\alpha_2-\alpha_1)}p$
appearing in the definition of $\varepsilon_0$ in Part \eqref{T 3}
might not be correct; indeed,  in the proof of \cite[Theorem 3]{t13},
to obtain the desired estimate of $A_0$,  the assumption that $\varepsilon_0\ge\frac{n(\alpha_2-\alpha_1)}p$
is enough, however, this is not enough to guarantee the desired
estimate of $A_1$ which was just claimed by the author without a proof.

Part \eqref{Ro 1.10} is \cite[Theorem 1.10]{ro03} for $p\in[1,\infty)$
and \cite[Theorem 4.1]{fr04} for $p\in(0,1]$;
case $p=1$ is covered by both.
The conditions in the respective results are stated slightly differently,
but can be readily cast into the above form.

Part \eqref{FR 2.6} is the recent \cite[Theorem 2.6]{fr21}.
\end{remark}

Theorem \ref{ad FJ-YY}\eqref{FJ 3.1} motivates the following definition from \cite[p.\,53]{fj90}.

\begin{definition}\label{def ad 0}
An infinite matrix $B:=\{b_{Q,R}\}_{Q,R\in\mathscr{Q}}\subset\mathbb{C}$
is said to be \emph{$\dot a^s_{p,q}$-almost diagonal}
if it is $(D,E,F)$-almost diagonal with $D,E,F$ the same as in \eqref{ad FJ}.
\end{definition}

Using this concept, we could rephrase Theorem \ref{ad FJ-YY}\eqref{FJ 3.1}
by saying that an almost diagonal operator on $\dot a^s_{p,q}$
is bounded on $\dot a^s_{p,q}$.
Similar definitions based on the other parts of Theorem \ref{ad FJ-YY}
are made in the cited articles of Remark \ref{rem FJ-YY}, where these results were obtained,
but we will refrain from that at this point
because we will later see that the conditions in Theorem \ref{ad FJ-YY} can be improved.

Our goal in this article is to obtain similar results
for general spaces $\dot a^{s,\tau}_{p,q}(W)$
involving both the parameter $\tau$,
the same as in part \eqref{YY 4.1} of Theorem \ref{ad FJ-YY}, and a weight $W$,
the same as in parts \eqref{Ro 1.10} and \eqref{FR 2.6}.
We begin by providing a result that has a relatively simple proof
(a straightforward adaptation of the argument used in the proof of \cite[Theorem 2.6]{fr21}
to the case at hand)
by reduction to the above known results about the unweighted case.
A disadvantage is that this reduction requires
relatively heavy assumptions on almost diagonal parameters under consideration,
and we will subsequently strive to improve these assumptions
by working directly with the matrix-weighted spaces.
The following Theorem \ref{77} gives us a benchmark against
which we can compare the subsequent results further below.

\begin{theorem}\label{77}
Let $s\in\mathbb{R}$, $\tau\in[0,\infty)$,
$p\in(0,\infty)$, $q\in(0,\infty]$,
and let $W\in A_p$ have $A_p$-dimensions $(d,\widetilde d,\Delta)$.
Then $B:=\{b_{Q,P}\}_{Q,P\in\mathscr{Q}}\subset\mathbb{C}$
is bounded on $\dot a^{s,\tau}_{p,q}(W)$
if $B$ is $(D,E,F)$-almost diagonal with
$$
D>J_\tau+\Delta,\
E>\frac{n}{2}+s+n\left(\tau-\frac{1}{p}\right)_+ +\frac{d}{p},\text{ and }
F>J_\tau-\frac{n}{2}-s-n\left(\tau-\frac{1}{p}\right)_+ +\frac{\widetilde d}{p'},
$$
where $J_\tau$ is the same as in \eqref{tildeJ}.
\end{theorem}

\begin{proof}
By \cite[Theorem 3.27]{bhyy1}, we find that,
to prove the boundedness of $B$ on $\dot a^{s,\tau}_{p,q}(W)$,
it suffices to show that, for any $\vec t\in\dot{a}^{s,\tau}_{p,q}(\mathbb{A})$,
\begin{equation}\label{78}
\left\|B\vec t\right\|_{\dot{a}^{s,\tau}_{p,q}(\mathbb{A})}
\lesssim\left\|\vec t\right\|_{\dot{a}^{s,\tau}_{p,q}(\mathbb{A})},
\end{equation}
where $\mathbb{A}:=\{A_Q\}_{Q\in\mathscr{Q}}$ is a sequence of
reducing operators of order $p$ for $W$.

Let $u:=\{u_Q\}_{Q\in\mathscr{Q}}$,
where $u_Q:=|A_Q\vec t_Q|$ for any $Q\in\mathscr{Q}$.
Let $\widetilde{B}:=\{\widetilde{b}_{Q,R}\}_{Q,R\in\mathscr{Q}}\subset\mathbb{C}$, where
$\widetilde{b}_{Q,R}:=|b_{Q,R}|\,\|A_QA_R^{-1}\|$ for any $Q,R\in\mathscr{Q}$.
Then, from the definitions of both $\widetilde{b}_{Q,R}$ and $u_R$,
we infer that, for any $Q\in\mathscr{Q}$,
\begin{equation}\label{5.7x}
\left|A_Q\left(B\vec t\right)_Q\right|
\leq\sum_{R\in\mathscr{Q}}\left|b_{Q,R}A_Q\vec t_R\right|
\leq\sum_{R\in\mathscr{Q}}\widetilde{b}_{Q,R}u_R
=\left(\widetilde{B}u\right)_Q.
\end{equation}
Moreover, using Lemma \ref{22 precise}, we conclude that,
for any $Q,R\in\mathscr{Q}$,
\begin{align*}
\widetilde{b}_{Q,R}
&\lesssim|b_{Q,R}|\max\left\{\left[\frac{\ell(R)}{\ell(Q)}\right]^{\frac{d}{p}},
\left[\frac{\ell(Q)}{\ell(R)}\right]^{\frac{\widetilde d}{p'}}\right\}
\left[1+\frac{|x_Q-x_R|}{\ell(R)\vee\ell(Q)}\right]^{\Delta}\\
&\lesssim\left[1+\frac{|x_Q-x_R|}{\ell(R)\vee \ell(Q)}\right]^{-(D-\Delta)}
\left\{\begin{aligned}
&\left[\frac{\ell(Q)}{\ell(R)}\right]^{E-\frac{d}{p}}&&\text{if }\ell(Q)\leq\ell(R),\\
&\left[\frac{\ell(R)}{\ell(Q)}\right]^{F-\frac{\widetilde d}{p'}}&&\text{if }\ell(R)<\ell(Q),
\end{aligned}\right.
\end{align*}
here and in the follows, $\frac 1p+\frac1{p'}=1$.
Thus, $\widetilde{B}$ is $(\widetilde D,\widetilde E,\widetilde F)$-almost diagonal with
\begin{equation*}
\left(\widetilde D,\widetilde E,\widetilde F\right)
=\left(D-\Delta,E-\frac{d}{p},F-\frac{\widetilde d}{p'}\right).
\end{equation*}
If $(D,E,F)$ satisfies the assumptions of the present theorem,
then $(\widetilde D,\widetilde E,\widetilde F)$ also satisfies \eqref{ad YY},
and hence $\widetilde{B}$ is bounded on $\dot a^{s,\tau}_{p,q}$
by Theorem \ref{ad FJ-YY}\eqref{YY 4.1}.
This, combined with \eqref{5.7x} and the definition of
$\|\cdot\|_{\dot{a}^{s,\tau}_{p,q}(\mathbb{A})}$, further implies that
$$
\left\|B\vec t\right\|_{\dot{a}^{s,\tau}_{p,q}(\mathbb{A})}
\leq\left\|\widetilde{B}u\right\|_{\dot a^{s,\tau}_{p,q}}
\lesssim\|u\|_{\dot a^{s,\tau}_{p,q}}
=\left\|\vec t\right\|_{\dot{a}^{s,\tau}_{p,q}(\mathbb{A})}.
$$
This finishes the proof of \eqref{78} and hence Theorem \ref{77}.
\end{proof}

\begin{remark}
When $\tau=0$ and $\dot a^{s,\tau}_{p,q}(W)=\dot f^s_{p,q}(W)$,
Theorem \ref{77} coincides with \cite[Theorem 2.6]{fr21}.
When $\dot a^{s,\tau}_{p,q}(W)=\dot b^s_{p,q}(W)$,
similar results with slightly weaker assumptions than
the corresponding ones of Theorem \ref{77}
were obtained in \cite[Theorem 1.10]{ro03} for $p\in(1,\infty)$
and \cite[Theorem 4.1]{fr04} for $p\in(0,1]$.
We will improve both sets of assumptions further below.
\end{remark}

\section{Boundedness Theorem for Subcritical Spaces and the Key Lemma}
\label{better}

In this section, we formulate a new result about the boundedness of
almost diagonal operators on $\dot a^{s,\tau}_{p,q}(W)$.
Due to the length of the proof, we split it into three parts.
Initially, we consider special cases $\dot b^{s,0}_{p,q}(W)=\dot b^s_{p,q}(W)$
in Section \ref{sec ad Besov} and $\dot f^{s,0}_{p,q}(W)=\dot f^s_{p,q}(W)$
in Section \ref{sec ad TL}.
Subsequently, we complete this proof in Section \ref{sec ad general}.

The following theorem is one of the main results of this section.
While it applies to the full range of the parameters,
it is at its best in the subcritical spaces, and we will later see,
in Section \ref{special},
how to improve it in the critical and the supercritical regimes.
Recall that the notions of subcritical, critical, and supercritical
are defined in \eqref{tildeJ}; we will consistently apply this terminology.

\begin{theorem}\label{ad BF}
Let $s\in\mathbb R$, $\tau\in[0,\infty)$, $p\in(0,\infty)$, $q\in(0,\infty]$,
and $W\in A_p$ have the $A_p$-dimension $d\in[0,n)$.
If $B$ is $(D,E,F)$-almost diagonal,
then $B$ is bounded on $\dot a^{s,\tau}_{p,q}(W)$ whenever
\begin{equation}\label{ad new}
D>J+n\widehat\tau,\
E>\frac{n}{2}+s+n\widehat\tau,\text{ and }
F>J-\frac{n}{2}-s,
\end{equation}
where $J$ is the same as in \eqref{J}, and
\begin{equation}\label{tauJ}
\widehat\tau:=\left[\tau-\frac{1}{p}\left(1-\frac{d}{n}\right)\right]_+.
\end{equation}
\end{theorem}

\begin{remark}\label{ad BF vs benchmark}
\begin{enumerate}[\rm(i)]
\item Theorem \ref{ad BF} only depends on the $A_p$-dimension $d$ of $W$, in contrast to Theorem \ref{77} and some other result that depend on the whole triple of $A_p$-dimensions $(d,\widetilde d,\Delta)$.
\item Theorem \ref{ad BF} improves Theorem \ref{77}
in the subcritical case of \eqref{tildeJ}, when $J_\tau=J$.
That is, in this case $\tau\leq\frac1p$, and hence $(\tau-\frac{1}{p})_+=0$ in \eqref{ad YY},
while $n\widehat\tau\leq\frac{d}{p}\leq\Delta$ in \eqref{tauJ}.
It follows that the lower bounds of each of the parameters $D,E,F$ in Theorem \ref{77}
are at least as large as the respective lower bounds in Theorem \ref{ad BF}:
\begin{equation*}
J_\tau+\Delta\geq J+n\widehat\tau,\
\frac{n}{2}+s+n\left(\tau-\frac{1}{p}\right)_+ +\frac{d}{p}
\geq\frac{n}{2}+s+0+n\widehat\tau,
\end{equation*}
and
\begin{equation*}
J_\tau-\frac{n}{2}-s-n\left(\tau-\frac{1}{p}\right)_+ +\frac{\widetilde d}{p'}
\geq J-\frac{n}{2}-s-0+0.
\end{equation*}
\item On the other hand, Theorem \ref{ad BF} is not always superior to Theorem \ref{77}
in the critical and the supercritical ranges, and we will later see that
a modification of the argument is preferable in these cases
to obtain the best possible conclusion.
\end{enumerate}
\end{remark}

To prove Theorem \ref{ad BF}, we need several technical lemmas.
The following simple lemma allows us to restrict our consideration to $s=0$
when studying the boundedness of almost diagonal operators.

\begin{lemma}\label{s=0 enough}
Let $s\in\mathbb{R}$, $\tau\in[0,\infty)$,
$p\in(0,\infty)$, $q\in(0,\infty]$, and $W\in A_p$.
Let $D,E,F\in\mathbb R$.
Let $B:=\{b_{Q,R}\}_{Q,R\in\mathscr{Q}}\subset\mathbb{C}$
and $\widetilde B:=\{\widetilde b_{Q,R}\}_{Q,R\in\mathscr{Q}}$,
where $\widetilde{b}_{Q,R}:=[\ell(R)/\ell(Q)]^sb_{Q,R}$
for any $Q,R\in\mathscr{Q}$. Then
\begin{enumerate}[\rm(i)]
\item\label{ad s vs 0}$B$ is $(D,E,F)$-almost diagonal if and only if
$\widetilde B$ is $(D,E-s,F+s)$-almost diagonal;
\item\label{bd s vs 0}$B$ is bounded on $\dot a^{s,\tau}_{p,q}(W)$ if and only if
$\widetilde B$ is bounded on $\dot a^{0,\tau}_{p,q}(W)$.
\end{enumerate}
\end{lemma}

\begin{proof}
Item \eqref{ad s vs 0} is immediately deduced from the definition.
To show \eqref{bd s vs 0}, let $(J_s\vec{t})_R:=[\ell(R)]^{-s}\vec{t}_R$ for each $R\in\mathscr{Q}$.
Then $J_s:\ \dot a^{s,\tau}_{p,q}(W)\to\dot a^{0,\tau}_{p,q}(W)$
is an isometric isomorphism and
$B\vec{t}=J_{s}^{-1}\widetilde{B}J_s\vec{t}$ for each $\vec t\in\dot a^{s,\tau}_{p,q}(W)$.
From these observations, we infer \eqref{bd s vs 0}.
This finishes the proof of Lemma \ref{s=0 enough}.
\end{proof}

As a key step towards a direct proof of the boundedness of
almost diagonal operators on $\dot a^s_{p,q}$, we have the following conclusion.
{\color{red}This lemma is critical to our program of circumventing the Fefferman--Stein vector-valued maximal inequality.
To simplify the right-hand side below, it might seem tempting to dominate the various averages simply by the maximal function, but we have on purpose left them as they are.}

\begin{lemma}\label{ad prelim}
Let $p\in(0,\infty)$, $q\in(0,\infty]$,
and $B$ be $(D,E,F)$-almost diagonal for some $D,E,F\in\mathbb R$.
Then there exists a positive constant $C$ such that,
for any $\vec{t}=\{\vec{t}_Q\}_{Q\in\mathscr Q}\subset\mathbb C^m$
satisfying that $B\vec t$ is well defined, any $a\in(0,1]$,
and any sequence $\{H_j:\ \mathbb{R}^n\to M_m(\mathbb{C})\}_{j\in\mathbb{Z}}$
of measurable matrix-valued functions,
\begin{align}\label{ABt}
\left\|\left\{H_j\left(B\vec{t}\right)_j\right\}_{j\in\mathbb Z}\right\|_{L\dot A_{p,q}}^r
&\leq C\sum_{k\in\mathbb{Z}}\sum_{l=0}^\infty
\Bigg[2^{-(E-\frac{n}{2})k_-}2^{-k_+(F+\frac{n}{2}-\frac{n}{a})}2^{-(D-\frac{n}{a})l} \\
&\quad\times\left.\left\|\left\{
\left[\fint_{B(\cdot,2^{l+k_+-i})}\left| H_{i-k}(\cdot)\vec{t}_{i}(y)\right|^a
\,dy\right]^{\frac{1}{a}}\right\}_{i\in\mathbb Z}\right\|_{L\dot A_{p,q}}\right]^r,\notag
\end{align}
where $r:=p\wedge q\wedge1$,
both $\vec{t}_i$ and $(B\vec{t})_j$ are the same as in \eqref{vec tj},
and the quasi-norms are taken in either of
$L\dot A_{p,q}\in\{L^p\ell^q,\ell^qL^p\}$.
\end{lemma}

Before going to the proof, let us note that two prominent choices of $H_j$ for any $j\in\mathbb Z$
give us the following norms on the left of \eqref{ABt}.
Let $p\in(0,\infty)$, $q\in(0,\infty]$, $W\in A_p$,
and $\mathbb{A}:=\{A_Q\}_{Q\in\mathscr{Q}}$ be a sequence of
reducing operators of order $p$ for $W$. Then
\begin{equation*}
\left\|\left\{H_j\left(B\vec{t}\right)_j\right\}_{j\in\mathbb Z}\right\|_{L\dot A_{p,q}}
=\begin{cases}
\left\|B\vec{t}\right\|_{\dot a^0_{p,q}(W)}
&\text{if }H_j=W^{\frac{1}{p}}\text{ for all }j\in\mathbb Z,\\
\left\|B\vec{t}\right\|_{\dot a^0_{p,q}(\mathbb{A})}
&\text{if }H_j=A_j\text{ for all }j\in\mathbb Z,
\end{cases}
\end{equation*}
where $A_j$ for any $j\in\mathbb Z$ is the same as in \eqref{Aj}.
Thus, whenever we can estimate the $L\dot A_{p,q}$ quasi-norm on the right-hand side
by $\|\vec{t}\|_{\dot a^0_{p,q}(W)}$ times a function of $(k,l)$
for which the double series is summable, then
\begin{align*}
\left\|B\vec{t}\right\|_{\dot a^0_{p,q}(W)}
\lesssim\left\|\vec{t}\right\|_{\dot a^0_{p,q}(W)}
\left\{\sum_{k\in\mathbb{Z}}\sum_{l=0}^\infty
\left[2^{-(E-\frac{n}{2})k_-}2^{-k_+(F+\frac{n}{2}-\frac{n}{a})}2^{-(D-\frac{n}{a})l}\right]^r \right\}^{\frac{1}{r}}
\sim\left\|\vec{t}\right\|_{\dot a^0_{p,q}(W)},
\end{align*}
we obtain the boundedness of $B$ on $\dot a^0_{p,q}(W)$,
recalling that this space coincides with $\dot a^0_{p,q}(\mathbb A)$
by \cite[Theorem 3.27]{bhyy1}.
We actually only use the below version with $H_j=W^{\frac{1}{p}}$
for all $j\in\mathbb Z$,
but we find it worthwhile pointing out that,
up to this point, everything works for the other version as well.

\begin{proof}[Proof of Lemma \ref{ad prelim}]
Let us begin by setting up some convenient notation.
We use the big-$O$ notation in vectorial computations as follows.
As usual, $f=O(g)$ means $|f|\leq Cg$ for some positive constant $C$ that
may change from one occurrence to the next
but it is independent of the main parameters in both $f$ and $g$. Thus, for instance,
\begin{equation}\label{vOvi}
\vec{v}=\sum_iO(1)\vec{v}_i
\end{equation}
means that
$\vec{v}=\sum_ic_i\vec{v}_i$ and, for any $i$, $|c_i|\leq C$
with $C$ being a positive constant independent of $i$.
This further implies, in particular, that
\begin{equation}\label{v<vi}
|\vec{v}|\lesssim\sum_i|\vec{v}_i|,
\end{equation}
but \eqref{vOvi} is strictly stronger than \eqref{v<vi}.
For instance, \eqref{vOvi} also further implies that, for any matrix $A$, we have
$\left|A\vec v\right|
\lesssim\sum_i\left|A\vec{v}_i\right|,$
which will be essential for us below.

Turning to the actual proof,
let $j\in\mathbb{Z}$ and $x\in Q\in\mathscr{Q}_j$.
Employing the notation just introduced with the almost diagonal assumptions, we have
\begin{align}\label{ad step 1}
\left(B\vec{t}\right)_Q
&=\sum_{R\in\mathscr{Q}}b_{Q,R}\vec{t}_R\\
&=\sum_{i\in\mathbb{Z}}
O\left(2^{-(j-i)_+E}2^{-(i-j)_+ F}\right)
\sum_{R\in\mathscr{Q}_i}
O\left(\left[1+2^{i\wedge j}|x-x_R|\right]^{-D}\right)\vec{t}_R.\notag
\end{align}
Notice that, by Lemma \ref{33y}, we obtain
\begin{align*}
\sum_{R\in\mathscr{Q}_i}O\left(\left[1+2^{i\wedge j}|x-x_R|\right]^{-D}\right)\vec{t}_R
&=\sum_{R\in\mathscr{Q}_i}2^{i\frac{n}{2}}\int_R
O\left(\left[1+2^{i\wedge j}|x-y|\right]^{-D}\right)\frac{\vec{t}_R}{|R|^{\frac12}}\,dy\\
&=2^{i\frac{n}{2}}\int_{\mathbb{R}^n}
O\left(\left[1+2^{i\wedge j}|x-y|\right]^{-D}\right)\vec{t}_i(y)\,dy,
\end{align*}
where $\vec t_i$ is the same as in \eqref{vec tj}. Moreover, we obviously have
\begin{align*}
&\int_{\mathbb{R}^n}
O\left(\left[1+2^{i\wedge j}|x-y|\right]^{-D}\right)\vec{t}_i(y)\,dy\\
&\quad=\int_{B(x,2^{-(i\wedge j)})}O(1)\vec{t}_i(y)\,dy +\sum_{l=1}^\infty\int_{B(x,2^{l-(i\wedge j)})\setminus B(x,2^{l-1-(i\wedge j)})}
O\left(2^{-D l}\right)\vec{t}_i(y)\,dy\\
&\quad=\sum_{l=0}^\infty2^{-Dl}2^{[l-(i\wedge j)]n}\fint_{B(x,2^{l-(i\wedge j)})}O(1)\vec{t}_i(y)\,dy.
\end{align*}
Recalling the notation \eqref{vec tj}, we thus have
\begin{align*}
\left(B\vec{t}\right)_j(x)
&=\sum_{i\in\mathbb{Z}}2^{-(j-i)_+E}2^{-(i-j)_+ F}
2^{(i+j)\frac{n}{2}}2^{-(i\wedge j)n}
\sum_{l=0}^\infty2^{-(D-n)l}
\fint_{B(x,2^{l-(i\wedge j)})}O(1)\vec{t}_i(y)\,dy\\
&=\sum_{i\in\mathbb{Z}}2^{-(j-i)_+(E-\frac{n}{2})}2^{-(i-j)_+(F-\frac{n}{2})}
\sum_{l=0}^\infty2^{-(D-n)l}\fint_{B(x,2^{l-(i\wedge j)})}O(1)\vec{t}_i(y)\,dy.
\end{align*}
Notice that, for any $i\in\mathbb{Z}$ and $R\in\mathscr{Q}_i$,
$\vec{t}_i$ is constant on $R$.
From this and Lemma \ref{famous 2}, we deduce that, for any matrix $A\in M_m(\mathbb{C})$,
any $x\in\mathbb{R}^n$, $\ell\in\mathbb{Z}_+$, and $i\in\mathbb{Z}$,
\begin{align*}
\fint_{B(x,2^{l-(i\wedge j)})}\left|A\vec{t}_i(y)\right|\,dy
&\lesssim\sum_{\substack{Q\in\mathscr{Q}_i\\
Q\cap B(x,2^{l-(i\wedge j)})\neq\varnothing}}
\left|A\vec{t}_i(x_Q)\right|\frac{2^{-in}}{2^{[l-(i\wedge j)]n}}\\
&\lesssim\left[\sum_{\substack{Q\in\mathscr{Q}_i\\
Q\cap B(x,2^{l-(i\wedge j)})\neq\varnothing}}
\left|A\vec{t}_i(x_Q)\right|^a\right]^{\frac{1}{a}}
\frac{2^{-in}}{2^{[l-(i\wedge j)]n}}\\
&\lesssim\left[\fint_{B(x,c_n2^{l-(i\wedge j)})}
\left|A\vec{t}_i(y)\right|^a\,dy\right]^{\frac{1}{a}}
2^{[l+i-(i\wedge j)]n(\frac{1}{a}-1)},
\end{align*}
where $i-(i\wedge j)=(i-j)_+$ and $c_n:=1+\sqrt{n}$.
Choosing $A=H_j(x)$ and reindexing the summation in $l$,
we can drop the expansion factor $c_n$ from the ball in the integral average to arrive at
\begin{align*}
&\left|H_j(x)\left(B\vec{t}\right)_j(x)\right|\\
&\quad\lesssim\sum_{i\in\mathbb{Z}}2^{-(j-i)_+(E-\frac{n}{2})}
2^{-(i-j)_+(F+\frac{n}{2}-\frac{n}{a})}
\sum_{l=0}^\infty2^{-(D-\frac{n}{a})l}
\left[\fint_{B(x,2^{l-(i\wedge j)})}
\left|H_j(x)\vec{t}_i(y)\right|^a\,dy\right]^{\frac{1}{a}}\\
&\quad=\sum_{k\in\mathbb{Z}}2^{-(E-\frac{n}{2})k_-}2^{-k_+(F+\frac{n}{2}-\frac{n}{a})}
\sum_{l=0}^\infty2^{-(D-\frac{n}{a})l}
\left[\fint_{B(x,2^{l-[(j+k)\wedge j]})}
\left|H_j(x)\vec{t}_{j+k}(y)\right|^a\,dy\right]^{\frac{1}{a}}.
\end{align*}
Recall that $L\dot A_{p,q}\in\{L^p\ell^q,\ell^qL^p\}$.
Since $\|\cdot\|_{L\dot A_{p,q}}^r$ satisfies the triangle inequality, it follows that
\begin{align}\label{5.15x}
\left\|\left\{H_j\left(B\vec{t}\right)_j\right\}_{j\in\mathbb Z}\right\|_{L\dot A_{p,q}}^r
&\lesssim\sum_{k\in\mathbb{Z}}\sum_{l=0}^\infty
\Bigg[2^{-(E-\frac{n}{2})k_-}2^{-k_+(F+\frac{n}{2}-\frac{n}{a})}2^{-(D-\frac{n}{a})l}\\
&\quad\times\left.\left\|\left\{\left[\fint_{B(\cdot,2^{l-[(j+k)\wedge j]})}
\left|H_j(\cdot)\vec{t}_{j+k}(y)\right|^a\,dy\right]^{\frac{1}{a}}
\right\}_{j\in\mathbb Z}\right\|_{L\dot A_{p,q}}\right]^r.\notag
\end{align}
By the shift invariance of $\ell^q$, we have
\begin{align*}
&\left\|\left\{\left[\fint_{B(\cdot,2^{l-[(j+k)\wedge j]})}
\left|H_j(\cdot)\vec{t}_{j+k}(y)\right|^a\,dy\right]^{\frac{1}{a}}
\right\}_{j\in\mathbb Z}\right\|_{L\dot A_{p,q}}\\
&\quad=\left\|\left\{\left[\fint_{B(\cdot,2^{l-[i\wedge (i-k)]})}
\left|H_{i-k}(\cdot)\vec{t}_{i}(y)\right|^a\,dy\right]^{\frac{1}{a}}
\right\}_{j\in\mathbb Z}\right\|_{L\dot A_{p,q}},
\end{align*}
where $i\wedge(i-k)=i-k_+$,
and substituting this back into \eqref{5.15x} gives the claimed estimate \eqref{ABt}.
This finishes the proof of Lemma \ref{ad prelim}.
\end{proof}

\begin{remark}\label{ad conv}
By inspection of the proof of Lemma \ref{ad prelim} [see in particular \eqref{ad step 1}],
the same proof gives a slightly stronger estimate,
with the left-hand side of \eqref{ABt} replaced by
\begin{equation*}
\left\|\left\{\sum_{Q\in\mathscr{Q}_j}\widetilde{\mathbf{1}}_Q(\cdot)
\sum_{R\in\mathscr{Q}}\left|H_j(\cdot)b_{Q,R}\vec{t}_R\right|
\right\}_{j\in\mathbb Z}\right\|_{L\dot A_{p,q}}^r.
\end{equation*}
In particular, whenever we have the finiteness of the right-hand side of \eqref{ABt},
then, for any $j\in\mathbb Z$, $Q\in\mathscr{Q}_j$, and almost every $x\in Q$,
we have
$\sum_{R\in\mathscr{Q}}
|H_j(x)b_{Q,R}\vec{t}_R|<\infty.$
Whenever we have the finiteness of the right-hand side of \eqref{ABt}
with functions $H_j$ such that,
for any $j\in\mathbb Z$ and $Q\in\mathscr{Q}_j$,
the matrix $H_j(x)$ is invertible at almost every $x\in Q$
(which holds in particular in the primary case of
interest with $H_j=W^{\frac{1}{p}}$ for any $j\in\mathbb Z$ and $W\in A_p$),
then it follows from
\begin{equation*}
\left|b_{Q,R}\vec{t}_R\right|
\leq\left\|\left[H_j(x)\right]^{-1}\right\|\left|H_j(x)b_{Q,R}\vec{t}_R\right|,\
\forall\,R\in\mathscr{Q},
\end{equation*}
that the series defining $(B\vec{t})_Q$ is absolutely convergent,
in other word,
$\sum_{R\in\mathscr{Q}}|b_{Q,R}\vec{t}_R|<\infty$
for any $Q\in\mathscr{Q}$ and $\vec{t}=\{\vec{t}_R\}_{R\in\mathscr{Q}}\subset\mathbb C^m$
that makes the right-hand side of \eqref{ABt} finite.
\end{remark}

\section{Boundedness on Besov Sequence Spaces}\label{sec ad Besov}

In this section, we prove the special case of Theorem \ref{ad BF}
for the Besov sequence space $\dot b_{p,q}^s(W)=\dot b_{p,q}^{s,0}(W)$.
Besides being an interesting case in its own right,
the techniques that we develop will also be applied to the proof of
the full Theorem \ref{ad BF} further below.
We begin by estimating the quantities on the right-hand side of
the conclusion of Lemma \ref{ad prelim} in the case at hand.

\begin{lemma}\label{ad B1}
Let $p\in(0,\infty)$, $q\in(0,\infty]$,
$W\in A_p$, and $a:=p\wedge1$. Then
there exists a positive constant $C$,
depending only on $p$ and $n$,
such that, for any $k,l\in\mathbb Z$ and $\vec{t}\in\dot b^0_{p,q}(W)$,
\begin{equation}\label{162}
\left\|\left\{\left[\fint_{B(\cdot,2^{l+k_+-i})}
\left|W^{\frac{1}{p}}(\cdot)\vec{t}_i(y)\right|^a\,dy\right]^{\frac{1}{a}}
\right\}_{i\in\mathbb Z}\right\|_{\ell^qL^p}
\leq C[W]_{A_p}^{\frac1p}\left\|\vec{t}\right\|_{\dot b^0_{p,q}(W)},
\end{equation}
where $\vec t_i$ for any $i\in\mathbb Z$ is the same as in \eqref{vec tj}.
\end{lemma}

\begin{proof}
Let $k,l,i\in\mathbb Z$ and $\vec{t}\in\dot b^0_{p,q}(W)$ be fixed.
For any $x\in\mathbb{R}^n$ and any measurable set
$E\subset\mathbb{R}^n$ with $|E|\in(0,\infty)$, write
$$
F(x,E):=\left[\fint_E
\left|W^{\frac{1}{p}}(x)\vec{t}_i(y)\right|^a\,dy\right]^{\frac{1}{a}}.
$$
We need to estimate $\|F(\cdot,B(\cdot,\rho))\|_{L^p}$,
where $\rho:=2^{\ell+k_+-i}$.
Let $\mathscr{D}:=\mathscr{Q}_{i-\ell-k_+}$
be a partition of $\mathbb R^n$ into dyadic cubes of edge length $\rho$.
Then, by the definition of $F(x,E)$, we have
\begin{align}\label{161}
\|F(\cdot,B(\cdot,\rho))\|_{L^p}^p
&=\sum_{Q\in\mathscr{D}}\int_Q [F(x,B(x,\rho))]^p\,dx\\
&\lesssim\sum_{Q\in\mathscr{D}}\int_Q [F(x,Q^*)]^p\,dx
\leq\sum_{Q\in\mathscr{D}}\int_{Q^*} [F(x,Q^*)]^p\,dx,\notag
\end{align}
where $Q^*:=3Q$ and we used the fact that $B(x,\rho)\subset Q^*$ for any $x\in Q$.

If $p\in(0,1]$, then, for any $Q\in\mathscr{D}$,
\begin{align*}
\int_{Q^*}[F(x,Q^*)]^p\,dx
&=\int_{Q^*}\fint_{Q^*}
\left|W^{\frac{1}{p}}(x)W^{-\frac{1}{p}}(y)W^{\frac{1}{p}}(y)\vec{t}_i(y)\right|^p\,dy\,dx\\
&\leq\int_{Q^*}\left[\fint_{Q^*}\left\|W^{\frac{1}{p}}(x)W^{-\frac{1}{p}}(y)\right\|^p\,dx\right]
\left|W^{\frac{1}{p}}(y)\vec{t}_i(y)\right|^p\,dy\\
&\leq[W]_{A_p}\int_{Q^*}\left|W^{\frac{1}{p}}(y)\vec{t}_i(y)\right|^p\,dy
\end{align*}
directly by Definition \ref{def ap}.
If $p\in(1,\infty)$, then,
from H\"older's inequality and Definition \ref{def ap} again, it follows that,
for any $Q\in\mathscr{D}$,
\begin{align*}
\int_{Q^*}[F(x,Q^*)]^p\,dx
&=\int_{Q^*}\left[\fint_{Q^*}
\left|W^{\frac{1}{p}}(x)W^{-\frac{1}{p}}(y)W^{\frac{1}{p}}(y)\vec{t}_i(y)\right|\,dy\right]^p\,dx\\
&\leq\int_{Q^*}
\left[\fint_{Q^*}\left\|W^{\frac{1}{p}}(x)W^{-\frac{1}{p}}(y)\right\|^{p'}\,dy\right]^{\frac{p}{p'}}
\left[\fint_{Q^*}\left|W^{\frac{1}{p}}(y)\vec{t}_i(y)\right|^p\,dy\right]\,dx\\
&=\fint_{Q^*}\left[\fint_{Q^*}
\left\|W^{\frac{1}{p}}(x)W^{-\frac{1}{p}}(y)\right\|^{p'}\,dy\right]^{\frac{p}{p'}}\,dx
\int_{Q^*}\left|W^{\frac{1}{p}}(y)\vec{t}_i(y)\right|^p\,dy\\
&\leq[W]_{A_p}\int_{Q^*}\left|W^{\frac{1}{p}}(y)\vec{t}_i(y)\right|^p\,dy.
\end{align*}
So in both cases we conclude that
\begin{align*}
\sum_{Q\in\mathscr{D}}\int_{Q^*}[F(x,Q^*)]^p\,dx
\leq[W]_{A_p}\sum_{Q\in\mathscr{D}}
\int_{Q^*}\left|W^{\frac{1}{p}}(y)\vec{t}_i(y)\right|^p\,dy
\sim[W]_{A_p}\left\|\vec{t}_i\right\|_{L^p(W)}^p,
\end{align*}
using the bounded overlap of the expanded cubes $Q^*$ in the last step.
By this and \eqref{161}, we obtain
\begin{equation*}
\left\|\left[\fint_{B(\cdot,2^{l+k_+-i})}
\left|W^{\frac{1}{p}}(\cdot)\vec{t}_i(y)\right|^a\,dy\right]^{\frac{1}{a}}\right\|_{L^p}
\lesssim[W]_{A_p}^{\frac{1}{p}}\left\|\vec{t}_i\right\|_{L^p(W)},
\end{equation*}
and hence \eqref{162} follows from taking the $\ell^q$ norms
with respect to $i\in\mathbb Z$ of both sides in the above inequality.
This finishes the proof of Lemma \ref{ad B1}.
\end{proof}

Next, we establish the boundedness of almost diagonal operators on $\dot b^s_{p,q}(W)$.

\begin{theorem}\label{ad Besov}
Let $s\in\mathbb R$, $p\in(0,\infty)$, $q\in(0,\infty]$, and $W\in A_p$.
Suppose that $B$ is $(D,E,F)$-almost diagonal with parameters
$D>J,$ $E>\frac{n}{2}+s,$ and $F>J-\frac{n}{2}-s,$
where $J:=\frac{n}{\min\{1,p\}}$.
Then $B$ is bounded on $\dot b^s_{p,q}(W)$
and there exists a positive constant $C$,
independent of $W$, such that,
for any $\vec t\in\dot b^s_{p,q}(W)$,
\begin{equation}\label{ad Besov Ap}
\left\|B\vec{t}\right\|_{\dot b^s_{p,q}(W)}
\leq C[W]_{A_p}^{\frac1p}\left\|\vec{t}\right\|_{\dot b^s_{p,q}(W)}.
\end{equation}
\end{theorem}

In this result, we have exceptionally chosen to keep track of
the dependence on the $A_p$ constant
because the argument allows a fairly clean form of this dependence.
This is due to the fact that we only use the $A_p$ condition
directly via Definition \ref{def ap},
rather than via some intermediate auxiliary results.

\begin{proof}[Proof of Theorem \ref{ad Besov}]
Suppose that we already know the result for $s=0$.
If $B$ satisfies the assumptions of the present theorem for a general $s\in\mathbb R$,
then $\widetilde{B}$ from Lemma \ref{s=0 enough}
satisfies the assumptions for $s=0$ by that lemma.
By the assumption that we already know the result for $s=0$,
we then obtain that $\widetilde{B}$ is bounded on $\dot b^0_{p,q}(W)$,
and an application of Lemma \ref{s=0 enough}
gives the boundedness of the original $B$ on $\dot b^s_{p,q}(W)$.
Thus, we may assume that $s=0$ to begin with.

Taking $s=0$, we apply Lemma \ref{ad prelim} with
both $H_j=W^{\frac{1}{p}}$ for any $j\in\mathbb Z$ and $a=\min\{1,p\}$.
Observe that $J=\frac{n}{a}$, and hence the coefficient
in front of the norm in \eqref{ABt} takes the form
\begin{equation}\label{coefB}
2^{-(E-\frac{n}{2})k_-}2^{-(F+\frac{n}{2}-\frac{n}{a})k_+}2^{-(D-\frac{n}{a})l}
=2^{-(E-\frac{n}{2})k_-}2^{-[F-(J-\frac{n}{2})]k_+}2^{-(D-J)l},
\end{equation}
where the coefficients of each of $k_-$, $k_+$, and $l$
are strictly positive by the assumption (with $s=0$).
On the other hand, by Lemma \ref{ad B1}, the norm in \eqref{ABt} is estimated by
\begin{equation}\label{163}
\left\|\left\{\left[\fint_{B(\cdot,2^{l+k_+-i})}\left|W^{\frac{1}{p}}(\cdot)\vec{t}_{i}(y)\right|^a\,dy
\right]^{\frac{1}{a}}\right\}_{i\in\mathbb Z}\right\|_{L\dot A_{p,q}}
\lesssim[W]_{A_p}^{\frac1p}\left\|\vec{t}\right\|_{\dot b^0_{pq}(W)},
\end{equation}
which, together with Remark \ref{ad conv},
further implies that $B\vec t$ is well defined.
Moreover, the right-hand side of \eqref{163} is independent of $(k,l)$,
and hence the series with coefficients \eqref{coefB} converges.
These further imply the claimed bound \eqref{ad Besov Ap}
and hence finish the proof of Theorem \ref{ad Besov}.
\end{proof}

\begin{remark}\label{rem max}
Let all the symbols be the same as in Theorem \ref{ad Besov}.
\begin{enumerate}[\rm(i)]
\item For any $p\in(1,\infty)$,
we could have also dominated the integral averages by,
for any $x\in\mathbb{R}^n$, $i\in\mathbb{Z}$, and $\rho\in(0,\infty)$,
\begin{align*}
\fint_{B(x,\rho)}\left|W^{\frac{1}{p}}(x)\vec{t}_i(y)\right|\,dy
=\fint_{B(x,\rho)}\left|W^{\frac{1}{p}}(x)W^{-\frac{1}{p}}(y)W^{\frac{1}{p}}(y)\vec{t}_i(y)\right|\,dy
\leq\mathcal{M}_{W,p}\left(W^{\frac{1}{p}}\vec{t}_i\right)(x),
\end{align*}
where
\begin{equation*}
\mathcal{M}_{W,p}\left(\vec{f}\right)(x)
:=\sup_{t\in(0,\infty)}\fint_{B(x,t)}
\left|W^{\frac{1}{p}}(x)W^{-\frac{1}{p}}(y)\vec{f}(y)\right|\,dy,
\end{equation*}
and, using the boundedness of this maximal operator
from $L^p(\mathbb R^n,\mathbb C^m)$ to $L^p(\mathbb R^n)$
for any $p\in(1,\infty)$ and $W\in A_p$, we obtain
\begin{align*}
&\left\|\left\{\fint_{B(\cdot,2^{l+k_+-i})}\left|W^{\frac{1}{p}}(\cdot)\vec{t}_i(y)\right|\,dy
\right\}_{i\in\mathbb Z}\right\|_{\ell^qL^p}\\
&\quad\leq\left\|\left\{\mathcal{M}_{W,p}\left(W^{\frac{1}{p}}\vec{t}_i\right)
\right\}_{i\in\mathbb Z}\right\|_{\ell^qL^p}
\lesssim[W]_{A_p}^{\frac{1}{p-1}}\left\|\left\{\left|W^{\frac{1}{p}}\vec{t}_i\right|
\right\}_{i\in\mathbb Z}\right\|_{\ell^qL^p}
=[W]_{A_p}^{\frac{1}{p-1}}\left\|\vec t\right\|_{\dot b_{p,q}^0(W)}.
\end{align*}
However, this would have produced a weaker quantitative estimate
$[W]_{A_p}^{\frac{1}{p-1}}$, which is the optimal $L^p$-norm bound
for $\mathcal{M}_{W,p}$ by \cite[Theorem 1.3]{im19}.
\item In principle, we could try an argument of the same type
also in the case of $\dot f^0_{p,q}(W)$.
The above pointwise bound is still valid
and we would then need to establish the estimate
\begin{equation}\label{FS MW}
\left\|\left\{\mathcal{M}_{W,p}\left(\vec{f}_i\right)\right\}_{i\in\mathbb Z}\right\|_{L^p\ell^q}
\lesssim\left\|\left\{\left|\vec{f}_i\right|\right\}_{i\in\mathbb Z}\right\|_{L^p\ell^q}
\end{equation}
for any $\vec{f}_i=W^{\frac{1}{p}}\vec{t}_i$ with $i\in\mathbb Z$.
The prospective bound \eqref{FS MW} would be a Fefferman--Stein vector-valued inequality
for the matrix-weighted maximal operator $\mathcal{M}_{W,p}$.
Whether or not such an estimate is valid {\color{blue} was an open problem, explicitly posed in \cite[Question 7.3]{dkps}, at the time of our completing this investigation. Accordingly,}
in Section \ref{sec ad TL}, we will explore some substitutes for 
estimate \eqref{FS MW}. {\color{blue} In the meantime, estimate \eqref{FS MW} has been established for $p,q\in(1,\infty)$ and $W\in A_p$ in \cite{ks24}, which would now provide an alternative route to our results in this range of parameters. However, we will still follow our original approach, which has the advantage of covering the full scale of exponents $p,q\in(0,\infty)$ at once. See also Remark~\ref{rem max F}.}
\end{enumerate}
\end{remark}

{\color{red}
\section{Sharpness on Besov Spaces}\label{sec sharp Besov}
}

Curiously, the assumptions on $D,E,F$ in Theorem \ref{ad Besov} are independent of the matrix weight $W\in A_p$. In the following, we show that these assumptions are optimal already in the unweighted case (which is of course a special case of the matrix-weighted case, where the weight is the constant identity matrix).

\begin{lemma}\label{ad Besov sharp}
Let $s\in\mathbb R$, $p\in(0,\infty)$, $q\in(0,\infty]$, and $D,E,F\in\mathbb{R}$.
Suppose that every $(D,E,F)$-almost diagonal matrix $B$ is bounded on $\dot b^s_{p,q}$.
Then $D>J,$ $E>\frac{n}{2}+s,$ and $F>J-\frac{n}{2}-s,$
where $J:=\frac{n}{\min\{1,p\}}$.
\end{lemma}

\begin{proof}
We consider the basic $(D,E,F)$-almost diagonal matrix
$B^{DEF}:=\{b_{Q,R}^{DEF}\}_{Q,R\in\mathscr{Q}}$ from \eqref{bDEF}.

We first estimate the range of $D$.
Let $\vec e\in\mathbb{C}^m$ satisfy $|\vec e|=1$
and $\vec t:=\{\vec t_Q\}_{Q\in\mathscr{Q}}$, where, for any $Q\in\mathscr{Q}$,
\begin{equation}\label{exa1}
\vec t_Q:=\begin{cases}
\vec e&\text{if }Q=Q_{0,\mathbf{0}},\\
\vec{\mathbf{0}}&\text{otherwise}.
\end{cases}
\end{equation}
Then $\|\vec t\|_{\dot b_{p,q}^s}=1$.
By this and the fact that $B$ is bounded on $\dot b^s_{p,q}$, we find that
\begin{align}\label{est1}
\infty
&>\left\|B\vec t\right\|_{\dot b_{p,q}^s}
\geq\left[\sum_{Q\in\mathscr{Q}_0}\left|\left(B\vec t\right)_Q\right|^p\right]^{\frac{1}{p}}\\
&=\left[\sum_{Q\in\mathscr{Q}_0}(b_{Q,Q_{0,\mathbf{0}}})^p\right]^{\frac{1}{p}}
=\left[\sum_{Q\in\mathscr{Q}_0}
\left(1+|x_Q|\right)^{-Dp}\right]^{\frac{1}{p}}\notag
\end{align}
and hence $D>\frac{n}{p}$.
We still need to prove $D>n$.
For any $N\in\mathbb{N}$, let $\vec t^{(N)}:=\{\vec t^{(N)}_Q\}_{Q\in\mathscr{Q}}$,
where, for any $Q\in\mathscr{Q}$,
$$
\vec t^{(N)}_Q
:=\begin{cases}
\vec e&\text{if }Q\in\mathscr{Q}_0\text{ and }|x_Q|<N,\\
\vec{\mathbf{0}}&\text{otherwise}.
\end{cases}
$$
Then, for any $N\in\mathbb{N}$,
$$
\left\|\vec t^{(N)}\right\|_{\dot b_{p,q}^s}
=\left[\sum_{Q\in\mathscr{Q}_0}\left|\vec t^{(N)}_Q\right|^p\right]^{\frac{1}{p}}
=\left[\sum_{k\in\mathbb{Z}^n,\,|k|<N}1\right]^{\frac{1}{p}}
\sim|B(\mathbf{0},N)|^{\frac{1}{p}}
\sim N^{\frac{n}{p}}
$$
and
\begin{align*}
\left\|B\vec t^{(N)}\right\|_{\dot b_{p,q}^s}
&\geq\left[\sum_{Q\in\mathscr{Q}_0}
\left|\left(B\vec t^{(N)}\right)_Q\right|^p\right]^{\frac{1}{p}}
=\left[\sum_{Q\in\mathscr{Q}_0}
\left(\sum_{R\in\mathscr{Q}_0,\,|x_R|<N}b_{Q,R}\right)^p\right]^{\frac{1}{p}}\\
&=\left\{\sum_{k\in\mathbb{Z}^n}
\left[\sum_{l\in\mathbb{Z}^n,\,|l|<N}
\left(\frac{1}{1+|k-l|}\right)^D\right]^p\right\}^{\frac{1}{p}} \\
&\geq\left\{\sum_{k\in\mathbb{Z}^n,\,|k|<\frac{N}{2}}
\left[\sum_{l\in\mathbb{Z}^n,\,|l|<N}
\left(\frac{1}{1+|k-l|}\right)^D\right]^p\right\}^{\frac{1}{p}}\\
&=\left\{\sum_{k\in\mathbb{Z}^n,\,|k|<\frac{N}{2}}
\left[\sum_{l\in\mathbb{Z}^n,\,|k-l|<N}
\left(\frac{1}{1+|l|}\right)^D\right]^p\right\}^{\frac{1}{p}}\\
&\geq\left\{\sum_{k\in\mathbb{Z}^n,\,|k|<\frac{N}{2}}
\left[\sum_{l\in\mathbb{Z}^n,\,|l|<\frac{N}{2}}
\left(\frac{1}{1+|l|}\right)^D\right]^p\right\}^{\frac{1}{p}}
=\sum_{l\in\mathbb{Z}^n,\,|l|<\frac{N}{2}}\left(\frac{1}{1+|l|}\right)^D
\left(\sum_{k\in\mathbb{Z}^n,\,|k|<\frac{N}{2}}1\right)^{\frac{1}{p}}\\
&\sim\sum_{l\in\mathbb{Z}^n,\,|l|<\frac{N}{2}}\left(\frac{1}{1+|l|}\right)^D
\left|B\left(\mathbf{0},\frac{N}{2}\right)\right|^{\frac{1}{p}}
\sim\sum_{l\in\mathbb{Z}^n,\,|l|<\frac{N}{2}}\left(\frac{1}{1+|l|}\right)^DN^{\frac{n}{p}},
\end{align*}
where the positive equivalence constants depend only on $n$.
These, combined with the boundedness of $B$ on $\dot b^s_{p,q}$,
further imply that, for any $N\in\mathbb{N}$,
$\sum_{l\in\mathbb{Z}^n,\,|l|<\frac{N}{2}}(\frac{1}{1+|l|})^D
\lesssim1,$
where the implicit positive constant depends only on $n$.
Letting $N\to\infty$, we obtain
\begin{equation}\label{est2}
\sum_{l\in\mathbb{Z}^n}\left(\frac{1}{1+|l|}\right)^D<\infty
\end{equation}
and hence $D>n$. This finishes the estimation of $D$.

Now, we estimate the ranges of both $E$ and $F$ by considering
the following two cases on $q$.

\emph{Case 1)} $q\in(0,\infty)$.
In this case, letting $\vec t$ be the same as in \eqref{exa1}, then
\begin{align}\label{est3}
\infty
&>\left\|B\vec t\right\|_{\dot b_{p,q}^s}^q
=\sum_{j\in\mathbb{Z}}2^{jsq}
\left[\sum_{Q\in\mathscr{Q}_j}|Q|^{1-\frac{p}{2}}
\left|\left(B\vec t\right)_Q\right|^p\right]^{\frac{q}{p}}\\
&=\sum_{j\in\mathbb{Z}}2^{jsq}
\left[\sum_{Q\in\mathscr{Q}_j}|Q|^{1-\frac{p}{2}}
(b_{Q,Q_{0,\mathbf{0}}})^p\right]^{\frac{q}{p}}\notag\\
&=\sum_{j=0}^\infty2^{jsq}
\left\{\sum_{Q\in\mathscr{Q}_j}|Q|^{1-\frac{p}{2}}
(1+|x_Q|)^{-Dp}[\ell(Q)]^{Ep}\right\}^{\frac{q}{p}}\notag\\
&\quad+\sum_{j=-\infty}^{-1}2^{jsq}
\left\{\sum_{Q\in\mathscr{Q}_j}|Q|^{1-\frac{p}{2}}
\left(1+2^j|x_Q|\right)^{-Dp}[\ell(Q)]^{-Fp}\right\}^{\frac{q}{p}}\notag\\
&=\sum_{j=0}^\infty\left\{2^{-j[(E-\frac{n}{2}-s)p+n]}
\sum_{k\in\mathbb{Z}^n}\left(1+2^{-j}|k|\right)^{-Dp}\right\}^{\frac{q}{p}}\notag\\
&\quad+\sum_{j=-\infty}^{-1}\left[2^{j(F+\frac{n}{2}+s-\frac{n}{p})p}
\sum_{k\in\mathbb{Z}^n}(1+|k|)^{-Dp}\right]^{\frac{q}{p}}.\notag
\end{align}
This, together with both $D>J$ and \eqref{33z}, further implies that
$$
\infty
>\sum_{j=0}^\infty2^{-j(E-\frac{n}{2}-s)q}
+\sum_{j=-\infty}^{-1}2^{j(F+\frac{n}{2}+s-\frac{n}{p})q}
$$
and hence
$E>\frac{n}{2}+s$ and $F>\frac{n}{p}-\frac{n}{2}-s.$

Next, we show $F>\frac{n}{2}-s$.
For any $N\in\mathbb{N}$, let $\vec t^{(N)}:=\{\vec t^{(N)}_Q\}_{Q\in\mathscr{Q}}$,
where, for any $Q\in\mathscr{Q}$,
\begin{equation}\label{exa2}
\vec t^{(N)}_{Q} :=\begin{cases}
|Q|^{\frac{s}{n}+\frac12}\vec e&\text{if }Q\in\mathscr{Q}_N\text{ with }|x_Q|<1,\\
\vec{\mathbf{0}}&\text{otherwise}.
\end{cases}
\end{equation}
Then, by the definitions of both $\vec t^{(N)}$
and $\|\cdot\|_{\dot b_{p,q}^s}$, we obtain
\begin{equation}\label{289x}
\left\|\vec t^{(N)}\right\|_{\dot b_{p,q}^s}
=\left[\sum_{Q\in\mathscr{Q}_N,\,|x_Q|<1}|Q|\right]^{\frac{1}{p}}
=\left(\sum_{k\in\mathbb{Z}^n,\,|k|<2^N}2^{-Nn}\right)^{\frac{1}{p}}
\sim1,
\end{equation}
where the positive equivalence constants depend only on $n$ and $p$.
On the other hand, from the definitions of both $B\vec t^{(N)}$
and $\|\cdot\|_{\dot b_{p,q}^s}$, it follows that
\begin{align}\label{5.28x}
\left\|B\vec t^{(N)}\right\|_{\dot b_{p,q}^s}
&\geq\left[\sum_{j=0}^{N-1}2^{jsq}
\left(\sum_{Q\in\mathscr{Q}_j}|Q|^{1-\frac{p}{2}}
\left|\left(B\vec t^{(N)}\right)_Q\right|^p
\right)^{\frac{q}{p}}\right]^{\frac{1}{q}}\\
&=\left[\sum_{j=0}^{N-1}\left(\sum_{Q\in\mathscr{Q}_j}|Q|
\left|\sum_{R\in\mathscr{Q}_N,\,|x_R|<1}
\frac{2^{(j-N)(F+s+\frac{n}{2})}}{(1+2^j|x_Q-x_R|)^D}
\right|^p\right)^{\frac{q}{p}}\right]^{\frac{1}{q}}\notag\\
&\geq\left[\sum_{j=0}^{N-1}\left(\sum_{k\in\mathbb{Z}^n,\,|k|<2^{j-1}}2^{-jn}
\left|\sum_{l\in\mathbb{Z}^n,\,|l|<2^N}
\frac{2^{(j-N)(F+s+\frac{n}{2})}}{(1+|k-2^{j-N}l|)^D}
\right|^p\right)^{\frac{q}{p}}\right]^{\frac{1}{q}},\notag
\end{align}
where
\begin{align}\label{5.28y}
&\sum_{l\in\mathbb{Z}^n,\,|l|<2^N}
\frac{1}{(1+|k-2^{j-N}l|)^D}\\
&\quad=\sum_{l\in\mathbb{Z}^n,\,|l+2^{N-j}k|<2^N}\frac{1}{(1+2^{j-N}|l|)^D}\notag\\
&\quad\geq\sum_{l\in\mathbb{Z}^n,\,|l|<2^{N-1}}\frac{1}{(1+2^{j-N}|l|)^D}
\geq\sum_{l\in\mathbb{Z}^n,\,|l|<2^{N-1-j}}\frac{1}{(1+2^{j-N}|l|)^D}\notag\\
&\quad\sim\sum_{l\in\mathbb{Z}^n,\,|l|<2^{N-1-j}}1
\sim \left|B\left(\mathbf{0},2^{N-1-j}\right)\right|
\sim2^{(N-j)n}.\notag
\end{align}
Therefore, using \eqref{5.28x} and \eqref{5.28y}, we find that
\begin{align}\label{est4}
\left\|B\vec t^{(N)}\right\|_{\dot b_{p,q}^s}
&\gtrsim\left[\sum_{j=0}^{N-1}\left(\sum_{k\in\mathbb{Z}^n,\,|k|<2^{j-1}}
2^{-jn}2^{(j-N)(F+s-\frac{n}{2})p}\right)^{\frac{q}{p}}\right]^{\frac{1}{q}}
\sim\left[\sum_{j=0}^{N-1}2^{(j-N)(F+s-\frac{n}{2})q}\right]^{\frac{1}{q}}.
\end{align}
This, combined with \eqref{289x} and
the assumption that $B$ is bounded on $\dot b^s_{p,q}$,
further implies that, for any $N\in\mathbb{N}$,
$$
\sum_{j=-N}^{-1}2^{j(F+s-\frac{n}{2})q}
=\sum_{j=0}^{N-1}2^{(j-N)(F+s-\frac{n}{2})q}
\lesssim\left\|\vec t^{(N)}\right\|_{\dot b_{p,q}^s}^q
\sim1.
$$
Letting $N\to\infty$, we obtain
$\sum_{j=-\infty}^{-1}2^{j(F+s-\frac{n}{2})q}\lesssim1$
and hence $F>\frac{n}{2}-s$. This finishes the estimations of both $E$ and $F$ in this case.

\emph{Case 2)} $q=\infty$.
In this case, we first prove $E>s+\frac{n}{2}$.
Let $\vec t:=\{\vec t_Q\}_{Q\in\mathscr{Q}}$, where, for each $Q\in\mathscr{Q}$,
\begin{equation}\label{exa3}
\vec t_{Q}:=\begin{cases}
|Q|^{\frac{s}{n}+\frac12}\vec e&\text{if }\ell(Q)\leq1\text{ and }|x_Q|<1,\\
\vec{\mathbf{0}}&\text{otherwise}.
\end{cases}
\end{equation}
Then, by the definitions of both $\vec t$
and $\|\cdot\|_{\dot b_{p,\infty}^s}$, we have
$$
\left\|\vec t\right\|_{\dot b_{p,\infty}^s}
=\sup_{j\in\mathbb{Z}_+}\left(\sum_{Q\in\mathscr{Q}_j,\,|x_Q|<1}|Q|\right)^{\frac{1}{p}}
=\sup_{j\in\mathbb{Z}_+}\left(\sum_{k\in\mathbb{Z}^n,\,|k|<2^j}2^{-jn}\right)^{\frac{1}{p}}
\sim1.
$$
This, together with the assumption that $B$ is bounded on $\dot b^s_{p,q} $,
further implies that
\begin{align}\label{5.30x}
\infty&>\left\|B\vec t\right\|_{\dot b_{p,\infty}^s}
\geq\sup_{j\in\mathbb{Z}_+}2^{js}
\left[\sum_{Q\in\mathscr{Q}_j}|Q|^{1-\frac{p}{2}}
\left|\left(B\vec t\right)_Q\right|^p\right]^{\frac{1}{p}}\\
&>\sup_{j\in\mathbb{Z}_+}2^{js}
\left\{\sum_{Q\in\mathscr{Q}_j}|Q|^{1-\frac{p}{2}}
\left[\sum_{i=0}^j\sum_{R\in\mathscr{Q}_i,\,|x_R|<1}
\left(1+2^i|x_Q-x_R|\right)^{-D}
2^{(i-j)E}2^{-i(s+\frac{n}{2})}\right]^p\right\}^{\frac{1}{p}}\notag\\
&=\sup_{j\in\mathbb{Z}_+}
\left\{\sum_{Q\in\mathscr{Q}_j}|Q|
\left[\sum_{i=0}^j\sum_{R\in\mathscr{Q}_i,\,|x_R|<1}
\left(1+2^i|x_Q-x_R|\right)^{-D}
2^{(i-j)(E-s-\frac{n}{2})}\right]^p\right\}^{\frac{1}{p}}.\notag
\end{align}
Notice that, for any $i\in\mathbb Z_+$ and $x\in B(\mathbf{0},1)$,
by some geometrical observations, we find that
there exists $R\in\mathscr{Q}_i$ with $|x_R|<1$
such that $|x-x_R|<2^{-i}\sqrt{n}$, and hence
\begin{equation}\label{5.30y}
\sum_{R\in\mathscr{Q}_i,\,|x_R|<1}
\left(1+2^i|x-x_R|\right)^{-D}
\gtrsim1.
\end{equation}
Therefore, from \eqref{5.30x} and \eqref{5.30y}, it follows that
\begin{align}\label{est5}
\infty
&>\left\|B\vec t\right\|_{\dot b_{p,\infty}^s}
\gtrsim\sup_{j\in\mathbb{Z}_+}
\left\{\sum_{Q\in\mathscr{Q}_j,\,|x_Q|<1}|Q|
\left[\sum_{i=0}^j 2^{(i-j)(E-s-\frac{n}{2})}\right]^p\right\}^{\frac{1}{p}}\\
&\sim\sup_{j\in\mathbb{Z}_+}\sum_{i=0}^j 2^{(i-j)(E-s-\frac{n}{2})}
=\sup_{j\in\mathbb{Z}_+}\sum_{i=-j}^0 2^{i(E-s-\frac{n}{2})}
=\sum_{i=-\infty}^0 2^{i(E-s-\frac{n}{2})},\notag
\end{align}
and hence $E>s+\frac{n}{2}$.

Now, we show $F>\frac{n}{2}-s$.
Let $\vec t:=\{\vec t_Q\}_{Q\in\mathscr{Q}}$, where, for each $Q\in\mathscr{Q}$,
\begin{equation}\label{exa4}
\vec t_{Q}:=\begin{cases}\displaystyle
|Q|^{\frac{s}{n}+\frac12}\frac{1}{(1+|x_Q|)^{\frac{n+1}{p}}}\vec e
&\text{if }j\in\mathbb{Z}_+\text{ and }Q\in\mathscr{Q}_j,\\
\vec{\mathbf{0}}&\text{otherwise}.
\end{cases}
\end{equation}
Then, by \eqref{33z}, we obtain
$$
\left\|\vec t\right\|_{\dot b_{p,\infty}^s}
=\sup_{j\in\mathbb{Z}_+}\left[\sum_{Q\in\mathscr{Q}_j}|Q|\frac{1}{(1+|x_Q|)^{n+1}}\right]^{\frac{1}{p}}
\sim1.
$$
From the assumption that $(B\vec t)_{Q_{0,\mathbf{0}}}$ makes sense,
$D>J$, and \eqref{33z} again, we infer that
\begin{align*}
\infty>\sum_{R\in\mathscr{Q}}\left|b_{Q_{0,\mathbf{0}},R}\vec t_R\right|
=\sum_{i=0}^\infty2^{-i(F+s+\frac{n}{2})}
\sum_{R\in\mathscr{Q}_i}\frac{1}{(1+|x_R|)^{D+\frac{n+1}{p}}}
\sim\sum_{i=0}^\infty2^{-i(F+s-\frac{n}{2})}
\end{align*}
and hence $F>\frac{n}{2}-s$.

Finally, we prove $F>\frac{n}{p}-\frac{n}{2}-s$.
Let $\vec t:=\{\vec t_Q\}_{Q\in\mathscr{Q}}$, where, for any $Q\in\mathscr{Q}$,
$$
\vec t_{Q}:=\begin{cases}
|Q|^{\frac{s}{n}+\frac12-\frac{1}{p}}\vec e&\text{if }x_Q=\mathbf{0},\\
\vec{\mathbf{0}}&\text{otherwise}.
\end{cases}
$$
Then $\|\vec t\|_{\dot b_{p,\infty}^s}=1$.
By the assumption that $(B\vec t)_{Q_{0,\mathbf{0}}}$ makes sense again, we conclude that
$$
\infty
>\sum_{R\in\mathscr{Q}}\left|b_{Q_{0,\mathbf{0}},R}\vec t_R\right|
>\sum_{i=0}^\infty2^{-i(F+s+\frac{n}{2}-\frac{n}{p})},
$$
and hence $F>\frac{n}{p}-\frac{n}{2}-s$.
This finishes the estimations of both $E$ and $F$ in this case
and hence the proof of Lemma \ref{ad Besov sharp}.
\end{proof}

Next, we are able to characterize the sharp almost diagonal conditions of $\dot b_{\infty,q}^s$; note that both Theorem \ref{ad Besov} and the previous Lemma \ref{ad Besov sharp} dealt with $\dot b_{p,q}^s$ with $p\in(0,\infty)$.

\begin{lemma}\label{ad Besov sharp infty}
Let $s\in\mathbb R$, $q\in(0,\infty]$, and $D,E,F\in\mathbb{R}$.
Suppose that every $(D,E,F)$-almost diagonal matrix $B$ is bounded on $\dot b^s_{\infty,q}$.
Then $D>n,$ $E>\frac{n}{2}+s,$ and $F>\frac{n}{2}-s.$
\end{lemma}

\begin{proof}
We consider the $(D,E,F)$-almost diagonal matrix
$B^{DEF}=\{b^{DEF}_{Q,R}\}_{Q,R\in\mathscr{Q}}$ from \eqref{bDEF}.
We first estimate the range of $D$.
Let $t:=\{t_Q\}_{Q\in\mathscr{Q}}$, where, for any $Q\in\mathscr{Q}$,
\begin{equation}\label{exa5}
t_Q:=\begin{cases}
1&\text{if }\ell(Q)=1,\\
0&\text{otherwise}.
\end{cases}
\end{equation}
Then, from the definitions of both $t$ and $\|\cdot\|_{\dot b^s_{\infty,q}}$,
we deduce that, when $q\in(0,\infty)$,
$$
\|t\|_{\dot b^s_{\infty,q}}
=\left[\sum_{j\in\mathbb{Z}}2^{j(s+\frac{n}{2})q}
\sup_{Q\in\mathscr{Q}_j}|t_Q|^q\right]^{\frac{1}{q}}
=1
$$
and
$$
\|t\|_{\dot b^s_{\infty,\infty}}
=\sup_{Q\in\mathscr{Q}}|Q|^{-(\frac{s}{n}+\frac12)}|t_Q|
=1.
$$
Thus, by the boundedness of $B$ on $\dot b^s_{\infty,q}$,
we find that $(Bt)_{Q_{0,\mathbf{0}}}$ makes sense and
$$
\infty
>\sum_{R\in\mathscr{Q}}b_{Q_{0,\mathbf{0}},R}t_R
=\sum_{k\in\mathbb{Z}^n}\frac{1}{(1+|k|)^D}.
$$
This further implies that $D>n$.

Now, we estimate the ranges of both $E$ and $F$ by considering the following two cases on $q$.

\emph{Case 1)} $q\in(0,\infty)$.
In this case, let $t$ be the same as in \eqref{exa5}.
Then, from $D>n$ and \eqref{33z},
we infer that, for any $j\in\mathbb{Z}_+$ and $Q\in\mathscr{Q}_j$,
$$
(Bt)_Q
=\sum_{R\in\mathscr{Q}}b_{Q,R}t_R
=\sum_{k\in\mathbb{Z}^n}\frac{1}{(1+|x_Q-k|)^D}2^{-jE}
\sim2^{-jE}
$$
and, for any $j\in\mathbb{Z}$ with $j\leq-1$ and $Q\in\mathscr{Q}_j$,
$$
(Bt)_Q
=\sum_{R\in\mathscr{Q}}b_{Q,R}t_R
=\sum_{k\in\mathbb{Z}^n}\frac{1}{(1+2^j|x_Q-k|)^D}2^{jF}
\sim2^{j(F-n)}.
$$
These further imply that
\begin{align*}
\infty
&>\|Bt\|_{\dot b^s_{\infty,q}}^q
\sim\sum_{j=0}^{\infty}2^{-j(E-s-\frac{n}{2})q}
+\sum_{j=-\infty}^{-1}2^{j(F+s-\frac{n}{2})q},
\end{align*}
and hence $E>\frac{n}{2}+s$ and $F>\frac{n}{2}-s$.
This finishes the estimations of both $E$ and $F$ in this case.

\emph{Case 2)} $q=\infty$.
In this case, Let $t:=\{t_Q\}_{Q\in\mathscr{Q}}$, where, for any $Q\in\mathscr{Q}$,
$t_{Q}:=|Q|^{\frac{s}{n}+\frac12}.$
Then, using the definitions of both $t$ and $\|\cdot\|_{\dot b^s_{\infty,\infty}}$, we obtain
$\|t\|_{\dot b^s_{\infty,\infty}}
=\sup_{Q\in\mathscr{Q}}|Q|^{-(\frac{s}{n}+\frac12)}|t_Q|
=1.$
Thus, by the boundedness of $B$ on $\dot b^s_{\infty,q}$,
we conclude that $(Bt)_{Q_{0,\mathbf{0}}}$ makes sense.
This, combined with $D>n$ and \eqref{33z}, further implies that
\begin{align*}
\infty
&>\sum_{R\in\mathscr{Q}}b_{Q_{0,\mathbf{0}},R}t_R\\
&=\sum_{i=-\infty}^0\sum_{k\in\mathbb{Z}^n}
\frac{1}{(1+|k|)^D}2^{i(E-s-\frac{n}{2})}
+\sum_{i=1}^{\infty}\sum_{k\in\mathbb{Z}^n}
\frac{1}{(1+2^{-i}|k|)^D}2^{-i(F+s+\frac{n}{2})}\\
&\sim\sum_{i=-\infty}^02^{i(E-s-\frac{n}{2})}
+\sum_{i=1}^{\infty}2^{-i(F+s-\frac{n}{2})},
\end{align*}
and hence $E>\frac{n}{2}+s$ and $F>\frac{n}{2}-s$.
This finishes the estimations of both $E$ and $F$ in this case
and hence the proof of Lemma \ref{ad Besov sharp infty}.
\end{proof}

From Theorem \ref{ad FJ-YY}(i) and Lemmas \ref{ad Besov sharp} and \ref{ad Besov sharp infty},
we immediately deduce the following conclusion concerning almost diagonal operators on the full scale of unweighted Besov sequence spaces; we omit the details.

\begin{theorem}
Let $s\in\mathbb R$, $p,q\in(0,\infty]$, and $D,E,F\in\mathbb{R}$.
Then every $(D,E,F)$-almost diagonal matrix $B$ is bounded on $\dot b^s_{p,q}$
if and only if
$D>J,$ $E>\frac{n}{2}+s,$ and $F>J-\frac{n}{2}-s,$
where $J:=\frac{n}{\min\{1,p\}}$.
\end{theorem}

\section{Boundedness on Triebel--Lizorkin Sequence Spaces}\label{sec ad TL}

We then turn to a more complicated case, namely Triebel--Lizorkin spaces.
As in Section \ref{sec ad Besov} in the case of Besov spaces,
the goal of this section is to show the special case of Theorem \ref{ad BF}
for the Triebel--Lizorkin sequence space $\dot f_{p,q}^s(W)=\dot f_{p,q}^{s,0}(W)$,
while at the same time developing tools that will allow us to
finish the proof of Theorem \ref{ad BF} in the following section.

We need the following well-known fact about shifted systems of dyadic cubes of $\mathbb R^n$
(see, for instance, \cite[Lemma 3.2.26]{hnvw}).

\begin{lemma}\label{shifted}
On $\mathbb R^n$, there exist $3^n$ shifted systems of dyadic cubes
$\mathscr{Q}^\alpha$, $\alpha\in\{1,\ldots,3^n\}$,
such that, for any cube $Q\subset\mathbb R^n$,
there exist $\alpha\in\{1,\ldots,3^n\}$ and a cube $S\in\mathscr{Q}^\alpha$
such that $Q\subset S$ and $\ell(S)\in(\frac32\ell(Q),3\ell(Q)]$.
\end{lemma}

\begin{remark}
For each $\gamma\in\{0,\frac13,\frac23\}^n$, let
$\mathscr{Q}^{\gamma}:=\bigcup_{j\in\mathbb{Z}}\mathscr{Q}^{\gamma}_j,$
where, for each $j\in\mathbb Z$,
$$
\mathscr{Q}^{\gamma}_j
:=\left\{2^{-j}\left([0,1)^n+k+(-1)^j\gamma\right):\ k\in\mathbb{Z}^n\right\}.
$$
Then $\{\mathscr{Q}^{\alpha}\}_{\alpha=1}^{3^n}$ in Lemma \ref{shifted}
can be chosen to be a rearrangement of $\{\mathscr{Q}^{\gamma}\}_{\gamma\in\{0,\frac13,\frac23\}^n}$.
\end{remark}

For any $j\in\mathbb{Z}$ and any nonnegative measurable function
$f$ on $\mathbb R^n$ or any $f\in L^1_{\mathop\mathrm{loc}}$, let
\begin{equation}\label{Ej}
E_j (f):=\sum_{Q\in\mathscr{Q}_j}
\left[\fint_Q f(x)\,dx\right]\mathbf{1}_Q.
\end{equation}

The following lemma is just \cite[Corollary 3.8]{fr21}.

\begin{lemma}\label{46x}
Let $p\in(0,\infty)$, $q\in(0,\infty]$, $W\in A_p$,
and $\{A_Q\}_{Q\in\mathscr{Q}}$ be a sequence of
reducing operators of order $p$ for $W$.
For any $j\in\mathbb{Z}$, let
$\gamma_j:=\sum_{Q\in\mathscr{Q}_j}
\|W^{\frac{1}{p}}A_Q^{-1}\|\mathbf{1}_Q.$
Then there exists a positive constant $C$ such that,
for any sequence $\{f_j\}_{j\in\mathbb{Z}}$ of nonnegative measurable functions
on $\mathbb R^n$ or for any $\{f_j\}_{j\in\mathbb{Z}}\subset L^1_{\mathop\mathrm{loc}}$,
$$
\left\|\left\{\gamma_jE_j\left(f_j\right)\right\}_{j\in\mathbb Z}\right\|_{L^p\ell^q}
\leq C\left\|\left\{E_j\left(f_j\right)\right\}_{j\in\mathbb Z}\right\|_{L^p\ell^q},
$$
where $E_j$ for any $j\in\mathbb Z$ is the same as in \eqref{Ej}.
\end{lemma}

We also need the following \cite[Theorem 3.7(ii)]{fr21}
which the authors of \cite{fr21} attributed to F. Nazarov.

\begin{lemma}\label{Nazarov}
Let $p\in(1,\infty)$ and $q\in[1,\infty]$,
and suppose that $\{\gamma_j\}_{j\in\mathbb Z}$
is a sequence of non-negative measurable functions on $\mathbb R^n$ satisfying that,
for some $\delta\in(0,\infty)$,
\begin{equation*}
\sup_{Q\in\mathscr{Q}}\fint_Q
\sup_{j\in\mathbb{Z},\,2^{-j}\leq\ell(Q)}
\left[\gamma_j(x)\right]^{p(1+\delta)}\,dx<\infty.
\end{equation*}
Then there exists a positive constant $C$ such that,
for any sequence $\{f_j\}_{j\in\mathbb Z}$ of nonnegative measurable
functions on $\mathbb R^n$ or for any $\{f_j\}_{j\in\mathbb{Z}}
\subset L^1_{\mathop\mathrm{loc}}$,
\begin{equation*}
\left\|\left\{\gamma_jE_j\left(f_j\right)\right\}_{j\in\mathbb Z}\right\|_{L^p(\ell^q)}
\leq C\|\{f_j\}_{j\in\mathbb Z}\|_{L^p(\ell^q)},
\end{equation*}
where $E_j(f_j)$ for each $j\in\mathbb Z$ is the same as in \eqref{Ej}.
\end{lemma}

{\color{red}We are now ready to estimate the right-hand side of the Key Lemma \ref{ad prelim} in Triebel--Lizorkin sequence spaces.}

\begin{lemma}\label{ad F1}
Let $p\in(0,\infty)$, $q\in(0,\infty]$, $W\in A_p$, and $a\in[0,p\wedge q\wedge 1)$.
Then there exists a positive constant $C$ such that,
for any $k,l\in\mathbb Z$ and $\vec{t}\in\dot f^0_{p,q}(W)$,
\begin{equation*}
\left\|\left\{\left[\fint_{B(\cdot,2^{l+k_+-i})}
\left|W^{\frac{1}{p}}(\cdot)\vec{t}_i(y)\right|^a\,dy\right]^{\frac{1}{a}}
\right\}_{i\in\mathbb Z}\right\|_{L^p\ell^q}
\leq C\left\|\vec{t}\right\|_{\dot f^0_{p,q}(W)},
\end{equation*}
where $\vec{t}_i$ for any $i\in\mathbb Z$ is the same as in \eqref{vec tj}.
\end{lemma}

\begin{proof}
For any $k,l,i\in\mathbb Z$ and $x\in\mathbb{R}^n$, consider the ball $B(x,2^{l+k_+-i})$.
It is clearly contained in a cube of edge length $2^{1+l+k_+-i}$
and hence, by Lemma \ref{shifted},
in some shifted dyadic cube $S\in\mathscr{Q}^\alpha$
of edge length $\ell(S)=2^{2+l+k_+-i}$; thus,
\begin{equation}\label{12}
S\in\mathscr{Q}^\alpha_{i-\ell-k_+-2}.
\end{equation}
For any $i\in\mathbb{Z}$, any matrix $M\in M_m(\mathbb{C})$,
and any measurable set $E\subset\mathbb{R}^n$ with $|E|\in(0,\infty)$, we define
$$
F_i(M,E):=\left[\fint_E\left|M\vec{t}_i(y)\right|^a\,dy\right]^{\frac{1}{a}}.
$$
From \eqref{12}, we infer that, for any $k,l\in\mathbb Z$ and $\vec t\in\dot f_{p,q}^0(W)$,
\begin{align*}
&\left\|\left\{F_i\left(W^{\frac{1}{p}}(\cdot),B\left(\cdot,2^{l+k_+-i}\right)\right)
\right\}_{i\in\mathbb Z}\right\|_{L^p\ell^q}\\
&\quad\lesssim\left\|\left\{
\max_{\alpha\in\{1,\ldots,3^n\}}
\sup_{S\in\mathscr{Q}^\alpha_{i-l-k_+-2}}\mathbf{1}_S(\cdot)
F_i\left(W^{\frac{1}{p}}(\cdot),S\right)
\right\}_{i\in\mathbb Z}\right\|_{L^p\ell^q}\\
&\quad\lesssim\max_{\alpha\in\{1,\ldots,3^n\}}
\left\|\left\{\sum_{S\in\mathscr{Q}^\alpha_{i-l-k_+-2}}\mathbf{1}_S(\cdot)
F_i\left(W^{\frac{1}{p}}(\cdot),S\right)
\right\}_{i\in\mathbb Z}\right\|_{L^{p}\ell^{q}}\\
&\quad\leq\max_{\alpha\in\{1, \ldots, 3^n\}}
\left\|\left\{\sum_{S\in\mathscr{Q}^\alpha_{j}}
\mathbf{1}_S(\cdot)\left\|W^{\frac{1}{p}}(\cdot)A_S^{-1}\right\|F_{j+l+k_++2}(A_S,S)
\right\}_{j\in\mathbb Z}\right\|_{L^p\ell^q},
\end{align*}
where in the last step we made the change of variables, $j:=i-l-k_+-2$,
and extracted a matrix norm after multiplying by
the reducing operator $A_S$ of order $p$ for $W$
and its inverse inside the integral.

Let $\alpha\in\{1, \ldots, 3^n\}$ be fixed.
We observe that the quantity on the right-hand side of the previous computation
has exactly the form considered in Lemma \ref{46x}
with the shifted dyadic system $\mathscr{Q}^\alpha$ in place of $\mathscr{Q}$
and with $E_j(f_j)$ replaced by
$\sum_{S\in\mathscr{Q}^\alpha_{j}}\mathbf{1}_S(\cdot)F_{j+l+k_++2}(A_S,S)$
for any $j\in\mathbb{Z}$.
The mentioned Lemma \ref{46x} thus further implies that
\begin{align*}
&\left\|\left\{\sum_{S\in\mathscr{Q}^\alpha_{j}}
\mathbf{1}_S(\cdot)\left\|W^{\frac{1}{p}}(\cdot)A_S^{-1}\right\|F_{j+l+k_++2}(A_S,S)
\right\}_{j\in\mathbb Z}\right\|_{L^p\ell^q}\\
&\quad\lesssim\left\|\left\{
\sum_{S\in\mathscr{Q}^\alpha_{j}}\mathbf{1}_S(\cdot)F_{j+l+k_++2}(A_S,S)
\right\}_{j\in\mathbb Z}\right\|_{L^p\ell^q}\\
&\quad=\left\|\left\{
\sum_{S\in\mathscr{Q}^\alpha_{j}}\mathbf{1}_S(\cdot)
\fint_{S}\left|A_S\vec{t}_{j+l+k_++2}(y)\right|^a\,dy
\right\}_{j\in\mathbb Z}\right\|_{L^{\frac{p}{a}}\ell^{\frac{q}{a}}}^{\frac{1}{a}}\\
&\quad=\left(\sup\left\{\sum_{j\in\mathbb Z}
\int_{\mathbb R^n}g_j(x)\left[\sum_{S\in\mathscr{Q}^\alpha_{j}}\mathbf{1}_S(x)
\fint_{S}\left|A_S\vec{t}_{j+l+k_+ +2}(y)\right|^a\,dy\right]\,dx\right.\right.\\
&\qquad\qquad\qquad:\ \left.\|\{g_j\}\|_{L^{(\frac{p}{a})'}\ell^{(\frac{q}{a})'}}\leq1 \Bigg\}\right)^{\frac{1}{a}},
\end{align*}
where in the last step we used the duality
$(L^{\frac{p}{a}}\ell^{\frac{q}{a}})'=L^{(\frac{p}{a})'}\ell^{(\frac{q}{a})'}$
(see, for instance, \cite[p.\,177, Proposition (i)]{t83}),
recalling the assumption that $a< p\wedge q \wedge 1 $.

The quantity inside the supremum above can be written as
\begin{align*}
&\sum_{j\in\mathbb Z}\int_{\mathbb R^n}g_j(x)
\left[\sum_{S\in\mathscr{Q}^\alpha_{j}}\mathbf{1}_S(x)
\fint_{S}\left|A_S\vec{t}_{j+l+k_++2}(y)\right|^a\,dy\right]\,dx \\
&\quad\leq\sum_{j\in\mathbb Z}\sum_{S\in\mathscr{Q}^\alpha_{j}}
\int_S\left\|A_S W^{-\frac{1}{p}}(y)\right\|^a
\left|W^{\frac{1}{p}}(y)\vec{t}_{j+l+k_++2}(y)\right|^a\,dy
\fint_Sg_j(x)\,dx\\
&\quad=\sum_{j\in\mathbb Z}\int_{\mathbb R^n}
\left|W^{\frac{1}{p}}(y)\vec{t}_{j+l+k_++2}(y)\right|^a
\left[\sum_{S\in\mathscr{Q}^\alpha_j}\mathbf{1}_S(y)
\left\|A_SW^{-\frac{1}{p}}(y)\right\|^a\fint_Sg_j(x)\,dx\right]\,dy\\
&\quad\leq\left\|\left\{
\left|W^{\frac{1}{p}}(\cdot)\vec{t}_{j+l+k_++2}(\cdot)\right|^a
\right\}_{j\in\mathbb Z}\right\|_{L^{\frac{p}{a}}\ell^{\frac{q}{a}}}\\
&\qquad\times\left\|\left\{
\sum_{S\in\mathscr{Q}^\alpha_j}\mathbf{1}_S(\cdot)
\left\|A_SW^{-\frac{1}{p}}(\cdot)\right\|^a\fint_Sg_j(x)\,dx
\right\}_{j\in\mathbb Z}\right\|_{L^{(\frac{p}{a})'}\ell^{(\frac{q}{a})'}}\\
&\quad=:\mathrm{I}\times\mathrm{II},
\end{align*}
where in the last step we used twice H\"older's inequality.
By reversing the earlier change of variables back to $i:=j+l+k_+ +2$, we find that
\begin{equation*}
\mathrm{I}
=\left\|\left\{\left|W^{\frac{1}{p}}(\cdot)\vec{t}_i(\cdot)\right|
\right\}_{i\in\mathbb Z}\right\|_{L^p\ell^q}^a
=\left\|\vec{t}\right\|_{\dot f^0_{p,q}(W)}^a.
\end{equation*}
The term  $\mathrm{II}$ will be estimated with the help of Lemma \ref{Nazarov}.
In the notation of that lemma, we can write
\begin{align*}
\mathrm{II}
=\left\|\left\{\gamma_jE_j\left(g_j\right)
\right\}_{j\in\mathbb Z}\right\|_{L^{(\frac{p}{a})'}\ell^{(\frac{q}{a})'}},
\end{align*}
where $E_j$ for each $j\in\mathbb Z$ is the same as in \eqref{Ej} and, for each $j\in\mathbb{Z}$,
$\gamma_j:=\sum_{S\in\mathscr{Q}^\alpha_j}
\mathbf{1}_S\|A_S W^{-\frac{1}{p}}\|^a.$
We intend to use Lemma \ref{Nazarov}
with $((\frac{p}{a})',(\frac{q}{a})')$ in place of $(p,q)$.
Since $a<p\wedge q$, it follows that we have the required conditions
$(\frac{p}{a})'\in(1,\infty)$ and $(\frac{q}{a})'\in[1,\infty]$.
By \cite[Lemma 2.19(i) and (2.11)]{bhyy1},
we conclude that, for any $Q\in\mathscr{Q}^\alpha$,
\begin{equation}\label{8xxx}
\fint_Q\sup_{j\in\mathbb{Z},\,2^{-j}\leq\ell(Q)}
\left[\gamma_j(y)\right]^{u}\,dy
=\fint_Q\sup_{R\in\mathscr{Q}^\alpha,\,y\in R\subset Q}
\left\|A_RW^{-\frac{1}{p}}(y)\right\|^{au}\,dy
\lesssim1,
\end{equation}
where we can take any $u\in(0,\infty)$ if $p\in(0,1]$
and some $u$ such that $au>p'$ if $p\in(1,\infty)$.

We claim that we can always take $u>(\frac{p}{a})'$.
If $p\in(0,1]$, this is obvious because we can take $u$ as large as we like.
If $p\in(1,\infty)$, the assumption that $a<1$ further implies that $\frac{p}{a}-1>p-1>0$, and hence
$u>\frac{p'}{a}
=\frac{p/a}{p-1}
> \frac{p/a}{p/a-1}
=(\frac{p}{a})'$
also in this case. Thus, \eqref{8xxx} is a sufficient condition
for the pointwise multipliers $\gamma_j$ for any $j\in\mathbb Z$
to apply Lemma \ref{Nazarov} in $L^{(\frac{p}{a})'}\ell^{(\frac{q}{a})'}$.
This application then gives
\begin{equation*}
\mathrm{II}=\left\|\left\{\gamma_jE_j\left(g_j\right)
\right\}_{j\in\mathbb Z}\right\|_{L^{(\frac{p}{a})'}\ell^{(\frac{q}{a})'}}
\lesssim\|\{g_j\}_{j\in\mathbb Z}\|_{L^{(\frac{p}{a})'}\ell^{(\frac{q}{a})'}}
\leq1,
\end{equation*}
where the last inequality is simply the condition
from the duality in $L^{\frac{p}{a}}\ell^{\frac{q}{a}}$ that we applied above.
Putting everything together, we have completed the proof of Lemma \ref{ad F1}.
\end{proof}

By Lemmas \ref{ad prelim} and \ref{ad F1},
we obtain the following conclusion. The result looks exactly the same as its Besov space analogue in Theorem \ref{ad Besov}; a slight difference in the statements is encoded into the convention concerning the number $J$.

\begin{theorem}\label{ad F0}
Let $p\in(0,\infty)$, $q\in(0,\infty]$, $s\in\mathbb R$, and $W\in A_p$.
Suppose that $B$ is $(D,E,F)$-almost diagonal with parameters
$D>J,$ $E>\frac{n}{2}+s,$ and $F>J-\frac{n}{2}-s,$
where $J:=\frac{n}{\min\{1,p,q\}}$.
Then $B$ is bounded on $\dot f^s_{p,q}(W)$.
\end{theorem}

\begin{proof}
As in the proof of Theorem \ref{ad Besov},
by Lemma \ref{s=0 enough}, we find that, to prove the present theorem,
it is enough to consider the case $s=0$.

We apply Lemma \ref{ad prelim} with both $H_j=W^{\frac{1}{p}}$
for any $j\in\mathbb Z$ and a suitable choice $a\in[0,p\wedge q\wedge 1)$.
Since $J=\frac{n}{p\wedge q\wedge 1}$, so if we choose $a$
very close to the upper bound of its admissible range,
we can make $\frac{n}{a}-J$ as small as we like.
In particular, given $D,E,F$ as in the assumptions (with $s=0$),
we can choose $a\in[0,p\wedge q\wedge 1)$
so that the coefficient in front of the norm in the right-hand side of \eqref{ABt}, namely
\begin{align}\label{coefF}
&2^{-(E-\frac{n}{2})k_-}2^{-(F+\frac{n}{2}-\frac{n}{a})k_+}2^{-(D-\frac{n}{a})l}\\
&\quad=2^{-(E-\frac{n}{2})k_-}2^{-[F+\frac{n}{2}-J-(\frac{n}{a}-J)]k_+}
2^{-[D-J-(\frac{n}{a}-J)]l},\notag
\end{align}
has exponential decay with respect to each of the variables $k_-$, $k_+$, and $l$.

On the other hand, for $a$ as above,
Lemma \ref{ad F1} guarantees that the norm in the right-hand side of \eqref{ABt} is estimated by
\begin{equation*}
\left\|\left\{\left[\fint_{B(\cdot,2^{l+k_+-i})}
\left|W^{\frac{1}{p}}(\cdot)\vec{t}_i(y)\right|^a\,dy\right]^{\frac{1}{a}}
\right\}_{i\in\mathbb Z}\right\|_{L^p\ell^q}
\lesssim\left\|\vec{t}\right\|_{\dot f^0_{p,q}(W)}.
\end{equation*}
This bound is independent of $(k,l)$,
and hence the series with coefficients \eqref{coefF} converges,
with $a$ chosen close enough to $p\wedge q\wedge 1$, as described.
This finishes the proof of Theorem \ref{ad F0}.
\end{proof}

{\color{blue}

\begin{remark}\label{rem max F}
Writing $W^{\frac1p}(\cdot)\vec t_i(y)=W^{\frac1p}(\cdot)W^{-\frac1p}(y)W^{\frac1p}(y)\vec t_i(y)$ and using H\"older's inequality to bound the $L^a$ average with $a<1$ by an $L^1$ average, the left-hand side of the inequality of Lemma \ref{ad F1} is seen to be dominated by
\begin{equation*}
  \left\|\left\{\left[\fint_{B(\cdot,2^{l+k_+-i})}
\left|W^{\frac{1}{p}}(\cdot)\vec{t}_i(y)\right|^a\,dy\right]^{\frac{1}{a}}
\right\}_{i\in\mathbb Z}\right\|_{L^p\ell^q}
  \leq \left\|\left\{\mathcal M_{W,p}(W^{\frac1p}(\cdot)\vec t_i(\cdot))\right\}_{i\in\mathbb Z}\right\|_{L^p\ell^q},
\end{equation*}
where $\mathcal M_{W,p}$ is the weighted maximal operator defined in Remark \ref{rem max}. If $p,q\in(1,\infty)$ and $W\in A_p$, then
\begin{equation*}
  \left\|\left\{\mathcal M_{W,p}(W^{\frac1p}(\cdot)\vec t_i(\cdot))\right\}_{i\in\mathbb Z}\right\|_{L^p\ell^q}
  \leq C\left\|\left\{W^{\frac1p}(\cdot)\vec t_i(\cdot)\right\}_{i\in\mathbb Z}\right\|_{L^p\ell^q}
  =C\left\|\vec t\right\|_{\dot f^0_{p,q}(W)}
\end{equation*}
by the matrix-weighted extension of the Fefferman--Stein vector-valued maximal inequality recently due to \cite{ks24}. Thus, an application of \cite[Theorem 1]{ks24} could replace Lemma \ref{ad F1} in the proof of Theorem \ref{ad F0} in the said range of exponents $p,q\in(1,\infty)$. However, our argument above still seems necessary to cover the full scale of parameters $p\in(0,\infty)$ and $q\in(0,\infty]$.

Note that a main point of the rescaling parameter $a\in(0,\min(p,q,1))$ of Lemma \ref{ad F1} is precisely to transform potentially small exponents to $p/a,q/a>1$. While this is a well-known trick in real analysis, a workable matrix-weighted version of the rescaled maximal operator $\mathcal [M(|f(\cdot)|^a)]^{1/a}$ does not  yet seem to be available.
\end{remark}

}

{\color{red}
\section{Sharpness on Triebel--Lizorkin Spaces with $q\geq\min(1,p)$}\label{sec sharp TL}
}

The following lemma, analogous to Lemma \ref{ad Besov sharp},
shows the sharpness of the assumptions of Theorem \ref{ad F0} in a large,
but not full, range of the parameters. That is, we are only able to prove
the exact same necessary conditions as in the Besov case of Lemma \ref{ad Besov sharp},
while the sufficient conditions in the Triebel--Lizorkin case in Theorem \ref{ad F0}
are sometimes stronger. The two conditions match when $J=\frac{n}{\min\{1,p\}}$
or, in other words, when $q\geq\min\{1,p\}$.

\medskip

\begin{lemma}\label{ad Triebel sharp}
Let $s\in\mathbb R$, $p\in(0,\infty)$, $q\in(0,\infty]$, and $D,E,F\in\mathbb{R}$.
Suppose that every $(D,E,F)$-almost diagonal matrix $B$ is bounded on $\dot f^s_{p,q} $.
Then $D>\frac{n}{\min\{1,p\}},$ $E>\frac{n}{2}+s,$ and
$F>\frac{n}{\min\{1,p\}}-\frac{n}{2}-s.$
\end{lemma}

\begin{proof}
We consider the $(D,E,F)$-almost diagonal matrix
$B^{DEF}:=\{b^{DEF}_{Q,R}\}_{Q,R\in\mathscr{Q}}$ from \eqref{bDEF}.
We first estimate the range of $D$.
By some arguments similar to those used in the estimations of both \eqref{est1} and \eqref{est2}, we obtain
$\sum_{Q\in\mathscr{Q}_0}\frac{1}{(1+|x_Q|)^{Dp}}<\infty$
and
$\sum_{l\in\mathbb{Z}^n}\frac{1}{(1+|l|)^D}<\infty.$
These further imply that $D>\frac{n}{\min\{1,p\}}$.

Next, we estimate the ranges of both $E$ and $F$ by considering the following two cases on $q$.

\emph{Case 1)} $q\in(0,\infty)$.
In this case, we first recall a classical embedding.
Applying \cite[Proposition 3.26]{bhyy1} with $\tau$ and $W$ replaced,
respectively, by $0$ and the identity matrix,
we find that, for any $\vec t:=\{\vec t_Q\}_{Q\in\mathscr{Q}}\subset\mathbb{C}^m$,
\begin{equation}\label{39y}
\left\|\vec t\right\|_{\dot b^{s}_{p,p\vee q}}
\leq\left\|\vec t\right\|_{\dot f^{s}_{p,q}}
\leq\left\|\vec t\right\|_{\dot b_{p,p\wedge q}^{s}}.
\end{equation}
Let $\vec t$ be the same as in \eqref{exa1}.
Then $\|\vec t\|_{\dot f_{p,q}^s}=1$. This, combined with
the boundedness of $B$ on $\dot f^s_{p,q}$,
\eqref{39y}, and \eqref{est3} with $q$ replaced by $p\vee q$, further implies that
\begin{align*}
\infty
&>\left\|B\vec t\right\|_{\dot f_{p,q}^s}^{p\vee q}
\geq\left\|B\vec t\right\|_{\dot b_{p,p\vee q}^s}^{p\vee q}\\
&=\sum_{j=0}^\infty\left\{2^{-j[(E-\frac{n}{2}-s)p+n]}
\sum_{k\in\mathbb{Z}^n}\left(1+2^{-j}|k| \right)^{-Dp}\right\}^{\frac{p\vee q}{p}}\\
&\quad+\sum_{j=-\infty}^{-1}\left\{2^{j[(F+\frac{n}{2}+s)p-n]}
\sum_{k\in\mathbb{Z}^n}(1+|k|)^{-Dp}\right\}^{\frac{p\vee q}{p}} .
\end{align*}
From this, the already showed fact that $D>\frac{n}{\min\{1,p\}}$, and \eqref{33z}, it follows that
$$
\infty
>\sum_{j=0}^\infty2^{-j(E-\frac{n}{2}-s)(p\vee q)}
+\sum_{j=-\infty}^{-1}2^{j(F+\frac{n}{2}+s-\frac{n}{p})(p\vee q)}
$$
and hence
$E>\frac{n}{2}+s$ and $F>\frac{n}{p}-\frac{n}{2}-s.$

Now, we prove $F>\frac{n}{2}-s$.
For any $N\in\mathbb{N}$, let $\vec t^{(N)}$ be the same as in \eqref{exa2}.
Then
\begin{equation}\label{289}
\left\|\vec t^{(N)}\right\|_{\dot f_{p,q}^s}
=\left(\sum_{Q\in\mathscr{Q}_N,\,|x_Q|<1}|Q|\right)^{\frac{1}{p}}\\
\sim1,
\end{equation}
where the positive equivalence constants depend only on $n$.
By \eqref{39y} and \eqref{est4} with $q$ replaced by $p\vee q$, we conclude that
$$
\left\|B\vec t\right\|_{\dot f_{p,q}^s}
\geq\left\|B\vec t\right\|_{\dot b_{p,p\vee q}^s}\\
\gtrsim\left[\sum_{j=0}^{N-1}2^{(j-N)(F+s-\frac{n}{2})(p\vee q)}\right]^{\frac{1}{p\vee q}}.
$$
This, together with \eqref{289} and the assumption that $B$ is bounded on $\dot f^s_{p,q}$,
further implies that, for any $N\in\mathbb{N}$,
$$
\sum_{j=-N}^{-1}2^{j(F+s-\frac{n}{2})(p\vee q)}
=\sum_{j=0}^{N-1}2^{(j-N)(F+s-\frac{n}{2})(p\vee q)}
\lesssim1.
$$
Letting $N\to\infty$, we obtain
$\sum_{j=-\infty}^{-1}2^{j(F+s-\frac{n}{2})(p\vee q)}\lesssim1$
and hence $F>\frac{n}{2}-s$.
This finishes the estimations of both $E$ and $F$ in this case.

\emph{Case 2)} $q=\infty$.
In this case, we first show $E>s+\frac{n}{2}$.
Let $\vec t$ be the same as in \eqref{exa3}. Then,
using the definitions of both $\vec t$ and $\|\cdot\|_{\dot f_{p,\infty}^s}$, we obtain
$$
\left\|\vec t\right\|_{\dot f_{p,\infty}^s}
=\left\|\sup_{Q\in\mathscr{Q},\,\ell(Q)\leq1,\,|x_Q|<1}
\mathbf{1}_Q\right\|_{L^p}
\sim1,
$$
where the positive equivalence constants depend only on $n$.
From this, the boundedness of $B$ on $\dot f^s_{p,\infty}$,
\eqref{39y}, and \eqref{est5}, we deduce that
$$
\infty
>\left\|B\vec t\right\|_{\dot f_{p,\infty}^s}
\geq\left\|B\vec t\right\|_{\dot b_{p,\infty}^s}
\gtrsim\sum_{i=-\infty}^02^{i(E-s-\frac{n}{2})}
$$
and hence $E>s+\frac{n}{2}$.

Next, we prove $F>\frac{n}{2}-s$.
Let $\vec t$ be the same as in \eqref{exa4}.
Notice that, for any $j\in\mathbb{Z}_+$, $Q\in\mathscr{Q}_j$, and $x\in Q$,
$|x-x_Q|<\sqrt{n}2^{-j}$ and hence
$1+|x_Q|\sim1+|x|,$
where the positive equivalence constants depend only on $n$.
This, combined with \eqref{33x}, further implies that
\begin{align*}
\left\|\vec t\right\|_{\dot f_{p,\infty}^s}
&=\left\|\sup_{j\in\mathbb{Z}_+}\sup_{Q\in\mathscr{Q}_j}
\frac{1}{(1+|x_Q|)^{\frac{n+1}{p}}}\mathbf{1}_Q\right\|_{L^p}\\
&\sim\left\|\frac{1}{(1+|\cdot|)^{\frac{n+1}{p}}}\right\|_{L^p}
\sim\left[\int_{\mathbb{R}^n}\frac{1}{(1+|x|)^{n+1}}\,dx\right]^{\frac{1}{p}}
\sim1.
\end{align*}
By the boundedness of $B$ on $\dot f^s_{p,\infty}$,
we conclude that $(B\vec t)_{Q_{0,\mathbf{0}}}$ makes sense.
This, together with $D>\frac{n}{\min\{1,p\}}$ and \eqref{33z}, further implies that
\begin{align*}
\infty
>\sum_{R\in\mathscr{Q}}\left|b_{Q_{0,\mathbf{0}},R}\vec t_R\right|
=\sum_{i=0}^\infty2^{-i(F+s+\frac{n}{2})}
\sum_{R\in\mathscr{Q}_i}\frac{1}{(1+|x_R|)^{D+\frac{n+1}{p}}}
=\sum_{i=0}^\infty2^{-i(F+s-\frac{n}{2})},
\end{align*}
and hence $F>\frac{n}{2}-s$.

Finally, we show $F>\frac{n}{p}-\frac{n}{2}-s$.
Let $\vec t$ be the same as in \eqref{exa1}.
Then $\|\vec t\|_{\dot f_{p,\infty}^s}=1$.
This, combined with both the assumption that $B$ is bounded on $\dot f^s_{p,\infty}$
and $F>\frac{n}{2}-s$, further implies that
\begin{align*}
\infty
&>\left\|B\vec t\right\|_{\dot f_{p,\infty}^s}^p
=\left\|\sup_{j\in\mathbb{Z}}2^{j(s+\frac{n}{2})}
\sup_{Q\in\mathscr{Q}_j}\left|\left(B\vec t\right)_Q\right|\mathbf{1}_Q\right\|_{L^p}^p\\
&=\int_{\mathbb{R}^n}\left[\sup_{j\in\mathbb{Z}}2^{j(s+\frac{n}{2})}
\sup_{Q\in\mathscr{Q}_j}|b_{Q,Q_{0,\mathbf{0}}}|\mathbf{1}_Q(x)\right]^p\,dx\\
&>\int_{\mathbb{R}^n}\left[\sup_{j\in\mathbb{Z},\,j\leq -1}2^{j(F+s+\frac{n}{2})}
\sup_{Q\in\mathscr{Q}_j}\left(1+2^j|x_Q|\right)^{-D}\mathbf{1}_Q(x)\right]^p\,dx\\
&>\sum_{i=0}^\infty\int_{Q_{-(i+1),\mathbf{0}}\setminus Q_{-i,\mathbf{0}}}
\left[\sup_{j\in\mathbb{Z},\,j\leq -1}2^{j(F+s+\frac{n}{2})}
\sup_{Q\in\mathscr{Q}_j}\left(1+2^j|x_Q| \right)^{-D}\mathbf{1}_Q(x)\right]^p\,dx\\
&>\sum_{i=0}^\infty|Q_{-(i+1),\mathbf{0}}\setminus Q_{-i,\mathbf{0}}|
\left[\sup_{j\in\mathbb{Z},\,j\leq -(i+1)}2^{j(F+s+\frac{n}{2})}\right]^p
\sim\sum_{i=0}^\infty2^{-i[(F+s+\frac{n}{2})p-n]},
\end{align*}
and hence $F>\frac{n}{p}-\frac{n}{2}-s$.
This finishes the estimations of both $E$ and $F$ in this case
and hence finishes the proof of Lemma \ref{ad Triebel sharp}.
\end{proof}

\section{Boundedness in the General Subcritical Case}
\label{sec ad general}

Combining the tools developed separately for Besov and Triebel--Lizorkin spaces,
we can now prove Theorem \ref{ad BF} in full generality:

\begin{proof}[Proof of Theorem \ref{ad BF}]
As in the proof of Theorem \ref{ad Besov},
by Lemma \ref{s=0 enough}, we conclude that,
to show the present theorem, it is enough to consider the case $s=0$.

Letting $P\in\mathscr{Q}$ be fixed, we need to estimate
$\|\{H_j(B\vec{t})_j\}_{j\in\mathbb Z}\|_{L\dot A_{p,q}}$,
where, for any $j\in\mathbb{Z}$,
$H_j:=\mathbf{1}_P\mathbf{1}_{[j_P,\infty)}(j)W^{\frac{1}{p}}$
and
$$
L\dot A_{p,q}:=
\begin{cases}
L^p\ell^q&\text{if }\dot a^{0,\tau}_{p,q}=\dot f^{0,\tau}_{p,q},\\
\ell^qL^p&\text{if }\dot a^{0,\tau}_{p,q}=\dot b^{0,\tau}_{p,q}.
\end{cases}
$$
Then, using Lemma \ref{ad prelim}, we obtain
\begin{align}\label{ABtx}
\left\|\left\{H_j\left(B\vec{t}\right)_j\right\}_{j\in\mathbb Z}\right\|_{L\dot A_{p,q}}^r
&\lesssim\sum_{k\in\mathbb{Z}}\sum_{l=0}^\infty
\Bigg[2^{-(E-\frac{n}{2})k_-}2^{-k_+(F+\frac{n}{2}-\frac{n}{a})}2^{-(D-\frac{n}{a})l}\\
&\quad\times\left.\left\|\left\{\left[\fint_{B(\cdot,2^{l+k_+-i})}
\left|\mathbf{1}_P(\cdot)W^{\frac{1}{p}}(\cdot)\vec{t}_{i}(y)\right|^a\,dy
\right]^{\frac{1}{a}}\right\}_{i\geq j_P+k}\right\|_{L\dot A_{p,q}}\right]^r,\notag
\end{align}
where $r:=p\wedge q\wedge 1$ and we chose
\begin{equation*}
\begin{cases}
a\in(0,p\wedge q\wedge 1)&\text{if }\dot a^{0,\tau}_{p,q}=\dot f^{0,\tau}_{p,q},\\
a:=p\wedge 1&\text{if }\dot a^{0,\tau}_{p,q}=\dot b^{0,\tau}_{p,q}.
\end{cases}
\end{equation*}

Notice that, for any $k,i\in\mathbb Z$ and $l\in\mathbb Z_+$, the bound
$2^{l+k_+-i}\leq\ell(P)=2^{-j_P}$ is equivalent to $i\geq j_P+k_+ +l$.
Then we further apply the $r$-triangle inequality to
the right-hand side of \eqref{ABtx} to obtain
\begin{align}\label{BF split}
&\left\|\left\{\left[\fint_{B(\cdot,2^{l+k_+-i})}
\left|\mathbf{1}_P(\cdot)W^{\frac{1}{p}}(\cdot)\vec{t}_{i}(y)\right|^a\,dy
\right]^{\frac{1}{a}}\right\}_{i\geq j_P+k}\right\|_{L\dot A_{p,q}}^r\\
&\quad\leq\left\|\left\{\left[\fint_{B(\cdot,2^{l+k_+-i})}
\left|\mathbf{1}_P(\cdot)W^{\frac{1}{p}}(\cdot)\vec{t}_{i}(y)\right|^a\,dy
\right]^{\frac{1}{a}}\right\}_{i\geq j_P+k_++l}\right\|_{L\dot A_{p,q}}^r\notag\\
&\qquad+\sum_{i=j_P+k}^{j_P+k_++l-1}
\left\|\left[\fint_{B(\cdot,2^{l+k_+-i})}
\left|\mathbf{1}_P(\cdot)W^{\frac{1}{p}}(\cdot)\vec{t}_{i}(y)\right|^a\,dy
\right]^{\frac{1}{a}}\right\|_{L^p}^r\notag\\
&\quad=:(\mathrm{I})^r+\sum_{i=j_P+k}^{j_P+k_++l-1}(J_i)^r,\notag
\end{align}
noticing that both $L^p\ell^q$ and $\ell^qL^p$ norms reduce to just the $L^p$ norm
when applied to an $\vec{f}=\{f_i\}_{i\in\mathbb Z}$
with only one non-zero component $f_i$ for some $i\in\mathbb Z$.

We first estimate $\mathrm{I}$.
By construction, in the first term on the right-hand side of \eqref{BF split},
we have $2^{l+k_+-i}\leq\ell(P)$.
Hence, if $x\in P$ and $y\in B(x,2^{l+k_+-i})\subset B(x,\ell(P))$,
then $y\in3P$. On the other hand,
it is clear that $j_P+k_++l\geq j_P$.
Thus, for any $\vec t\in\dot a^{0,\tau}_{p,q}(W)$, we have
\begin{align*}
\mathrm{I}
&\leq\left\|\left\{\left[\fint_{B(\cdot,2^{l+k_+-i})}
\left|W^{\frac{1}{p}}(\cdot)\left\{\mathbf{1}_{[j_P,\infty)}(i)
\mathbf{1}_{3P}(y)\vec{t}_{i}(y)\right\}\right|^a\,dy
\right]^{\frac{1}{a}}\right\}_{i\in\mathbb Z}\right\|_{L\dot A_{p,q}}\\
&\lesssim\left\|\left\{\mathbf{1}_{[j_P,\infty)}(i)\mathbf{1}_{3P}(\cdot)
W^{\frac{1}{p}}(\cdot)\vec{t}_{i}(\cdot)\right\}_{i\in\mathbb Z}\right\|_{L\dot A_{p,q}}
\lesssim|P|^{\tau}\left\|\vec{t}\right\|_{\dot a^{0,\tau}_{p,q}(W)},
\end{align*}
where, depending on whether we are in the Besov-type or the Triebel--Lizorkin-type case,
we used Lemma \ref{ad B1} or \ref{ad F1} to
$\{\mathbf{1}_{[j_P,\infty)}(i)\mathbf{1}_{3P}(\cdot)\vec{t}_{i}(\cdot)\}_{i\in\mathbb Z}$
in place of $\{\vec{t}_i\}_{i\in\mathbb Z}$ and then the definition of $\dot a^{0,\tau}_{p,q}(W)$.

Next, we estimate $J_i $.
Let $i\in\{j_P+k,\ldots,j_P+k_++l-1\}$ be fixed.
Notice that, when $x\in P$, there holds
\begin{equation}\label{5.25x}
B\left(x,2^{l+k_+-i}\right)\subset3\cdot2^{j_P+k_++l-i}P
\text{ and }
j_P+k_++l-i\in[1,k_-+l].
\end{equation}
We now claim that
\begin{equation}\label{165}
(J_i)^p\lesssim2^{(i-j_P-k_+-l)n}2^{d(j_P+k_++l-i+2)}
\int_{2^{j_P+k_++l-i+2}P}\left|W^{\frac{1}{p}}(y)\vec{t}_i(y)\right|^p\,dy,
\end{equation}
where $d$ is the $A_p$-dimension of $W$.
We consider the following two cases on $p$.

\emph{Case 1)} $p\in(0,1]$. In this case, since $a\leq p$, it follows from Jensen's inequality,
\eqref{5.25x}, and the definition of the $A_p$-dimension (Definition \ref{Ap dim}) that
\begin{align*}
(J_i)^p
&\leq\left\|\,\left[\fint_{B(\cdot,2^{l+k_+-i})}
\left|\mathbf{1}_P(\cdot)W^{\frac{1}{p}}(\cdot)\vec{t}_{i}(y)\right|^p\,dy
\right]^{\frac{1}{p}}\right\|_{L^p}^p\\
&\lesssim\int_P\fint_{3\cdot2^{j_P+k_++l-i}P}
\left\|W^{\frac{1}{p}}(x)W^{-\frac{1}{p}}(y)\right\|^p
\left|W^{\frac{1}{p}}(y)\vec{t}_i(y)\right|^p\,dy\,dx\\
&\lesssim\frac{|P|}{|2^{j_P+k_+ +l-i}P|}\int_{2^{j_P+k_++l-i+2}P}
\left[\fint_P\left\|W^{\frac{1}{p}}(x)W^{-\frac{1}{p}}(y)\right\|^p\,dx\right]
\left|W^{\frac{1}{p}}(y)\vec{t}_i(y)\right|^p\,dy\\
&\lesssim2^{(i-j_P-k_+-l)n}2^{d(j_P+k_++l-i+2)}
\int_{2^{j_P+k_+ +l-i+2}P}\left|W^{\frac{1}{p}}(y)\vec{t}_i(y)\right|^p\,dy.
\end{align*}
This finishes the proof of \eqref{165} in this case.

\emph{Case 2)} $p\in(1,\infty)$. In this case,
since $a\leq 1$, it follows from Jensen's inequality, \eqref{5.25x}, H\"older's inequality,
and the definition of the $A_p$-dimension (Definition \ref{Ap dim}) that
\begin{align*}
(J_i)^p
&\leq\left\|\fint_{B(\cdot,2^{l+k_+-i})}
\left|\mathbf{1}_P(\cdot)W^{\frac{1}{p}}(\cdot)\vec{t}_{i}(y)\right|\,dy\right\|_{L^p}^p\\
&\lesssim\int_P\left[\fint_{3\cdot2^{j_P+k_++l-i}P}
\left\|W^{\frac{1}{p}}(x)W^{-\frac{1}{p}}(y)\right\|
\left|W^{\frac{1}{p}}(y)\vec{t}_i(y)\right|\,dy\right]^p\,dx\\
&\lesssim\int_P\left[\fint_{3\cdot2^{j_P+k_++l-i}P}
\left\|W^{\frac{1}{p}}(x)W^{-\frac{1}{p}}(y)\right\|^{p'}\,dy\right]^{\frac{p}{p'}}
\fint_{3\cdot2^{j_P+k_++l-i}P}
\left|W^{\frac{1}{p}}(z)\vec{t}_i(z)\right|^p\,dz\,dx\\
&\lesssim\frac{|P|}{|2^{j_P+k_++l-i}P|}
\fint_P\left[\fint_{2^{j_P+k_++l-i+2}P}
\left\|W^{\frac{1}{p}}(x)W^{-\frac{1}{p}}(y)\right\|^{p'}\,dy\right]^{\frac{p}{p'}}\,dx\\
&\quad\times
\int_{2^{j_P+k_++l-i+2}P}
\left|W^{\frac{1}{p}}(z)\vec{t}_i(z)\right|^p\,dz\\
&\lesssim2^{(i-j_P-k_+-l)n}2^{d(j_P+k_++l-i+2)}
\int_{2^{j_P+k_++l-i+2}P}\left|W^{\frac{1}{p}}(z)\vec{t}_i(z)\right|^p\,dz.
\end{align*}
This finishes the proof of \eqref{165} in this case.

Using \eqref{165}, we conclude that
\begin{align*}
J_i
&\lesssim2^{(j_P+k_++l-i)\frac{d-n}{p}}
\left\|\mathbf{1}_{2^{j_P+k_++l-i+2}P}
W^{\frac{1}{p}}\vec{t}_i\right\|_{L^p}\\
&\leq2^{(j_P+k_++l-i)\frac{d-n}{p}}
\left|2^{j_P+k_++l-i+2}P\right|^{\tau}
\left\|\vec{t}\right\|_{\dot a^{0,\tau}_{p,q}(W)}
\sim2^{(j_P+k_++l-i)(\frac{d-n}{p}+n\tau)}
|P|^{\tau}\left\|\vec{t}\right\|_{\dot a^{0,\tau}_{p,q}(W)}\\
&\leq2^{(k_-+l)(\frac{d-n}{p}+n\tau)_+}|P|^{\tau}
\left\|\vec{t}\right\|_{\dot a^{0,\tau}_{p,q}(W)}
=2^{(k_-+l)n\widehat\tau}|P|^{\tau}
\left\|\vec{t}\right\|_{\dot a^{0,\tau}_{p,q}(W)},
\end{align*}
making use of the notation \eqref{tauJ} in the last step.
We have this bound for each of the terms in the sum over $i$ in \eqref{BF split},
and the total number of these terms is
$(j_P+k_++l)-(j_P+k)=k_-+l.$
Hence, we find that
\begin{align*}
&\left\|\left\{\left[\fint_{B(\cdot,2^{l+k_+-i})}
\left|\mathbf{1}_P(\cdot)W^{\frac{1}{p}}(\cdot)\vec{t}_{i}(y)\right|^a\,dy
\right]^{\frac{1}{a}}\right\}_{i\geq j_P+k}\right\|_{L\dot A_{p,q}}^r\\
&\quad\lesssim(1+k_-+l)\left[2^{(k_-+l)n\widehat\tau}|P|^{\tau}
\left\|\vec{t}\right\|_{\dot a^{0,\tau}_{p,q}(W)}\right]^r.
\end{align*}
Substituting the above estimate into \eqref{ABtx},
dividing by the factor $|P|^{\tau}$,
and then taking the supremum over all $P\in\mathscr{Q}$, we obtain
\begin{align*}
\left\|B\vec{t}\right\|_{\dot a^{0,\tau}_{p,q}(W)}^r
&\lesssim\sum_{k\in\mathbb{Z}}\sum_{l=0}^\infty(1+k_-+l)\\
&\quad\times\left[2^{-(E-\frac{n}{2})k_-}2^{-k_+(F+\frac{n}{2}-\frac{n}{a})}
2^{-(D-\frac{n}{a})l}2^{(k_-+l)n\widehat\tau}
\left\|\vec{t}\right\|_{\dot a^{0,\tau}_{p,q}(W)}\right]^r,
\end{align*}
and next it is easy to read off the conditions for convergence,
which requires that each of the summation variables $k_-$, $k_+$, and $l$
needs to have a negative coefficient in the exponent.
(Once this condition is satisfied, the polynomial factor in front is irrelevant.)
For $k_-$, it is immediate that we need $E>\frac{n}{2}+n\widehat\tau$.
For $k_+$, recalling that $\frac{n}{a}-J$ is either zero
or can be made as close to zero as we like by an appropriate choice of $a$,
we find that we need $F>J-\frac{n}{2}$.
On the coefficient of $l$, for similar reasons, we need $D>J+n\widehat\tau$.
These are precisely the assumptions of the present theorem when $s=0$, and hence
$\|B\vec{t}\|_{\dot a^{0,\tau}_{p,q}(W)}
\lesssim\|\vec{t}\|_{\dot a^{0,\tau}_{p,q}(W)}.$
This finishes the proof of Theorem \ref{ad BF}.
\end{proof}

\begin{remark}\label{ad BF rem}
Let us compare Theorem \ref{ad BF} with the various cases of Theorem \ref{ad FJ-YY}.
\begin{enumerate}[\rm(i)]
\item When $\tau=0$, $m=1$, and $W\equiv1$,
we have $\widehat\tau=0$ and hence, in this case,
Theorem \ref{ad BF} reduces to Theorem \ref{ad FJ-YY}\eqref{FJ 3.1}.
\item In the case when $\tau=0$ and a general $W\in A_p$,
we still have $\widehat\tau=0$ and hence, in this case,
Theorem \ref{ad BF} improves parts \eqref{Ro 1.10} and \eqref{FR 2.6}
of Theorem \ref{ad FJ-YY} because $\beta_W\in[n,\infty)$
(see, for instance, \cite[Remark 2.21]{bhyy1}).
More precisely, the conditions in both of them involve some extra non-negative terms,
in addition to those present in Theorem \ref{ad BF}.
\item For any $\tau\in[0,\infty)$, $m=1$, and $W\equiv1$, we have $d=0$;
hence $\widehat\tau=(\tau-\frac{1}{p})_+=0$ for any $\tau\leq\frac{1}{p}$,
and one can then find that, in this case, Theorem \ref{ad BF}
is just Theorem \ref{ad FJ-YY}\eqref{YY 4.1}
when $\tau\in[0,\frac{1}{p})$
or $\dot a^{s,\tau}_{p,q}(W)=\dot b^{s,\frac{1}{p}}_{p,q}(W)$ with $q\in(0,\infty)$,
but not necessarily in the remaining cases.
\item When $m=1$, Theorem \ref{ad FJ-YY}(iii) requires that,
for any cubes $Q,R\subset\mathbb{R}^n$ with $Q\subset R$,
$$
\frac{w(Q)}{w(R)}\lesssim\left(\frac{|Q|}{|R|}\right)^{\alpha_1}
$$
and hence, for any $j\in\mathbb{Z}_+$,
$$
\frac{w(Q)}{w(2^jQ)}\lesssim2^{-jn\alpha_1}=2^{j(d-n)},
$$
where $d:=n(1-\alpha_1)$. This, together with \cite[Proposition 2.37]{bhyy1},
further implies that $w$ has the $A_p$-dimension $d$.
Therefore, the assumptions of Theorem \ref{ad BF} read as
\begin{align}\label{260}
D>J+n\left(\tau-\frac{\alpha_1}{p}\right)_+,\
E>\frac{n}{2}+s+n\left(\tau-\frac{\alpha_1}{p}\right)_+,\text{ and }
F>J-\frac{n}{2}-s,
\end{align}
where $J$ is the same as in \eqref{J}
and $\alpha_1,\alpha_2\in(0,\infty)$ satisfying $\alpha_1\leq\alpha_2$
are the same as in Theorem \ref{ad FJ-YY}.
Comparing with the assumptions of Theorem \ref{ad FJ-YY}(iii), that is,
$D>J+\varepsilon_0,$ $E>\frac{n}{2}+s+\frac{\varepsilon_0}{2},$ and
$F>J-\frac{n}{2}-s+\frac{\varepsilon_0}{2},$
since
$$
\varepsilon_0:=\max\left\{\frac{n(\alpha_2-\alpha_1)}{p},\
2n\left(\tau-\frac{\alpha_1}{p}\right)\right\}
\geq2n\left(\tau-\frac{\alpha_1}{p}\right)_+,
$$
\eqref{260} obviously relax the assumptions on $D$, $E$, and $F$.
\end{enumerate}
\end{remark}

{\color{red}
\section{Sharpness in Besov-Type Spaces with $\tau=\frac1p$}\label{sec sharp 1p}
}

As we saw in Sections \ref{sec sharp Besov} and \ref{sec sharp TL}, the conditions of Theorem \ref{ad BF} are sharp
both in the full range of Besov spaces $\dot b_{p,q}^s=\dot b_{p,q}^{s,0}$,
as well as in those Triebel--Lizorkin spaces
$\dot f_{p,q}^s=\dot f_{p,q}^{s,0}$ with $q\geq\min\{1,p\}$.
While $\tau=0$ in both these cases, the following lemma will imply the
sharpness of Theorem \ref{ad BF} also in a range of spaces with $\tau=\frac1p$.
In contrast to the previous examples that were given in the unweighted case,
the following example will be a genuinely weighted one.
The weight in question is a usual scalar-valued power weight;
this is of course naturally identified with a matrix weight
whose values are scalar multiples of the identity matrix.

\begin{lemma}\label{sharp2}
Let $p\in(0,\infty)$, $q\in(1,\infty]$, $s\in\mathbb R$, $d\in[0,n)$,
and $W(x):=|x|^{-d}$ for any $x\in\mathbb R^n\setminus\{\mathbf{0}\}$.
Let $(D,E,F)\in\mathbb R^3$, and suppose that every $(D,E,F)$-almost diagonal matrix $B$
is bounded on $\dot a_{p,q}^{s,\frac1p}(W)$. Then
\begin{equation}\label{262}
D>n+\frac{d}{p},\
E>\frac{n}{2}+s+\frac{d}{p},
\text{ and }
F>\frac{n}{2}-s.
\end{equation}
\end{lemma}

\begin{proof}
From \cite[Lemmas 2.39 and 2.40]{bhyy1}, we infer that
$W\in A_p$ has the $A_p$-dimension $d$.
Let $\mathbb A=\{A_Q\}_{Q\in\mathscr Q}$ be a sequence of reducing operators of order $p$ for $W$.
By \cite[Corollary 2.42]{bhyy1}, we conclude that,
for any $j\in\mathbb Z$, $k\in\mathbb Z^n$, $Q:=Q_{j,k}\in\mathscr{Q}$,
and $\vec z\in\mathbb{C}^m$,
\begin{align*}
\left|A_Q\vec z\right|^p
&\sim\fint_Q\left|W^{\frac{1}{p}}(x)\vec z\right|^p\,dx
=\fint_Q|x|^{-d}\,dx\left|\vec z\right|^p\\
&\sim\left[\ell(Q)+|c_Q|\right]^{-d}\left|\vec z\right|^p
\sim2^{jd}(1+|k|)^{-d}\left|\vec z\right|^p.
\end{align*}
Therefore, without loss of generality, we may assume
$A_Q:=2^{j\frac{d}{p}}(1+|k|)^{-\frac{d}{p}}$
for each $Q:=Q_{j,k}\in\mathscr{Q}$ with $j\in\mathbb Z$ and $k\in\mathbb Z^n$.
Moreover, we consider the $(D,E,F)$-almost diagonal matrix
$B^{DEF}:=\{b^{DEF}_{Q,R}\}_{Q,R\in\mathscr{Q}}$ from \eqref{bDEF}.

By \cite[Theorem 3.27 and Corollary 4.18]{bhyy1}, we find that
\begin{equation*}
\dot b_{p,q}^{s,\frac{1}{p}}(W)=\dot b_{p,q}^{s,\frac{1}{p}}(\mathbb{A})
\text{ and }
\dot f_{p,q}^{s,\frac{1}{p}}(W)=\dot f_{p,q}^{s,\frac{1}{p}}(\mathbb A)=\dot f_{\infty,q}^{s}(\mathbb{A}).
\end{equation*}
Hence, by the assumption, $B$ is bounded on either $\dot b_{p,q}^{s,\frac{1}{p}}(\mathbb{A})$
or $\dot f_{p,q}^{s,\frac{1}{p}}(\mathbb A)=\dot f_{\infty,q}^{s}(\mathbb{A})$.

We first estimate the range of $D$.
Let $\vec e\in\mathbb{C}^m$ satisfy $|\vec e|=1$
and $\vec t:=\{\vec t_Q\}_{Q\in\mathscr{Q}}$, where, for any $Q\in\mathscr{Q}$,
$$
\vec t_Q:=\begin{cases}
A_Q^{-1}\vec e&\text{if }\ell(Q)=1,\\
\vec{\mathbf{0}}&\text{otherwise}.
\end{cases}
$$
For a sequence like this that is nonzero on one ``level'' only,
it is easy to find that the Besov-type and the Triebel--Lizorkin-type norms
coincide, and we have
$$
\left\|\vec t\right\|_{\dot a_{p,q}^{s,\frac{1}{p}}(\mathbb{A})}^p
=\sup_{P\in\mathscr{Q},\,|P|\geq1}\frac{1}{|P|}
\sum_{Q\in\mathscr{Q}_0,\,Q\subset P}
\left|A_Q\vec t_Q\right|^p
=1.
$$
Hence $(B\vec t)_{Q_{0,\mathbf{0}}}$ makes sense.
From this, it follows that
$$
\infty
>\sum_{R\in\mathscr{Q}}\left|b_{Q_{0,\mathbf{0}},R}\vec t_R\right|
=\sum_{k\in\mathbb{Z}^n}\left|b_{Q_{0,\mathbf{0}},Q_{0,k}}\vec t_{Q_{0,k}}\right|
=\sum_{k\in\mathbb{Z}^n}\left(1+|k|\right)^{-(D-\frac{d}{p})}
$$
and hence $ D>n+\frac{d}{p}.$

Now, we estimate the range of $E$.
Let $\vec e\in\mathbb{C}^m$ satisfy $|\vec e|=1$
and $\vec t:=\{\vec t_Q\}_{Q\in\mathscr{Q}}$, where, for any $Q\in\mathscr{Q}$,
$$
\vec t_{Q}:=\begin{cases}
|Q|^{\frac{s}{n}+\frac12}A_Q^{-1}\vec e&\text{if }x_Q=\mathbf{0},\\
\vec{\mathbf{0}}&\text{otherwise}.
\end{cases}
$$
Then, using the definitions of $\vec t$ and the relevant quasi-norms, we obtain
\begin{align*}
\left\|\vec t\right\|_{\dot b_{p,q}^{s,\frac{1}{p}}(\mathbb{A})}
&=\sup_{P\in\mathscr{Q},\,x_P=\mathbf{0}}\frac{1}{|P|^{\frac{1}{p}}}
\left\{\sum_{j=j_P}^\infty2^{jsq}
\left[|Q|^{1-\frac{p}{2}}|Q|^{(\frac{s}{n}+\frac12)p}
\right]^{\frac{q}{p}}\right\}^{\frac{1}{q}}\\
&=\sup_{P\in\mathscr{Q},\,x_P=\mathbf{0}}\frac{1}{|P|^{\frac{1}{p}}}
\left(\sum_{j=j_P}^\infty2^{-jn\frac{q}{p}}\right)^{\frac{1}{q}}
\sim1,
\end{align*}
as well as
$$
\left\|\vec t\right\|_{\dot f^s_{\infty,q}(\mathbb A)}
=\sup_{P\in\mathscr{Q},\,x_P=\mathbf{0}}\left[
\frac{1}{|P|}\sum_{j=j_P}^\infty|Q_{j,\mathbf{0}}|\right]^{\frac{1}{q}}
\sim1 \text{ when } q\in(1,\infty)
$$
and
$$
\left\|\vec t\right\|_{\dot f^s_{\infty,\infty}(\mathbb A)}
=\sup_{Q\in\mathscr{Q}}|Q|^{-(\frac{s}{n}+\frac12)}\left|A_Q\vec t_Q\right|
=1.
$$
Thus, $(B\vec t)_{Q_{0,\mathbf{0}}}$ is well defined in each case, which further implies that
$$
\infty
>\sum_{R\in\mathscr{Q}}\left|b_{Q_{0,\mathbf{0}},R}\vec t_R\right|
=\sum_{i\in\mathbb{Z}}\left|b_{Q_{0,\mathbf{0}},Q_{i,\mathbf{0}}}\vec t_{Q_{i,\mathbf{0}}}\right|
=\sum_{i=-\infty}^02^{i(E-s-\frac{n}{2}-\frac{d}{p})}
$$
and hence
$E>\frac{n}{2}+s+\frac{d}{p}.$

Finally, we estimate the range of $F$.
Let $\vec e\in\mathbb{C}^m$ satisfy $|\vec e|=1$
and $\vec t:=\{\vec t_Q\}_{Q\in\mathscr{Q}}$, where, for any $Q\in\mathscr{Q}$,
$\vec t_Q:=(1+|j_Q|)^{-1}|Q|^{\frac{s}{n}+\frac12}A_Q^{-1}\vec e.$
Then, from the definitions of $\vec t$ and the relevant quasi-norms, we deduce that
\begin{align*}
\left\|\vec t\right\|_{\dot b_{p,q}^{s,\frac{1}{p}}(\mathbb{A})}
&=\sup_{P\in\mathscr{Q}}\frac{1}{|P|^{\frac{1}{p}}}
\left\{\sum_{j=j_P}^\infty2^{jsq}(1+|j|)^{-q}
\left[\sum_{Q\in\mathscr{Q}_j,\,Q\subset P}|Q|^{1-\frac{p}{2}}|Q|^{(\frac{s}{n}+\frac12)p}
\right]^{\frac{q}{p}}\right\}^{\frac{1}{q}}\\
&=\sup_{P\in\mathscr{Q}}
\left[\sum_{j=j_P}^\infty(1+|j|)^{-q}\right]^{\frac{1}{q}}
\sim1,
\end{align*}
as well as
$$
\left\|\vec t\right\|_{\dot f^s_{\infty,q}(\mathbb A)}
=\sup_{P\in\mathscr{Q}}\left[\sum_{j=j_P}^\infty(1+|j|)^{-q}\right]^{\frac{1}{q}}
\sim 1 \text{ when } q\in(1,\infty)
$$
and
$$
\left\|\vec t\right\|_{\dot f^s_{\infty,\infty}(\mathbb A)}
=\sup_{Q\in\mathscr{Q}}|Q|^{-(\frac{s}{n}+\frac12)}\left|A_Q\vec t_Q\right|
=1.
$$
Thus, $(B\vec t)_{Q_{0,\mathbf{0}}}$ is well defined in each case, and hence
\begin{align*}
\infty
&>\sum_{R\in\mathscr{Q}}\left| b_{Q_{0,\mathbf{0}},R}\vec t_R\right|
=\sum_{i=1}^\infty\sum_{k\in\mathbb{Z}^n}
\left(1+2^{-i}|k|\right)^{-D}2^{-iF}(1+i)^{-1}
2^{-i(s+\frac{n}{2})}2^{-i\frac{d}{p}}(1+|k|)^{\frac{d}{p}}\\
&=\sum_{i=1}^\infty2^{-i(F+s+\frac{n}{2}+\frac{d}{p})}
(1+i)^{-1}\sum_{k\in\mathbb{Z}^n}
\left(1+2^{-i}|k|\right)^{-D}(1+|k|)^{\frac{d}{p}}\\
&\sim\sum_{i=1}^\infty2^{-i(F+s-\frac{n}{2})}(1+i)^{-1},
\end{align*}
which further implies that $F>\frac{n}{2}-s$.
This finishes the proof of Lemma \ref{sharp2}.
\end{proof}

\begin{remark}
When $s\in\mathbb R$, $p\in[1,\infty)$, $q\in(1,\infty]$,
and $\dot a^{s,\tau}_{p,q}(W)=\dot b_{p,q}^{s,\frac{1}{p}}(W)$,
the assumptions \eqref{262} are just the assumptions \eqref{ad new},
which further implies that \eqref{262} are the sharp almost diagonal conditions
for $\dot b_{p,q}^{s,\frac{1}{p}}(W)$ in this case.
\end{remark}

\section{Critical and Supercritical Spaces}
\label{special}

By a different argument, we here improve Theorem \ref{77}
in the critical and the supercritical cases as follows:

\begin{theorem}\label{ad iff cor}
Let $p\in(0,\infty)$, $q\in(0,\infty]$, $s\in\mathbb R$, $\tau\in[\frac1p,\infty)$, and
\begin{equation*}
J_\tau:=\begin{cases}
n &\displaystyle \text{if }\tau\in\left(\frac1p,\infty\right), \\
\displaystyle \frac{n}{\min\{1,q\}} & \displaystyle\text{if }\tau=\frac1p.
\end{cases}
\end{equation*}
If $\tau=\frac1p$ and $q\in(0,\infty)$,
we also assume that $\dot a^{s,\tau}_{p,q}(W)=\dot f^{s,\frac1p}_{p,q}(W)$
is of Triebel--Lizorkin-type. Then the following implications hold.
\begin{enumerate}[\rm(i)]
\item If three numbers $D,E,F\in\mathbb R$ satisfy
\begin{align}\label{173}
D>J_\tau+\frac{d}{p},\
E>\frac{n}{2}+s+n\left(\tau-\frac1p\right)+\frac{d}{p},\text{ and }
F>J_\tau-\frac{n}{2}-s-n\left(\tau-\frac1p\right),
\end{align}
then all $(D,E,F)$-almost diagonal operators are bounded on $\dot a^{s,\tau}_{p,q}(W)$ for all $W\in A_p$ with $A_p$-dimension $d$.
\item If $\tau\in(\frac1p,\infty)$ or $q\in(1,\infty]$,
then \eqref{173} is also necessary for this conclusion.
\end{enumerate}
\end{theorem}

Before proving Theorem \ref{ad iff cor}, we recall the
averaging matrix-weighted Triebel--Lizorkin sequence space
(see \cite[Definition 4.4]{bhyy1}).

\begin{definition}\label{def TLpInfty}
Let $s\in\mathbb{R}$, $q\in(0,\infty]$,
$p\in(0,\infty)$, $W\in A_p$, and
$\mathbb{A}:=\{A_Q\}_{Q\in\mathscr{Q}}$ be a sequence of
reducing operators of order $p$ for $W$.
The \emph{homogeneous averaging matrix-weighted Triebel--Lizorkin sequence space}
$\dot f_{\infty,q}^s(\mathbb{A})$ is defined to be the set of all sequences
$\vec t:=\{\vec t_Q\}_{Q\in\mathscr{Q}}\subset\mathbb{C}^m$ such that
$$
\left\|\vec t\right\|_{\dot f_{\infty,q}^s(\mathbb{A})}
:=\sup_{P\in\mathscr{Q}}\left[\fint_P\sum_{j=j_P}^\infty
2^{jsq}\left|A_j(x)\vec t_j(x)\right|^q\,dx\right]^{\frac{1}{q}}<\infty,
$$
where, for any $j\in\mathbb Z$, $A_j$ and $\vec t_j$
are the same as, respectively, in \eqref{Aj} and \eqref{vec tj}.
\end{definition}

To prove Theorem \ref{ad iff cor},
we first consider the case $\dot f^{s,\frac1p}_{p,q}(W)$.

\begin{lemma}\label{110}
Let $p\in(0,\infty)$, $q\in(0,\infty]$, $s\in\mathbb R$, $\tau=\frac1p$, and
let $W\in A_p$ have the $A_p$-dimension $d\in[0,n)$.
Suppose that $B$ is $(D,E,F)$-almost diagonal with parameters
$D>J_\tau+\frac{d}{p},$ $E>\frac{n}{2}+s+\frac{d}{p},$ and
$F>J_\tau-\frac{n}{2}-s,$
where $J_\tau:=\frac{n}{\min\{1,q\}}$.
Then $B$ is bounded on $\dot f^{s,\frac1p}_{p,q}(W)$.
\end{lemma}

\begin{proof}
If $p\geq\min\{1,q\}$, then the assumptions of Theorem \ref{ad BF}
coincide with those of the present lemma that we are proving:
Firstly, $\min\{1,p,q\}=\min\{1,q\}$ so that
$J(\dot f^{s,\frac{1}{p}}_{p,q})=\frac{n}{\min\{1,q\}}=J_\tau$.
Secondly, when $\tau=\frac1p$, it follows that $n\widehat\tau=\frac{d}{p}$.
Substituting these value into the conditions on $(D,E,F)$ in Theorem \ref{ad BF},
we find that the present lemma is just a reformulation of Theorem \ref{ad BF} in this case.

When $p<\min\{1,q\}$,
Theorem \ref{ad BF} needs stronger assumptions.
Thus, we present a separate argument using the identity
$\dot f^{s,\frac{1}{p}}_{p,q}(W)=\dot f^s_{\infty,q}(\mathbb A)$
provided by \cite[Corollary 4.18]{bhyy1}, where $\mathbb A:=\{A_Q\}_{Q\in\mathscr{Q}}$
is a sequence of reducing operators of order $p$ for $W$.
From this point on, we work with the space $\dot f^s_{\infty,q}(\mathbb A)$.

Our argument will make critical use of Lemma \ref{22 precise},
and it is useful to make the following observation:
While the bound of Lemma \ref{22 precise} in principle depends on
the triple of $A_p$-dimensions $(d,\widetilde d,\Delta)$ of $W\in A_p$,
we are now working under the assumption that $p<1\wedge q\leq 1$.
Thus, according to Definition \ref{def Ap dims},
we have $\widetilde d=0$ and $\Delta=\frac{d}{p}$.
Thus, the conclusion \eqref{22 p<1} of Lemma \ref{22 precise} now takes the form
\begin{equation}\label{22 orig}
\left\|A_QA_R^{-1}\right\|
\leq C\max\left\{\left[\frac{\ell(R)}{\ell(Q)}\right]^{\frac{d}{p}},1\right\}
\left[1+\frac{|c_Q-c_R|}{\ell(Q)\vee\ell(R)}\right]^{\frac{d}{p}}.
\end{equation}

With this observation made, we proceed to the actual argument.
For any sequence $\vec t:=\{\vec t_R\}_{R\in\mathscr{Q}}\in\dot f^s_{\infty,q}(\mathbb A)$
and any $Q\in\mathscr{Q}$, we split $B$ into its upper and lower triangular parts
$$
\left(B_0\vec t\right)_Q
:=\sum_{R\in\mathscr{Q},\,\ell(R)\geq\ell(Q)}b_{Q,R}\vec t_R
\text{ and }
\left(B_1\vec t\right)_Q
:=\sum_{R\in\mathscr{Q},\,\ell(R)<\ell(Q)}b_{Q,R}\vec t_R,
$$
where $\{b_{Q,R}\}_{Q,R\in\mathscr{Q}}$ is the matrix related to $B$.
To show that $B$ is bounded on $\dot f^s_{\infty,q}(\mathbb A)$,
it suffices to prove that both $B_0$ and $B_1$ are bounded on it.

\textbf{The upper triangular part:}
We first establish the boundedness of $B_0$ on $\dot f^s_{\infty,q}(\mathbb A)$.
Let $u_R:=|R|^{-(\frac{s}{n}+\frac12)}|A_R\vec t_R|$, and notice that
\begin{equation}\label{5320}
\sup_{R\in\mathscr{Q}}u_R
=\left\|\vec t\right\|_{f^s_{\infty,\infty}(\mathbb A)}
\leq\left\|\vec t\right\|_{f^s_{\infty,q}(\mathbb A)}.
\end{equation}

From \eqref{22 orig}, we infer that, for any $j\in\mathbb{Z}$ and $Q\in\mathscr{Q}_j$,
\begin{align}\label{5321}
|Q|^{-(\frac{s}{n}+\frac12)}\left|A_Q\left(B_0\vec t\right)_Q\right|
&\leq|Q|^{-(\frac{s}{n}+\frac12)}\sum_{R\in\mathscr{Q},\,\ell(R)\geq\ell(Q)}|b_{Q,R}|
\left\|A_QA_R^{-1}\right\|\left|A_R\vec t_R\right|\\
&\lesssim|Q|^{-(\frac{s}{n}+\frac12)}\sum_{R\in\mathscr{Q},\,\ell(R)\geq\ell(Q)}
\left[1+\frac{|x_Q-x_R|}{\ell(R)}\right]^{-D}
\left[\frac{\ell(Q)}{\ell(R)}\right]^E\notag \\
&\quad\times
\left[\frac{\ell(R)}{\ell(Q)}\right]^{\frac{d}{p}}
\left[1+\frac{|x_Q-x_R|}{\ell(R)}\right]^{\frac{d}{p}}|R|^{s+\frac{n}{2}}u_R\notag \\
&=\sum_{R\in\mathscr{Q},\,\ell(R)\geq\ell(Q)}
\left[\frac{\ell(Q)}{\ell(R)}\right]^{\widetilde{E}}
\left[1+\frac{|x_Q-x_R|}{\ell(R)}\right]^{-\widetilde{D}}u_R,\notag
\end{align}
where
\begin{equation}\label{5322}
\widetilde D:=D-\frac{d}{p}>\frac{n}{\min\{1,q\}}\geq n
\text{ and }
\widetilde E:=E-s-\frac{n}{2}-\frac{d}{p}>0.
\end{equation}
By Lemma \ref{one fits all} and \eqref{5321}, we conclude that
\begin{align*}
|Q|^{-(\frac{s}{n}+\frac12)}\left|A_Q\left(B_0\vec t\right)_Q\right|
&\lesssim \left\{\sum_{R\in\mathscr{Q},\,\ell(R)\geq\ell(Q)}
\left[\frac{\ell(Q)}{\ell(R)}\right]^{\widetilde{E}\min\{1,q\}}
\left[1+\frac{|x_Q-x_R|}{\ell(R)}\right]^{-\widetilde{D}\min\{1,q\}}u_R^q\right\}^{\frac1q} \\
&\quad\times \left\{\sum_{R\in\mathscr{Q},\,\ell(R)\geq\ell(Q)}
\left[\frac{\ell(Q)}{\ell(R)}\right]^{\widetilde{E}}
\left[1+\frac{|x_Q-x_R|}{\ell(R)}\right]^{-\widetilde{D}}\right\}^{(1-\frac1q)_+}  \\
&=:\mathrm{I}^{\frac1q}\times\mathrm{II}^{(1-\frac1q)_+},
\end{align*}
where
\begin{equation*}
\mathrm{II}=\sum_{i=-\infty}^{j}2^{(i-j)\widetilde{E}}
\sum_{k\in\mathbb Z^n}\left(1+|2^i x_Q-k|\right)^{-\widetilde{D}}\sim 1
\end{equation*}
by \eqref{5322} and \eqref{33z}. For $q=\infty$, we obtain simply
\begin{equation*}
|Q|^{-(\frac{s}{n}+\frac12)}\left|A_Q\left(B_0\vec t\right)_Q\right|
\lesssim\mathrm{II}\times\sup_{R\in\mathscr Q}u_R
\sim \|\vec t\|_{\dot f_{\infty,\infty}^s(\mathbb A)}
\end{equation*}
and hence $\|B\vec t\|_{\dot f_{\infty,\infty}^s(\mathbb A)}\lesssim\|\vec t\|_{\dot f_{\infty,\infty}^s(\mathbb A)}$ by taking the supremum over $Q\in\mathscr Q$.

For any $q\in(0,\infty)$, we find that
\begin{align}\label{167}
\left\|B_0\vec t\right\|_{\dot f^s_{\infty,q}(\mathbb A)}^q
&=\sup_{P\in\mathscr{Q}}
\frac{1}{|P|}\sum_{j=j_P}^\infty\sum_{Q\in\mathscr{Q}_j,\,Q\subset P}|Q|
\left[|Q|^{-(\frac{s}{n}+\frac12)}\left|A_Q\left(B_0\vec t\right)_Q\right|\right]^q\\
&\lesssim\sup_{P\in\mathscr{Q}}\frac{1}{|P|}\sum_{j=j_P}^\infty
\sum_{i=-\infty}^{j}2^{(i-j)\widetilde{E}\min\{1,q\}}I_1(i,j),\notag
\end{align}
where, for each $i,j\in\mathbb Z$,
$$
I_1(i,j):=\sum_{R\in\mathscr{Q}_i}\sum_{Q\in\mathscr{Q}_j,\,Q\subset P}|Q|
\left(1+2^i|x_Q-x_R|\right)^{-\widetilde{D}\min\{1,q\}}(u_R)^q.
$$
Notice that $\widetilde{D}\min\{1,q\}>n$ by \eqref{5322}.
From \eqref{33z} and \eqref{5320},
we then deduce that, for any $j\in\mathbb{Z}$ with $j\geq j_P$
and for any $i\in\mathbb{Z}$ with $i<j_P$,
\begin{align*}
I_1(i,j)
&\leq\sum_{R\in\mathscr{Q}_i}\sum_{Q\in\mathscr{Q}_j,\,Q\subset P}|Q|
\left(1+2^i|x_Q-x_R|\right)^{-\widetilde{D}\min\{1,q\}}
\left\|\vec t\right\|_{\dot f^s_{\infty,\infty}(\mathbb A)}^q\\
&=\sum_{Q\in\mathscr{Q}_j,\,Q\subset P}|Q|
\sum_{k\in\mathbb{Z}^n}\left(1+|2^i x_Q-k|\right)^{-\widetilde{D}\min\{1,q\}}
\left\|\vec t\right\|_{\dot f^s_{\infty,\infty}(\mathbb A)}^q\\
&\sim\sum_{Q\in\mathscr{Q}_j,\,Q\subset P}|Q|
\left\|\vec t\right\|_{\dot f^s_{\infty,\infty}(\mathbb A)}^q
=|P|\left\|\vec t\right\|_{\dot f^s_{\infty,\infty}(\mathbb A)}^q.
\end{align*}
On the other hand, from Lemmas \ref{33y} and \ref{33} and from \eqref{33x}, we infer that,
for any $j\in\mathbb{Z}$ with $j\geq j_P$
and for any $i\in\{j_P,j_P+1,\ldots,j\}$,
\begin{align*}
I_1(i,j)
&\sim\sum_{R\in\mathscr{Q}_i,\,R\subset 3P}
\int_P\frac{1}{(1+2^i|x-x_R|)^{\widetilde{D}\min\{1,q\}}}\,dx\,(u_R)^q\\
&\quad+\sum_{k\in\mathbb{Z}^n,\,\|k\|_\infty\geq2}
\sum_{R\in\mathscr{Q}_i,\,R\subset P+k\ell(P)}
\sum_{Q\in\mathscr{Q}_j,\,Q\subset P}|Q|
\left(2^{i-j_P}|k|\right)^{-\widetilde{D}\min\{1,q\}}(u_R)^q \\
&=:I_2(i,j)+I_3(i,j),
\end{align*}
where
$I_2(i,j)\lesssim\sum_{R\in\mathscr{Q}_i,\,R\subset 3P}|R|(u_R)^q$
and, using again that  $\widetilde{D}\min\{1,q\}>n$ by \eqref{5322},
\begin{align*}
I_3(i,j)&= 2^{(i-j_P)(n-\widetilde{D}\min\{1,q\})}\sum_{k\in\mathbb{Z}^n,
\,\|k\|_\infty\geq2}|k|^{-\widetilde{D}\min\{1,q\}}
\sum_{R\in\mathscr{Q}_i,\,R\subset P+k\ell(P)}|R|(u_R)^q\\
&\leq 2^{(i-j_P)(n-\widetilde{D}\min\{1,q\})}
\sum_{k\in\mathbb{Z}^n,\,\|k\|_\infty\geq2}
|k|^{-\widetilde{D}\min\{1,q\}}|P|
\left\|\vec t\right\|_{\dot f^s_{\infty,\infty}(\mathbb A)}^q \\
&\lesssim 2^{(i-j_P)(n-\widetilde{D}\min\{1,q\})} |P|
\left\|\vec t\right\|_{\dot f^s_{\infty,\infty}(\mathbb A)}^q.
\end{align*}
Substituting the different bounds for $I_1(i,j)$ into \eqref{167}, we obtain
\begin{align*}
\left\|B_0\vec t\right\|_{\dot f^s_{\infty,q}(\mathbb A)}
&\lesssim\left\|\vec t\right\|_{\dot f^s_{\infty,\infty}(\mathbb A)}
\sup_{P\in\mathscr{Q}}
\left\{\sum_{j=j_P}^\infty\sum_{i=-\infty}^{j_P-1}
2^{(i-j)\widetilde{E}\min\{1,q\}}\right\}^{\frac{1}{q}}\\
&\quad+\sup_{P\in\mathscr{Q}}
\left\{\frac{1}{|P|}\sum_{j=j_P}^\infty\sum_{i=j_P}^{j}
2^{(i-j)\widetilde{E}\min\{1,q\}}\sum_{R\in\mathscr{Q}_i,\,R\subset 3P}
|R|(u_R)^q\right\}^{\frac{1}{q}}\\
&\quad+\left\|\vec t\right\|_{\dot f^s_{\infty,\infty}(\mathbb A)}
\sup_{P\in\mathscr{Q}}
\left\{\sum_{j=j_P}^\infty\sum_{i=j_P}^{j}
2^{(i-j)\widetilde{E}\min\{1,q\}}2^{(i-j_P)(n-\widetilde{D}\min\{1,q\})}\right\}^{\frac{1}{q}} \\
&=:\mathrm K_1+\mathrm K_2+\mathrm K_3,
\end{align*}
where a straightforward estimation of geometric series, using \eqref{5322}, gives
\begin{equation*}
\mathrm K_1+\mathrm K_3
\lesssim \left\|\vec t\right\|_{\dot f^s_{\infty,\infty}(\mathbb A)}
\leq \left\|\vec t\right\|_{\dot f^s_{\infty,q}(\mathbb A)}.
\end{equation*}
Finally, by rearranging the order of summation,
\begin{align*}
\mathrm K_2
&=\sup_{P\in\mathscr{Q}}
\left\{\frac{1}{|P|}\sum_{i=j_P}^\infty\sum_{j=i}^\infty
2^{(i-j)\widetilde{E}q}\sum_{R\in\mathscr{Q}_i,\,R\subset 3P}
|R|(u_R)^q\right\}^{\frac{1}{q}}\\
&\lesssim\sup_{P\in\mathscr{Q}}
\left\{\frac{1}{|P|}\sum_{i=j_P}^\infty\sum_{R\in\mathscr{Q}_i,\,R\subset 3P}
|R|(u_R)^q\right\}^{\frac{1}{q}}
\leq\sup_{P\in\mathscr{Q}}
\left\{\frac{1}{|P|}\sum_{R\in\mathscr{Q},\,R\subset 3P}
|R|(u_R)^q\right\}^{\frac{1}{q}}
\sim\left\|\vec t\right\|_{\dot f^s_{\infty,q}(\mathbb A)},
\end{align*}
which completes the proof of the boundedness of $B_0$ on $\dot f^s_{\infty,q}(\mathbb A)$.
Notice that only for term $\mathrm K_2$ did we need to
invoke the norm of $\dot f^s_{\infty,q}(\mathbb A)$,
while all other parts could be estimated even in
the larger space $\dot f^s_{\infty,\infty}(\mathbb A)$.

\textbf{The lower triangular part:}
Now, we show the boundedness of $B_1$ on $\dot f^s_{\infty,q}(\mathbb A)$.

Using the assumed almost diagonal estimates with Lemma \ref{22 precise}, we obtain,
for each $j\in\mathbb{Z}$ and $Q\in\mathscr{Q}_j$,
\begin{align*}
|Q|^{-(\frac{s}{n}+\frac12)}\left|A_Q\left(B_1\vec t\right)_Q\right|
&\leq|Q|^{-(\frac{s}{n}+\frac12)}\sum_{R\in\mathscr{Q},\,\ell(R)<\ell(Q)}
|b_{Q,R}|\left\|A_QA_R^{-1}\right\|\left|A_R\vec t_R\right|\\
&\lesssim|Q|^{-(\frac{s}{n}+\frac12)}\sum_{R\in\mathscr{Q},\,\ell(R)<\ell(Q)}
\left[1+\frac{|x_Q-x_R|}{\ell(Q)}\right]^{-D}
\left[\frac{\ell(R)}{\ell(Q)}\right]^F\\
&\quad\times\left[1+\frac{|x_Q-x_R|}{\ell(Q)}\right]^{\frac{d}{p}}  |R|^{\frac{s}{n}+\frac12}u_R \\
&=\sum_{R\in\mathscr{Q},\,\ell(R)<\ell(Q)}
\left[\frac{\ell(R)}{\ell(Q)}\right]^{\widetilde{F}}
\left[1+\frac{|x_Q-x_R|}{\ell(Q)}\right]^{-\widetilde D}u_R,
\end{align*}
where
\begin{equation}\label{5323}
\widetilde D:=D-\frac{d}{p}>J
\text{ and }
\widetilde F:=F+s+\frac{n}{2}>J=\frac{n}{\min\{1,q\}}\geq n.
\end{equation}
From Lemma \ref{one fits all}, we deduce that
\begin{align*}
|Q|^{-(\frac{s}{n}+\frac12)}\left|A_Q\left(B_1\vec t\right)_Q\right|
&\leq\left\{\sum_{R\in\mathscr{Q},\,\ell(R)<\ell(Q)}
\left[\frac{\ell(R)}{\ell(Q)}\right]^{\widetilde{F}\min\{1,q\}}
\left[1+\frac{|x_Q-x_R|}{\ell(Q)}\right]^{-\widetilde D\min\{1,q\}}u_R^q\right\}^{\frac1q} \\
&\quad\times\left\{\sum_{R\in\mathscr{Q},\,\ell(R)<\ell(Q)}
\left[\frac{\ell(R)}{\ell(Q)}\right]^{\widetilde{F}}
\left[1+\frac{|x_Q-x_R|}{\ell(Q)}\right]^{-\widetilde D}\right\}^{(1-\frac1q)_+} \\
&=:\mathrm{I}^{\frac1q}\times\mathrm{II}^{(1-\frac1q)_+}.
\end{align*}
Using \eqref{5323}, we find that
\begin{align*}
\mathrm{II}&=\sum_{i=j+1}^{\infty}2^{(j-i)\widetilde{F}}
\sum_{k\in\mathbb Z^n}\left(1+2^j|x_Q-2^{-i}k|\right)^{-\widetilde{D}} \\
&\sim\sum_{i=j+1}^{\infty}2^{(j-i)\widetilde{F}}\int_{\mathbb R^n}\left(1+2^j|x_Q-2^{-i}y|\right)^{-\widetilde{D}}\,dy
\sim\sum_{i=j+1}^{\infty}2^{(j-i)(\widetilde{F}-n)}\sim 1.
\end{align*}
For $q=\infty$, we obtain simply
\begin{equation*}
|Q|^{-(\frac{s}{n}+\frac12)}\left|A_Q\left(B_1\vec t\right)_Q\right|
\lesssim\mathrm{II}\times\sup_{R\in\mathscr Q}u_R\sim\left\|\vec t\right\|_{\dot f^s_{\infty,\infty}(\mathbb A)},
\end{equation*}
and hence $\|B_1\vec t\|_{\dot f^s_{\infty,\infty}(\mathbb A)}
\lesssim\|\vec t\|_{\dot f^s_{\infty,\infty}(\mathbb A)}$
by taking the supremum over $Q\in\mathscr Q$.

For any $q\in(0,\infty)$, we proceed with
\begin{align}\label{170}
\left\|B_1\vec t\right\|_{\dot f^s_{\infty,q}(\mathbb A)}^q
&=\sup_{P\in\mathscr{Q}}
\frac{1}{|P|}\sum_{j=j_P}^\infty\sum_{Q\in\mathscr{Q}_j,\,Q\subset P}|Q|
\left[|Q|^{-(\frac{s}{n}+\frac12)}\left|A_Q\left(B_1\vec t\right)_Q\right|\right]^q \\
&\lesssim\sup_{P\in\mathscr{Q}}
\frac{1}{|P|}\sum_{j=j_P}^\infty
\sum_{i=j+1}^{\infty}2^{(j-i)\widetilde{F}\min\{1,q\}}J_1(i,j),\notag
\end{align}
where, for any $i,j\in\mathbb Z$,
$$
J_1(i,j):=\sum_{R\in\mathscr{Q}_i}\sum_{Q\in\mathscr{Q}_j,\,Q\subset P}|Q|
\left(1+2^j|x_Q-x_R|\right)^{-\widetilde{D}\min\{1,q\}}(u_R)^q.
$$
From Lemmas \ref{33y} and \ref{33}, and \eqref{33x}, we infer that,
for any $j\in\mathbb{Z}$ with $j\geq j_P$ and for any $i\in\mathbb{Z}$ with $i>j$,
\begin{align*}
J_1(i,j)
&\sim\sum_{R\in\mathscr{Q}_i,\,R\subset3P}
\int_P\frac{1}{(1+2^j|x-x_R|)^{\widetilde{D}\min\{1,q\}}}\,dx\,(u_R)^q\\
&\quad+\sum_{k\in\mathbb{Z}^n,\,\|k\|_\infty\geq2}
\sum_{R\in\mathscr{Q}_i,\,R\subset P+k\ell(P)}
\sum_{Q\in\mathscr{Q}_j,\,Q\subset P}|Q|
\left(2^{j-j_P}|k|\right)^{-\widetilde{D}\min\{1,q\}}(u_R)^q\\
&\lesssim2^{(i-j)n}\sum_{R\in\mathscr{Q}_i,\,R\subset3P}|R|(u_R)^q\\
&\quad+2^{(i-j)n}2^{(j-j_P)(n-\widetilde{D}\min\{1,q\})}
\sum_{k\in\mathbb{Z}^n,\,\|k\|_\infty\geq2}|k|^{-\widetilde{D}\min\{1,q\}}
\sum_{R\in\mathscr{Q}_i,\,R\subset P+k \ell(P)}|R|(u_R)^q\\
&\sim2^{(i-j)n}\sum_{R\in\mathscr{Q}_i,\,R\subset3P}
|R|(u_R)^q+2^{(i-j)n}2^{(j-j_P)(n-\widetilde{D}\min\{1,q\}}|P|
\left\|\vec t\right\|_{\dot f^s_{\infty,\infty}(\mathbb A)}^q.
\end{align*}
Substituting into \eqref{170}, we conclude that
\begin{align*}
\left\|B_1\vec t\right\|_{\dot f^s_{\infty,q}(\mathbb A)}
&\lesssim\sup_{P\in\mathscr{Q}}
\left[\frac{1}{|P|}\sum_{j=j_P}^\infty\sum_{i=j+1}^\infty
2^{(j-i)(\widetilde{F}\min\{1,q\}-n)}\sum_{R\in\mathscr{Q}_i,\,R\subset3P}
|R|(u_R)^q\right]^{\frac{1}{q}}\\
&\quad+\left\|\vec t\right\|_{\dot f^s_{\infty,\infty}(\mathbb A)}
\sup_{P\in\mathscr{Q}}\left[\sum_{j=j_P}^\infty\sum_{i=j+1}^\infty
2^{(j-i)(\widetilde{F}\min\{1,q\}-n)}2^{(j-j_P)(n-\widetilde{D}\min\{1,q\})}\right]^{\frac{1}{q}}\\
&=:\mathrm L_1+\mathrm L_2.
\end{align*}
A straightforward estimation of geometric series, using \eqref{5323}, gives
$\mathrm L_2\lesssim\|\vec t\|_{\dot f^s_{\infty,\infty}(\mathbb A)}
\leq\|\vec t\|_{\dot f^s_{\infty,q}(\mathbb A)}.$
By rearranging the summation order, we also obtain
\begin{align*}
\mathrm L_1 &=\sup_{P\in\mathscr{Q}}
\left[\frac{1}{|P|}\sum_{i=j_P+1}^\infty\sum_{j=j_P}^{i-1}2^{(j-i)(\widetilde{F}\min\{1,q\}-n)}
\sum_{R\in\mathscr{Q}_i,\,R\subset3P}|R|(u_R)^q\right]^{\frac{1}{q}} \\
&\lesssim\sup_{P\in\mathscr{Q}}
\left[\frac{1}{|P|}\sum_{R\in\mathscr{Q},\,R\subset3P}|R|(u_R)^q\right]^{\frac{1}{q}}
\sim\left\|\vec t\right\|_{\dot f^s_{\infty,q}(\mathbb A)}.
\end{align*}
This finishes the proof of the boundedness of $B_1$
and hence Lemma \ref{110}.
\end{proof}

Now, we are ready to prove Theorem \ref{ad iff cor}.

\begin{proof}[Proof of Theorem \ref{ad iff cor}]
Let $B$ be $(D,E,F)$-almost diagonal with $D,E,F$ as in \eqref{173}.
We need to show the boundedness of $B$ on $\dot a^{s,\tau}_{p,q}(W)$.
The case that $\tau=\frac1p$ and $q<\infty$ is covered by Lemma \ref{110},
so it remains to consider $\tau>\frac1p$ or $(\tau,q)=(\frac1p,\infty)$,
which is precisely the supercritical case. In this case,
denoting by $\mathbb A=\{A_Q\}_{Q\in\mathscr Q}$
a sequence of reducing operators of order $p$ for $W$,
we deduce from \cite[Corollary 4.18]{bhyy1} that
\begin{equation}\label{1732}
\dot a^{s,\tau}_{p,q}(W)
=\dot f^{s+n(\tau-\frac1p)}_{\infty,\infty}(\mathbb A)
=\dot f^{s+n(\tau-\frac1p),\frac 1p}_{p,\infty}(W)
\text{ if }\begin{cases}
\tau>\frac1p\quad\text{ or } \\
(\tau,q)=(\frac1p,\infty),
\end{cases}
\end{equation}
where the space on the right is again in the scope of Lemma \ref{110},
with $q=\infty$ and $\widehat s:=s+n(\tau-\frac1p)$ in place of $s$.
According to Lemma \ref{110}, where $\frac{n}{\min\{1,q\}}=\frac{n}{\min\{1,\infty\}}=n$,
a sufficient condition for the boundedness of a $(D,E,F)$-almost diagonal operator
on this space is that
\begin{equation}\label{1734}
D>n+\frac{d}{p},\
E>\frac{n}{2}+\widehat s+\frac{d}{p},
\text{ and }
F>n-\frac{n}{2}-\widehat s.
\end{equation}
Substituting the value of $\widehat s$ and observing that $n=J_\tau$
in the supercritical case under consideration,
it is immediate that this coincides with \eqref{173},
which completes the proof of the sufficiency part of Theorem \ref{ad iff cor}.

Turning to the necessity, consider the $A_p$-matrix weight $W(x):=|x|^{-d}$
for any $x\in\mathbb R^n\setminus\{\mathbf{0}\}$
of the $A_p$-dimension $d\in[0,n)$, and suppose that
every $(D,E,F)$-almost diagonal operator is bounded on $\dot a^{s,\tau}_{p,q}(W)$.
If $\tau=\frac1p$ and $q>1$, then Lemma \ref{sharp2} implies that
\begin{equation}\label{1733}
D>n+\frac{d}{p},\
E>n+s+\frac{d}{p},
\text{ and }
F>n-\frac{n}{2}-s.
\end{equation}
On the other hand, for such $\tau$ and $q$, we have $n=\frac{n}{\min\{1,q\}}=J_\tau$,
so that \eqref{1733} coincide with \eqref{173}, proving its necessity in this case.

Let us finally consider the necessity of \eqref{173} when $\tau>\frac1p$.
In this case, we again make use of the identity \eqref{1732}.
Thus, we are assuming that every $(D,E,F)$-almost diagonal operator
is bounded on $\dot f^{s+n(\tau-\frac1p),\frac 1p}_{p,\infty}(W)$.
Then Lemma \ref{sharp2} applies again, now with $q=\infty>1$
and $\widehat s=s+n(\tau-\frac1p)$ in place of $s$,
and proves the necessity of \eqref{1734}. As before,
substituting the value of $\widehat s$ and observing that $n=J_\tau$
in the supercritical case under consideration,
it is immediate that this coincides with \eqref{173},
which completes the proof of the necessity part of \eqref{173}
and hence Theorem \ref{ad iff cor}.
\end{proof}

\begin{remark}\label{ad iff rem}
When $d=0$, the assumptions \eqref{173} of Theorem \ref{ad iff cor} are just
the assumptions \eqref{ad YY} of Theorem \ref{ad FJ-YY},
and hence Theorem \ref{ad iff cor} was known in the unweighted case.
But, even in this special case, the sharpness of these conditions seems to be new.
Moreover, Theorem \ref{ad iff cor} is a natural
matrix-weighted extension of the said unweighted result,
with sharp dependence on the $A_p$-dimension of the matrix weight under consideration.

It is also immediate to compare Theorem \ref{ad iff cor} with
our benchmark Theorem \ref{77}, which was obtained from the unweighted
Theorem \ref{ad FJ-YY} via a straightforward application of Lemma \ref{22 precise}.
Recalling that $\Delta=\frac{d}{p}+\frac{\widetilde d}{p'}$,
where $W\in A_p$ has $A_p$-dimensions $(d,\widetilde d,\Delta)$,
we find that the assumptions on both $D$ and $F$ in Theorem \ref{ad iff cor}
are better than the corresponding ones in Theorem \ref{77}
by the non-negative term $\frac{\widetilde d}{p'}$, of course only
in the critical and the supercritical regimes where Theorem \ref{ad iff cor} applies.
On the other hand, in the subcritical regime we have Theorem \ref{ad BF},
and it was already observed in Remark \ref{ad BF vs benchmark} that
Theorem \ref{ad BF} improves Theorem \ref{77} in this case. Thus,
Theorems \ref{ad BF} and \ref{ad iff cor} together improve Theorem \ref{77} in all cases.

Let us then compare \eqref{173} of Theorem \ref{ad iff cor} with the assumptions of Theorem \ref{ad BF}.
Spelling out the abbreviation $\widehat\tau:=(\tau-\frac{1-d/n}{p})_+$,
so that $n\widehat\tau=n(\tau-\frac{1}{p})+\frac{d}{p}=n(\tau-\frac{1}{p})_+ +\frac{d}{p}$
for any $\tau\in[\frac{1}{p},\infty)$, the latter reads as
\begin{equation*}
D>J+n\left(\tau-\frac{1}{p}\right)_+ +\frac{d}{p},\
E>\frac{n}{2}+s+n\left(\tau-\frac{1}{p}\right)_+ +\frac{d}{p},\text{ and }
F>J-\frac{n}{2}-s,
\end{equation*}
where $J$ is the same as in \eqref{J}.
Comparing with \eqref{173}, we find that the assumptions on $E$ are identical,
but \eqref{173} relax the assumptions on both $D$ and $F$
by means of subtracting the non-negative quantity $n(\tau-\frac{1}{p})_+$ and
replacing $J$ with the smaller $J_\tau$, so indeed Theorem \ref{ad iff cor} is always preferable over Theorem \ref{ad BF} in the critical and the supercritical spaces.
\end{remark}

\section{Synthesis}\label{summary}

Combining Theorems \ref{ad BF} and \ref{ad iff cor},
we obtain the following theorem
which is the final result about the boundedness of almost diagonal operators on $\dot a^{s,\tau}_{p,q}(W)$;
we omit the details.

\begin{theorem}\label{ad BF2}
Let $s\in\mathbb R$, $\tau\in[0,\infty)$, $p\in(0,\infty)$, $q\in(0,\infty]$,
and $d\in[0,n)$. Let $J$ and $J_\tau$ be the same as,
respectively, in \eqref{J} and \eqref{tildeJ}, and
\begin{equation}\label{tauJ2}
\widehat\tau:=\left[\left(\tau-\frac{1}{p}\right)+\frac{d}{np}\right]_+,\
\widetilde{J}:=J_{\tau}+\left[\left(n\widehat\tau\right)\wedge\frac{d}{p}\right],\text{ and }
\widetilde{s}:=s+n\widehat\tau.
\end{equation}
If $B$ is $(D,E,F)$-almost diagonal with
\begin{equation}\label{ad new2}
D>\widetilde{J},\
E>\frac{n}{2}+\widetilde{s},\text{ and }
F>\widetilde{J}-\frac{n}{2}-\widetilde{s},
\end{equation}
then $B$ is bounded on $\dot a^{s,\tau}_{p,q}(W)$
whenever $W\in A_p$ has the $A_p$-dimension $d$.
\end{theorem}

\begin{remark}
\begin{enumerate}[\rm(i)]
\item The assumptions of Theorem \ref{ad BF2} are sharp
whenever $\dot a^{s,\tau}_{p,q}(W)$ is one of the following:
\begin{enumerate}
\item any usual Besov space $\dot b^{s}_{p,q}(W)$ (by Lemma \ref{ad Besov sharp});
\item a usual Triebel--Lizorkin space $\dot f^{s}_{p,q}(W)$
with $q\in[\min\{1,p\},\infty]$ (by Lemma \ref{ad Triebel sharp});
\item a critical space $\dot f^{s,\frac1p}_{p,q}(W)$ with $q\in(1,\infty)$ (by Theorem \ref{ad iff cor});
\item any supercritical space, namely with $\tau\in(\frac1p,\infty)$
or $(\tau,q)=(\frac1p,\infty)$ (by Theorem \ref{ad iff cor}).
\end{enumerate}

\item By Remarks \ref{ad BF rem} and \ref{ad iff rem}, we find that
Theorem \ref{ad BF2} improves all cases of the previously known Theorem \ref{ad FJ-YY}, as well as its straightforward matrix-valued extension in Theorem \ref{77}.

\item\label{ad BF2 remark3} Similarly to Remark \ref{ad conv},
we have a slightly stronger estimate than the boundedness of $B$; that is,
\begin{equation*}
\left\|\left\{2^{js}\sum_{Q\in\mathscr{Q}_j}\widetilde{\mathbf{1}}_Q(\cdot)
\sum_{R\in\mathscr{Q}}\left|H_j(\cdot)b_{Q,R}\vec{t}_R\right|
\right\}_{j\in\mathbb Z}\right\|_{L\dot A_{p,q}^{\tau}}
\leq C\left\|\vec{t}\right\|_{\dot a^{s,\tau}_{p,q}(W)},
\end{equation*}
where $ L\dot A_{p,q}^{\tau}\in\{ L\dot B_{p,q}^\tau, L\dot F_{p,q}^\tau\}$
is the same as in \eqref{LApq} and,
for any $j\in\mathbb Z$, $H_j=W^{\frac{1}{p}}$ or
$H_j=A_j$ with $A_j$ the same as in \eqref{Aj}.
(The difference with the mere boundedness of $B$ is that
we can take the norms $|\cdot|$ in $\mathbb C^m$ inside the summation.)
\end{enumerate}
\end{remark}

Next, motivated by Theorem \ref{ad BF2},
we extend Definition \ref{def ad 0} as follows.

\begin{definition}
Let $s\in\mathbb R$, $\tau\in[0,\infty)$, $p\in(0,\infty)$, $q\in(0,\infty]$, and $d\in[0,n)$.
An infinite matrix $B:=\{b_{Q,R}\}_{Q,R\in\mathscr{Q}}\subset\mathbb{C}$
is said to be \emph{$\dot a^{s,\tau}_{p,q}(d)$-almost diagonal}
if it is $(D,E,F)$-almost diagonal with $D,E,F$ the same as in \eqref{ad new2}.
For $d=0$, we use $\dot a^{s,\tau}_{p,q}$-almost diagonal
as a synonym to $\dot a^{s,\tau}_{p,q}(0)$-almost diagonal.
\end{definition}

\begin{remark}
With this notation, Theorem \ref{ad BF2} says that
an $\dot a^{s,\tau}_{p,q}(d)$-almost diagonal operator
is bounded on $\dot a^{s,\tau}_{p,q}(W)$,
whenever $W\in A_p$ has the $A_p$-dimension $d$.

Here, of course, ``$\dot a^{s,\tau}_{p,q}(d)$'',
unlike $\dot a^{s,\tau}_{p,q}(W)$, is not a space,
but just a notation for the collection of the various parameters.
This notation is in the same spirit as the concept of
$\mathbf{ad}^{\alpha,q}_p(\beta_W)$ of \cite[Definition 2.5]{fr21},
where the parameter $\beta_W$ is the doubling exponent of the doubling matrix weight $W$ of order $p$,
which is attached to the parameters of a Triebel--Lizorkin space under consideration.

This terminology is preferable over, say,
``$\dot a^{s,\tau}_{p,q}(W)$-almost diagonal'' for the following reasons.
Firstly, this concept only depends on the dimension $d$, but not otherwise on $W$.
Secondly, there might not be a unique $d$ attached to a given $W\in A_p$.
It is clear that, if $W$ has some $A_p$-dimension $d_0\in[0,n)$,
it also has the $A_p$-dimension $d$ for every $d\in[d_0,n)$,
but it could also happen that the set of admissible $A_p$-dimensions
has the form $(d_0,n)$, so that there is no minimal value of $d$
[see \cite[Proposition 2.43(ii)]{bhyy1}.

It is easy to show that, for $\tau=d=0$,
the concept of $\dot a^{s,0}_{p,q}(0)$-almost diagonal
agrees with the concept of $\dot a^s_{p,q}$-almost diagonal from Definition \ref{def ad 0}.
\end{remark}

We conclude with the following easy but useful observation that
allows us to make use of known results
about almost diagonal operators on $\dot a^s_{p,q}$,
without a need to reprove everything in our generalised setting.
An illustration of the application of this lemma
is given in the subsequent Corollary \ref{AB ad}.

\begin{lemma}\label{ad vs old}
Let $s\in\mathbb R$, $\tau\in[0,\infty)$,
$p\in(0,\infty)$, $q\in(0,\infty]$, and $d\in[0,n)$.
Let both $\widetilde J\in[n,\infty)$ and $\widetilde s\in\mathbb R$
be the same as in \eqref{tauJ2}, and let
\begin{equation}\label{hatr}
\widetilde r:=\frac{n}{\widetilde{J}}\in(0,1].
\end{equation}
Let $\widetilde p\in(0,\infty)$ and $\widetilde q\in(0,\infty]$ be any numbers such that
\begin{equation}\label{hatpq}
\begin{cases}
\widetilde p=\widetilde r&\text{if }\dot a=\dot b,\\
\widetilde p\wedge\widetilde q=\widetilde r&\text{if }\dot a=\dot f.
\end{cases}
\end{equation}
For an infinite matrix $B$, the following are equivalent:
\begin{enumerate}[\rm(i)]
\item\label{ad orig}$B$ is $\dot a^{s,\tau}_{p,q}(d)$-almost diagonal;
\item\label{ad hatpq}$B$ is $\dot a^{\widetilde s}_{\widetilde p,\widetilde q}$-almost diagonal;
\item\label{ad hatr}$B$ is $\dot a^{\widetilde s}_{\widetilde r,\widetilde r}$-almost diagonal.
\end{enumerate}
\end{lemma}

Since $\dot a^{\widetilde s}_{\widetilde r,\widetilde r}
=\dot b^{\widetilde s}_{\widetilde r,\widetilde r}
=\dot f^{\widetilde s}_{\widetilde r,\widetilde r}$,
the last condition (iii) allows the reduction of the almost diagonality
to a space of the preferred type $\dot b$ or $\dot f$,
no matter the type of the original space $\dot a^{s,\tau}_{p,q}$ that we start with.

\begin{proof}[Proof of Lemma \ref{ad vs old}]
We first show that (i) and (ii) are equivalent.
For a $(D,E,F)$-almost diagonal matrix $B$,
the conditions of being $\dot a^{s,\tau}_{p,q}(d)$-almost diagonal, from \eqref{ad new2}, are
$D>\widetilde{J},$ $E>\frac{n}{2}+\widetilde{s},$ and
$F>\widetilde{J}-\frac{n}{2}-\widetilde{s},$
where both $\widetilde{J}$ and $\widetilde{s}$ are the same as in \eqref{tauJ2}.
On the other hand, for any $(\widetilde p,\widetilde q)$ the same as in \eqref{hatpq},
the conditions of being $\dot a^{\widetilde s}_{\widetilde p,\widetilde q}$-almost diagonal
[from Definition \ref{def ad 0}, which refers to \eqref{ad FJ}] are
$D>\frac{n}{\widetilde r},$ $E>\frac{n}{2}+\widetilde{s},$ and
$F>\frac{n}{\widetilde r}-\frac{n}{2}-\widetilde{s},$
and it is clear that the two sets of conditions agree
if $\widetilde r$ is the same as in \eqref{hatr}.
This proves the equivalence of \eqref{ad orig} and \eqref{ad hatpq},
and their equivalences with \eqref{ad hatr}
follow from the fact that $\widetilde p=\widetilde q=\widetilde r$
is a particular case of numbers that satisfy \eqref{hatpq}.
This finishes the proof of Lemma \ref{ad vs old}.
\end{proof}

{\color{red}
The following useful corollary also serves as an illustration of applying Lemma \ref{ad vs old} to deduce results about almost diagonal operators in our more general setting from already available in the existing literature.
}

\begin{corollary}\label{AB ad}
Let $s\in\mathbb R$, $\tau\in[0,\infty)$, $p\in(0,\infty)$, $q\in(0,\infty]$, and $d\in[0,n)$.
Let $A:=\{a_{Q,R}\}_{Q,R\in\mathscr Q}\subset\mathbb C$
and $B:=\{b_{R,P}\}_{R,P\in\mathscr Q}\subset\mathbb C$
be $\dot a^{s,\tau}_{p,q}(d)$-almost diagonal. Then
\begin{equation*}
A\circ B:=\left\{\sum_{R\in\mathscr{Q}}a_{Q,R}b_{R,P}\right\}_{Q,P\in\mathscr Q}
\end{equation*}
is also $\dot a^{s,\tau}_{p,q}(d)$-almost diagonal.
\end{corollary}

\begin{proof}
Let both $\widetilde s$ and $\widetilde r$ be the same as in Lemma \ref{ad vs old}.
By that lemma, we find that both $A$ and $B$ are
$\dot f^{\widetilde s}_{\widetilde r,\widetilde r}$-almost diagonal.
Then, from \cite[Theorem 9.1]{fj90}, it follows that the composition $A\circ B$
is also $\dot f^{\widetilde s}_{\widetilde r,\widetilde r}$-almost diagonal.
By Lemma \ref{ad vs old} again, we conclude that
$A\circ B$ is $\dot a^{s,\tau}_{p,q}(d)$-almost diagonal,
which completes the proof of Corollary \ref{AB ad}.
\end{proof}

{\color{red}
\subsection*{Acknowledgements}

We would like to thank the anonymous referees for several constructive comments that improved the presentation.
}

\bigskip

\noindent Fan Bu

\medskip

\noindent Laboratory of Mathematics and Complex Systems (Ministry of Education of China),
School of Mathematical Sciences, Beijing Normal University, Beijing 100875, The People's Republic of China

\smallskip

\noindent{\it E-mail:} \texttt{fanbu@mail.bnu.edu.cn}

\bigskip

\noindent Tuomas Hyt\"onen

\medskip

\noindent Department of Mathematics and Systems Analysis, Aalto University, P.O. Box 11100 (Otakaari~1), FI-00076 Aalto, Finland

\noindent{\it E-mail:} \texttt{tuomas.p.hytonen@aalto.fi}

\bigskip

\noindent Dachun Yang (Corresponding author) and Wen Yuan

\medskip

\noindent Laboratory of Mathematics and Complex Systems (Ministry of Education of China),
School of Mathematical Sciences, Beijing Normal University, Beijing 100875, The People's Republic of China

\smallskip

\noindent{\it E-mails:} \texttt{dcyang@bnu.edu.cn} (D. Yang)

\noindent\phantom{{\it E-mails:} }\texttt{wenyuan@bnu.edu.cn} (W. Yuan)

\end{document}